\documentclass[11pt]{amsart}
\usepackage[utf8]{inputenc}
\usepackage{Covariant_Wave}

\title[A gauge-covariant wave equation with space-time white noise]{Well-posedness of a gauge-covariant wave equation with space-time white noise forcing}

\author[Bjoern Bringmann]{Bjoern Bringmann}
\address{Bjoern Bringmann, School of Mathematics, Institute for Advanced Study, Princeton, NJ 08540 \& Department of Mathematics, Princeton University, Princeton, NJ 08544}
\email{bjoern@ias.edu}

\author[Igor Rodnianski]{Igor Rodnianski}
\address{Igor Rodnianski, Department of Mathematics, Princeton University, Princeton, NJ 08544}
\email{irod@princeton.edu}

\date{\today}

\begin{document}

\begin{abstract}
We first introduce a new model for a two-dimensional gauge-covariant wave equation with space-time white noise. In our main theorem, we obtain the probabilistic global well-posedness of this model in the Lorenz gauge. 

\noindent
Furthermore, we prove the failure of a probabilistic null-form estimate, which exposes a potential obstruction towards the probabilistic well-posedness of a stochastic Maxwell-Klein-Gordon equation. 
\end{abstract}

\maketitle

\tableofcontents

\section{Introduction}

Many of the most significant nonlinear wave equations are geometric wave equations, such as the wave map equation, the Maxwell-Klein-Gordon equation, the hyperbolic Yang-Mills equation and the Einstein equations. In contrast to scalar nonlinear wave equations, the nonlinearity in geometric wave equations comes from the curvature of the state space itself: the target manifold for the wave problem, the 
connection for the Maxwell-Klein-Gordon and the Yang-Mills equations and the metric for the Einstein equations. Simultaneously, another distinguishing feature of the geometric wave equations  is their gauge-invariance, arising from the invariance of the equations under  a group of local transformations. This, in a nutshell, can be thought of as the freedom to choose different coordinate-systems. These two phenomena: the nature of nonlinearity and gauge invariance  create interesting analytical and geometric challenges. In the last three decades, there has been many significant developments in the \emph{deterministic} theory of geometric wave equations. Since a complete (or even partial) overview of this literature is well beyond the scope of this introduction and since the focus of this paper is on the low regularity (stochastic) solutions, we limit our references to the local low-regularity well-posedness in \cite{C93,KM93,KM94,KM95,KT99,KRS15,LR15,LR17,MS04,O14,Pecher18,Pecher20,ST10,ST16,Tao01b,Tao03}, the small-data critical results in \cite{C93,KST15,KT17,RT04,Tao01a}, and the large-data critical results in \cite{C93,KL15,LO19a,LO19b,OT19,Tao04}. \\

For scalar nonlinear wave (and other dispersive) equations, there has also been significant progress towards a \emph{probabilistic} rather than \emph{deterministic} theory. The randomness is commonly introduced in the evolution problem through either random initial data and/or stochastic forcing. While the regularity of the random initial data or stochastic forcing is often beyond the threshold for deterministic well-posedness, it can still be possible to prove probabilistic well-posedness. The first theorem on the probabilistic well-posedness of a dispersive equation was obtained by \cite{B96}, who studied the two-dimensional cubic nonlinear Schr\"odinger equation with random initial data just below the scaling-critical regularity\footnote{The random initial data in \cite{B96} and many related works is drawn from the corresponding Gibbs measure, but Gibbs measures will not be discussed further here.}. More recently, there has been significant progress for higher-order nonlinearities, dispersive equations in three dimensions, and on problems from wave turbulence, see e.g. \cite{B24,B21,BDNY24,CG19,DH21b,DH23, DNY19,DNY21,DNY20,GKO18b,GKO24,OOT20,OOT21,ST21}. For a more detailed discussion, we refer the reader to the introduction of \cite{BDNY24}.  \\

Despite the significant progress towards a probabilistic theory of scalar nonlinear wave equations, there has only been little progress towards a probabilistic well-posedness theory for geometric wave equations. At this point, the primary references are \cite{BLS24} and \cite{KLS23}. In \cite{BLS24}, the first author, Lührmann, and Staffilani proved the probabilistic local well-posedness of the $(1+1)$-dimensional wave maps equation with Brownian paths as initial data. The most important aspect of \cite{BLS24} is that Brownian paths, i.e., the random initial data, are natural from both geometric and probabilistic perspectives. In \cite{KLS23}, Krieger, Lührmann, and Staffilani proved the small data probabilistic global well-posedness of the energy-critical Maxwell-Klein-Gordon equation (MKG) with initial data drawn from a Wiener randomization. The most difficult aspect of \cite{KLS23} is that the well-posedness of the energy-critical MKG is already extremely involved at the deterministic level (see e.g. \cite{KST15}), and it is therefore hard to combine the deterministic functional framework with probabilistic ideas. The random initial data in \cite{KLS23}, however, is not completely natural from a geometric perspective. In particular, it is not gauge-invariant in law.  

We note that, in comparison to random geometric wave equations, stochastic geometric parabolic equations are much better understood. We refer the reader to the review article \cite{C22} and the recent research articles \cite{BGHZ21,CC23,CC24,CCHS22,CCHS24,CS23,S21}. \\

The main goal of this article is to further explore probabilistic aspects of geometric wave equations. In \eqref{intro:eq-covariant} below, we introduce a model for gauge-covariant wave equations with space-time white noise, i.e., stochastic forcing. We consider here the simplest and most natural model: a stochastic scalar field driven by a stochastic abelian connection (vector potential). The stochastic forcing leads to randomness in both the vector potential and scalar field, coupled through the interaction between the two but {\it only} in the equation for the scalar field.
In our main theorem (Theorem \ref{intro:thm-main}), we then obtain the probabilistic global well-posedness of our model in the Lorenz gauge. 

In the last subsection of this introduction, we compare our new model with a stochastic Maxwell-Klein-Gordon equation. Through the failure of a probabilistic null-form estimate, we also expose a potential obstruction in proving probabilistic well-posedness for the stochastic Maxwell-Klein-Gordon equation. 

In view of the above, the current paper is the first step towards building a theory of stochastic geometric wave equations, with serious challenges already presenting themselves in the attempts to make sense of the {\it fully coupled} stochastic interaction between a scalar field and an abelian connection (stochastic Maxwell-Klein-Gordon). Further challenges await in constructing stochastic solutions of geometric hyperbolic {\it semilinear} and, eventually, {\it quasilinear}, equations.

\subsection{A gauge-covariant wave equation with space-time white noise}\label{section:covariant}

In this subsection, we introduce our model for a (gauge-)covariant wave equation with space-time white noise and state the main theorem of this article. 

Initially, our discussion will be performed in a general spatial dimension $d\geq 2$, but our main result only concerns the two-dimensional setting, i.e., $d=2$. 
We work on the $(1+d)$-dimensional space-time $\R_t \times \T_x^d$. Here, $\T^d := (\R \backslash 2\pi \Z)^2$ denotes the $d$-dimensional torus and thus our space-time is periodic in the spatial coordinates. We let $\eta$ be the Minkowski metric with signature $(-\, + \, \hdots \, +)$ and raise and lower all indices with respect to $\eta$. 

We let $(A_\alpha)_{\alpha=0}^d \colon \R_t \times \T_x^d \rightarrow \R^{1+d}$ be a vector potential. We then define the corresponding curvature tensor $(F_{\alpha\beta})$ and covariant derivatives $(D_\alpha)$ by
\begin{align}
F_{\alpha \beta } &:= \partial_\alpha A_\beta - \partial_\beta A_\alpha, \label{intro:eq-curvature-tensor}  \\ 
D_\alpha &:= \partial_\alpha + i A_\alpha \label{intro:eq-covariant-derivative}. 
\end{align}
We also let $\phi \colon \R_t \times \T_x^d \rightarrow \C$ be a complex-valued scalar field. The vector potential $(A_\alpha)_{\alpha=0}^d$ and scalar field $\phi$ will soon serve as the unknowns in our covariant wave equation. 
Finally, we let $\scrm \colon \R_t \rightarrow \R$ be a time-dependent mass, $\zeta \colon \R_t \times \T_x^d \rightarrow \C$ be a complex-valued stochastic forcing term, and $(J^\alpha)_{\alpha=0}^d\colon \R_t \times \T_x^d \rightarrow \R^{1+d}$ be a vector-valued stochastic forcing term.\\ 

Equipped with the above definitions, we now consider a stochastic covariant wave equation, which is given by 
\begin{equation}\label{intro:eq-covariant} 
\begin{aligned}
\partial_\alpha F^{\alpha \beta} &= J^\beta \qquad (t,x) \in \mathbb{R}_t \times \T_x^d, \\
\big( D_\alpha D^\alpha + \scrm^{\hspace{-0.4ex}2} \big) \phi &= \zeta \qquad \hspace{1.5ex} (t,x) \in \mathbb{R}_t \times \T_x^d. 
\end{aligned}
\end{equation}
We note that the evolution of the curvature tensor $F$ is completely determined by the space-time current $(J^\beta)_{\beta=0}^d$. In particular, since the space-time current $(J^\beta)_{\beta=0}^d$ will be chosen as a stochastic forcing-term, this leads to a random curvature tensor $F$.

We also note that the evolution equation of the scalar field $\phi$ is the
 Euler-Lagrange equation for 
\begin{equation}\label{intro:eq-Lagrangian}
\phi \mapsto \int_\R \dt \int_{\T^d} \dx \, \Big( \tfrac{1}{2} D_\alpha \phi \overline{D^\alpha \phi} - \tfrac{1}{2} \scrm^{\hspace{-0.4ex}2} |\phi|^2 - \Re \big( \zeta \overline{\phi} \big) \Big). 
\end{equation}
We emphasize that the massive term $-\scrm^{\hspace{-0.4ex}2} |\phi|^2$ is non-coercive. This is necessary for implementing our renormalization, see e.g. Definition \ref{vector:def-mass} and Lemma \ref{vector:lem-quadratic} below. \\

As we will see momentarily, \eqref{intro:eq-covariant} cannot (or should not) be studied for general stochastic forcing terms $\zeta$ and $(J^\beta)_{\beta=0}^d$. Instead, we now introduce assumptions on $\zeta$ and $(J^\beta)_{\beta=0}^d$ which formally guarantee both the gauge-invariance in law and the consistency of the evolution equations imposed on the curvature tensor. 

\subsubsection{Gauge-invariance} 

We first derive a condition on the stochastic forcing term $\zeta$ by formally imposing gauge-invariance. 
For any deterministic (or, more generally, progressively measurable) smooth $\varphi \colon \R_t \times \T_x^d \rightarrow \R$, we consider the gauge transformation 
\begin{equation}\label{intro:eq-gauge}
(A_\alpha, \phi) \mapsto (\widetilde{A}_\alpha,\widetilde{\phi}):= ( A_\alpha + \partial_\alpha \varphi, e^{-i\varphi} \phi ). 
\end{equation}
We note that the curvature tensor is gauge-invariant and the covariant derivative is gauge-covariant with respect to the gauge transformation in \eqref{intro:eq-gauge}. More precisely, if $\widetilde{F}$ and $\widetilde{D}_\alpha$ are the curvature tensor and covariant derivative corresponding to $\widetilde{A}$, then it holds that 
\begin{equation}\label{intro:eq-gauge-F-D}
\widetilde{F}_{\alpha\beta} = F_{\alpha \beta} \qquad \text{and} \qquad \widetilde{D}_\alpha \widetilde{\phi} = \widetilde{D_\alpha \phi}. 
\end{equation}

Using \eqref{intro:eq-gauge-F-D}, we obtain for any solution $(A_\alpha,\phi)$ of the stochastic covariant wave equation that the gauge-transformed unknowns $(\widetilde{A}_\alpha,\widetilde{\phi})$ satisfy 
\begin{equation}\label{intro:eq-covariant-gauge-transformed} 
\begin{aligned}
\partial_\alpha \widetilde{F}^{\alpha \beta} &= J^\beta \qquad \hspace{4.25ex} (t,x) \in \mathbb{R}_t \times \T_x^d, \\
\big( \widetilde{D}_\alpha \widetilde{D}^\alpha + \scrm^{\hspace{-0.4ex}2} \big) \widetilde{\phi} &= e^{-i\varphi} \zeta \qquad \hspace{1.5ex} (t,x) \in \mathbb{R}_t \times \T_x^d.
\end{aligned}
\end{equation}

From a physical perspective, it is natural to impose that the solution $(A_\alpha,\phi)$ should be gauge-invariant in law. After comparing \eqref{intro:eq-covariant} and \eqref{intro:eq-covariant-gauge-transformed}, it is therefore natural to impose that 
\begin{equation}\label{intro:eq-gauge-law-zeta}
\operatorname{Law}\big( e^{-i\varphi} \zeta \big) = \operatorname{Law} \big( \zeta\big). 
\end{equation}
In order to satisfy \eqref{intro:eq-gauge-law-zeta}, we let the stochastic forcing term $\zeta$ be a complex-valued space-time white noise. We re-emphasize that the choice of $\zeta$ as complex-valued space-time white noise is not of our own making, but is essentially dictated by gauge-invariance, i.e, dictated by physical considerations. Since space-time white noise has temporal regularity $-1/2-$ and spatial regularity $-d/2-$, which decreases in the spatial dimension, this suggests that the probabilistic well-posedness theory of our stochastic covariant wave equation is easier in lower dimensions.

\subsubsection{Continuity equation and random currents} 

We now derive a condition on the stochastic forcing terms $(J^\alpha)_{\alpha=0}^d$ from the properties of the curvature tensor. 
From the definition of the curvature tensor $F$ in \eqref{intro:eq-curvature-tensor}, it is clear that $F$ is skew-symmetric, i.e., $F_{\alpha \beta}=-F_{\beta \alpha}$ for all $0\leq \alpha, \beta \leq d$. As a result, it follows that
\begin{equation}\label{intro:eq-curvature-skew-symmetry-condition}
    \partial_\alpha \partial_\beta F^{\alpha \beta}=0. 
\end{equation}
In order for the stochastic covariant wave equation \eqref{intro:eq-covariant} to be consistent, it is therefore necessary that
\begin{equation}\label{intro:eq-continuity-equation} 
\partial_\beta J^\beta = 0. 
\end{equation}
Consistency conditions such as \eqref{intro:eq-continuity-equation} are often called continuity equations and appear naturally in many physical applications. In electrodynamics, where $J^0$ represents the charge and $(J^j)_{j=1}^d$ represents the spatial currents, \eqref{intro:eq-continuity-equation} is the conservation law for the charge. 
Due to \eqref{intro:eq-continuity-equation}, the space-time currents $(J^\alpha)_{\alpha=0}^d$ cannot be chosen independently. We now introduce a model for a space-time white noise current, in which the spatial currents $(J^j)_{j=1}^d$ are chosen independently and the time-current $J^0$ is determined by \eqref{intro:eq-continuity-equation}. 

\begin{definition}[Space-time white noise current]\label{intro:def-white-noise-current}
A vector-valued distribution $(J^\alpha)_{\alpha=0}^d \colon \R_t \times \T_x^d \rightarrow \R^{1+d}$ is called a space-time white noise current if the following two conditions are satisfied:
\begin{enumerate}[label=(\roman*)]
    \item $(J^j)_{j=1}^d$ are independent, real-valued space-time white noises. 
    \item $J^0$ satisfies $\partial_t J^0 = - \partial_j J^j$ and $J^0\big|_{t=0}=0$. 
\end{enumerate}
\end{definition}

In the following, we assume that the stochastic forcing terms $\zeta$ and $(J^\alpha)_{\alpha=0}^d$ are probabilistically independent. 
This independence assumption is rather natural and is satisfied in the context of (hyperbolic or parabolic) stochastic quantization, see e.g. \cite{CCHS24,OOT21}.

\subsubsection{Lorenz gauge} 
Due to the gauge-invariance in law of \eqref{intro:eq-covariant}, it is possible to impose a gauge condition on the vector potential. The classical gauge choices for geometric wave equations are the Coulomb gauge $\partial_j A^j =0$, the Lorenz gauge $\partial_\alpha A^\alpha=0$, and the temporal gauge $A^0=0$ (see e.g. \cite{KM94,MS04,ST10,Tao03}). In this article, we work in the Lorenz gauge and our reasons for this are two-fold: First, the Lorenz gauge leads to a complete derivative in \eqref{intro:eq-phi-Lorenz} below, which simplifies the analysis of high$\times$low-interactions (see Section \ref{section:ansatz}). Second, the Lorenz gauge leads to a system involving only hyperbolic and no elliptic equations, which enables a unified treatment of all components of the vector potential. Despite this, however, we believe that our main theorem can also be proven in the Coulomb and temporal gauges.

We now examine the covariant wave equation \eqref{intro:eq-covariant} in the Lorenz gauge. Due to the Lorenz condition $\partial_\alpha A^\alpha =0$, it holds that 
\begin{equation}\label{intro:eq-curvature-Lorenz} 
\partial_\alpha F^{\alpha \beta} 
= \partial_\alpha \partial^\alpha A^\beta - \partial_\alpha \partial^\beta A^\alpha 
= \partial_\alpha \partial^\alpha A^\beta - \partial^\beta \big( \partial_\alpha A^\alpha\big) 
= \partial_\alpha \partial^\alpha A^\beta. 
\end{equation}
Furthermore, it holds that
\begin{equation}\label{intro:eq-phi-Lorenz}
\begin{aligned}
D_\alpha D^\alpha \phi 
&= \partial_\alpha \partial^\alpha \phi + 2i \partial_\alpha \big( A^\alpha \phi) - i \big( \partial_\alpha A^\alpha \big) \phi - A_\alpha A^\alpha \phi \\
&= \partial_\alpha \partial^\alpha \phi + 2i \partial_\alpha \big( A^\alpha \phi)  - A_\alpha A^\alpha \phi . 
\end{aligned}
\end{equation}
We emphasize that the only quadratic term in \eqref{intro:eq-phi-Lorenz} contains a complete derivative, which is the main advantage of working in the Lorenz gauge. 
Thus, the stochastic covariant wave equation in the Lorenz gauge takes the form 
\begin{align}
    \partial_\alpha \partial^\alpha A^\beta &= J^\beta, \label{intro:eq-covariant-Lorenz-q1} \\ 
    \partial_\alpha A^\alpha &=0, \label{intro:eq-covariant-Lorenz-q2} \\ 
\partial_\alpha \partial^\alpha \phi + 2i \partial^\alpha \big( A_\alpha \phi \big) &= (A_\alpha A^\alpha - \scrm^{\hspace{-0.4ex}2}) \phi + \zeta. \label{intro:eq-covariant-Lorenz-q3}
\end{align}
In order to determine the vector potential $(A^\alpha)_{\alpha=0}^d$ and the scalar field $\phi$ completely, we also need to impose initial conditions. For notational simplicity (and in order to reduce the number of terms), we require that 
\begin{equation}\label{intro:eq-initial-A}
A^\alpha [0]:= \big( A^\alpha (0), \, \partial_t A^\alpha(0) \big) = 0. 
\end{equation} 
For the scalar field $\phi$, we allow non-zero initial data, i.e., impose a condition of the form 
\begin{equation}\label{intro:eq-initial-phi}
\phi [0] = \big( \phi(0), \, \partial_t \phi(0) \big)  = (\phi_0,\phi_1), 
\end{equation}
since it does not significantly increase the number of terms in our analysis. \\

After putting everything together, we arrive at the following initial value problem
\begin{align}
\partial_\alpha \partial^\alpha A^\beta &= J^\beta \qquad \hspace{17.7ex}   (t,x) \in \mathbb{R}_t \times \T_x^d, \label{intro:eq-covariant-Lorenz-e1} \\ 
\partial_\alpha A^\alpha &= 0 \qquad \hspace{19.2ex}   (t,x) \in \mathbb{R}_t \times \T_x^d, \label{intro:eq-covariant-Lorenz-e2} \\ 
\partial_\alpha \partial^\alpha \phi + 2i \partial^\alpha \big( A_\alpha \phi \big) &= (A_\alpha A^\alpha - \scrm^{\hspace{-0.4ex}2}) \phi + \zeta \qquad \hspace{1.25ex} (t,x) \in \mathbb{R}_t \times \T_x^d, \label{intro:eq-covariant-Lorenz-e3} \\
A^\alpha[0]=0, \qquad \phi[0]&=(\phi_0,\phi_1).\label{intro:eq-covariant-Lorenz-e4}
\end{align}

\begin{remark}[Random vector potential] 
The vector potential $A^\alpha$ is determined by the wave equation \eqref{intro:eq-covariant-Lorenz-e1}, the gauge condition \eqref{intro:eq-covariant-Lorenz-e2}, and the initial data $A^\alpha[0]$. As part of our stochastic gauge-covariant wave equation \eqref{intro:eq-covariant}, we have therefore introduced a model for a random vector potential. While this article is only focused on wave equations, our model for the random vector potential $A^\alpha$ may also be interesting in other settings. For example, it may be interesting to study the Schr\"{o}dinger equation for a quantum particle in the electromagnetic field corresponding to the random vector potential $A^\alpha$.
\end{remark}

\subsubsection{Renormalization} 

We now describe the final step in the derivation of our covariant wave equation with space-time white noise, which concerns a renormalization and the resulting infinite and time-dependent mass.
In our discussion of the renormalization, we specialize to the spatial dimension $d=2$. This is because the form of the renormalization depends heavily on the analytical properties of the stochastic forcing, which in turn depend heavily on the dimension.

As we will see in Section \ref{section:vector}, the vector potential $(A^\alpha)_{\alpha=0}^2$ only has spatial regularity $0-$ and, as may then be expected, the quadratic expression $A_\alpha A^\alpha$ diverges in the sense of space-time distributions. To counteract this divergence, we need to introduce a frequency-truncation and a divergent, time-dependent mass $\scrm_{\leq N}\colon \R_t \rightarrow \R$. To be precise, we let $P_{\leq N}$ be the Littlewood-Paley projection from \eqref{prelim:eq-littlewood-paley} and let $\scrm_{\leq N}$ be as in Definition \ref{vector:def-mass} below. Then, we consider the frequency-truncated initial value problem 
\begin{align}
\partial_\alpha \partial^\alpha A^\beta_{\leq N} &= P_{\leq N} J^\beta \qquad \hspace{23.4ex}   (t,x) \in \mathbb{R}_t \times \T_x^2, \label{intro:eq-covariant-Lorenz-truncated-e1} \\ 
\partial_\alpha A^\alpha_{\leq N} &= 0 \qquad \hspace{29.6ex}   (t,x) \in \mathbb{R}_t \times \T_x^2, \label{intro:eq-covariant-Lorenz-truncated-e2} \\ 
\partial_\alpha \partial^\alpha \phi_{\leq N} + 2i \partial_\alpha \big( A^\alpha_{\leq N} \phi_{\leq N} \big) &= (A_{\leq N,\alpha}A^\alpha_{\leq N} - \scrm^{\hspace{-0.4ex}2}_{\leq N}) \phi_{\leq N} + \zeta \qquad \hspace{1.25ex} (t,x) \in \mathbb{R}_t \times \T_x^2, \label{intro:eq-covariant-Lorenz-truncated-e3} \\
A^\alpha_{\leq N}[0]=0, \qquad \phi_{\leq N}[0]&=(\phi_0,\phi_1). \label{intro:eq-covariant-Lorenz-truncated-e4}
\end{align}
In \eqref{intro:eq-covariant-Lorenz-truncated-e1}-\eqref{intro:eq-covariant-Lorenz-truncated-e4}, the sub-script $\leq \hspace{-0.5ex}  N$ is used to denote the new unknowns $(A_{\leq N}^\alpha,\phi_{\leq N})$ and the only Littlewood-Paley operator occurs in $P_{\leq N} J^\beta$. 

For a fixed frequency-truncation parameter $N\geq 1$, it is easy to see that \eqref{intro:eq-covariant-Lorenz-truncated-e1}-\eqref{intro:eq-covariant-Lorenz-truncated-e4} has a unique global solution. Indeed, the inhomogeneous linear wave equation $\partial_\alpha \partial^\alpha A^\beta_{\leq N}=P_{\leq N} J^\beta$ has a unique global solution. Furthermore, it satisfies the Lorenz gauge condition $\partial_\alpha A^{\alpha}_{\leq N}=0$ and is smooth in the spatial variables. As a result, energy estimates easily imply the global well-posedness of the covariant wave equation for $\phi_{\leq N}$. Thus, our main question now concerns the convergence of $(A_{\leq N}^\alpha, \phi_{\leq N})$ as $N$ tends to infinity. 

\subsubsection{Main results} 

We previously introduced the $(1+2)$-dimensional covariant wave equation with space-time white noise in the Lorenz gauge \eqref{intro:eq-covariant-Lorenz-e1}-\eqref{intro:eq-covariant-Lorenz-e4} and its frequency-truncated version \eqref{intro:eq-covariant-Lorenz-truncated-e1}-\eqref{intro:eq-covariant-Lorenz-truncated-e4}. In our main result, we obtain the probabilistic global well-posedness of the corresponding initial value problem. In this statement, we use the notation 
\begin{equation}\label{intro:eq-Sobolev}
\mathscr{H}_x^s (\T^2) := H_x^s (\T^2) \times H_x^{s-1}(\T^2),
\end{equation}
where $s\in \R$ and $H_x^s(\T^2)$ denotes the usual $L^2$-based Sobolev space (see \eqref{prelim:eq-Sobolev} below). 

\begin{theorem}[Probabilistic global well-posedness]\label{intro:thm-main} 
Let $d=2$, let $\delta>0$, and let \mbox{$(\phi_0,\phi_1) \in \mathscr{H}_x^{1/4}(\T^2)$.}
Furthermore,  let $\zeta$ be a complex-valued space-time white noise, let  $(J^\alpha)_{\alpha=0}^2$ be a space-time white noise current, and assume that $\zeta$ and $(J^\alpha)_{\alpha=0}^2$ are probabilistically independent. For all $N\geq 1$, let $(A_{\leq N}[t],\phi_{\leq N}[t])$ be the unique global solution of \eqref{intro:eq-covariant-Lorenz-truncated-e1}-\eqref{intro:eq-covariant-Lorenz-truncated-e4}. Then, the limit 
\begin{equation}
(A[t],\phi[t]) := \lim_{N\rightarrow \infty} (A_{\leq N}[t],\phi_{\leq N}[t]) 
\end{equation}
almost surely exists in
\begin{equation}\label{intro:eq-thm-space}
\Cs_t^0 \mathscr{H}_x^{-\delta} \big( [-T,T]\times \T^2 \rightarrow \R^3) \times 
\Cs_t^0 \mathscr{H}_x^{-\delta} \big( [-T,T]\times \T^2 \rightarrow \C)
\end{equation}
for all $T>0$.
\end{theorem}

\begin{remark}
The global well-posedness follows rather easily from our local estimates, which are the main part of this article, and the algebraic structure of \eqref{intro:eq-covariant-Lorenz-truncated-e1}-\eqref{intro:eq-covariant-Lorenz-truncated-e4}. Indeed, since the vector potential $A^\alpha$ solves a stochastic linear wave equation, it obeys global estimates (see Corollary~\ref{vector:cor-regularity}). While the evolution equation for the scalar field \eqref{intro:eq-covariant-Lorenz-truncated-e3} is nonlinear in $(A^\alpha,\phi)$, it is still linear in $\phi$. Since $A^\alpha$ is already known to obey global bounds, our local estimates and a Gronwall-type argument then imply global bounds for $\phi$. \\ 
The only point of caution is that the initial data $(\phi_0,\phi_1)$ in Theorem \ref{intro:thm-main} belongs to $\mathscr{H}_x^{1/4}$, which is not preserved by the evolution of the scalar field.  However, this problem can be circumvented by iterating our estimates over a sequence of growing intervals (rather than iterating over a sequence of initial times), and we refer the reader to the proof of Proposition \ref{proof:prop-psi} for the details.
\end{remark}

We now briefly discuss the main difficulties and ideas in the proof of Theorem \ref{intro:thm-main}. During this discussion, we formally set $N=\infty$. Most of our argument deals with the Duhamel integral of the derivative nonlinearity in \eqref{intro:eq-covariant-Lorenz-e3}, i.e., 
\begin{equation}\label{intro:eq-bilinear}
\Duh \big[ \partial_\alpha (A^\alpha \phi) \big]. 
\end{equation}
In the literature on wave equations, bilinear estimates for terms similar to \eqref{intro:eq-bilinear} have a long history. Bilinear estimates were first utilized by Klainerman and Machedon in \cite{KM93} and then further studied in \cite{KM95D,KM97,KM97D,KS97,KT99,Tao01b}. For a systematic treatment, we refer the reader to \cite{DFS10,DFS12,KS02}. Due to the space-time white noise in \eqref{intro:eq-covariant-Lorenz-e1}-\eqref{intro:eq-covariant-Lorenz-e4}, however, both $A^\alpha$ and $\phi$ only have spatial regularity $0-$. This regularity is well-beyond the regime of the aforementioned bilinear estimates (see Remark \ref{operator:rem-comparison-bilinear}) and, more generally, deterministic methods\footnote{Based on \cite{CP14,Pecher20}, we expect that the deterministic well-posedness theory requires (at the very least) that $(A,\phi) \in \Cs_t^0 \mathscr{H}_x^s \times \Cs_t^0 \mathscr{H}_x^s$ with $s\geq 1/4$. In fact, the threshold $s=1/4$ shows up even in the much simpler cubic wave equation, where it corresponds to the Lorentz-critical regularity.}. For this reason, our argument is based on probabilistic rather than deterministic estimates.\\ 

One of the main ideas behind probabilistic approaches to random dispersive equations is to decompose the solution into a term with a random structure, such as the solution to the inhomogeneous linear wave equation with stochastic forcing, and a smoother remainder (see e.g. \cite{B96,B21,DNY19,GKO24}). In our setting, this method is complicated by the absence of nonlinear smoothing, which stems from low$\times$high-interactions. To motivate the absence of nonlinear smoothing, we briefly consider the Duhamel integral
\begin{equation}\label{intro:eq-absence-smoothing-motivation}
\Duh \big[ \partial_\alpha \big( A^\alpha \parall \chi \big) \big] = - \int_0^t \ds \, \sin\big((t-s)|\nabla|\big) \frac{\partial_\alpha}{|\nabla|} \big( A^\alpha_{\leq N} \parall \phi_{\leq N}\big),
\end{equation}
where $\parall$ is the low$\times$high-paraproduct from Definition \ref{prelim:def-para-products}. In \eqref{intro:eq-absence-smoothing-motivation}, the $\partial_\alpha$-derivatives are compensated by the inverse gradient $|\nabla|^{-1}$. However, as will be explained in Section \ref{section:ansatz}, there is no further gain of derivatives in the low$\times$high-interaction in \eqref{intro:eq-absence-smoothing-motivation}, and hence \eqref{intro:eq-absence-smoothing-motivation} cannot be any smoother than $\phi$. In order to isolate the main term in $\phi$, we therefore define $\chi \colon \R_t \times \T_x^2 \rightarrow \C$ as the solution of 
\begin{equation}\label{intro:eq-chi}
\partial_\alpha \partial^\alpha \chi + 2 i \partial_\alpha \big( A^\alpha \parall \chi \big) = \zeta, \hspace{10ex} \chi[0]=0.
\end{equation}
To better understand the solution of \eqref{intro:eq-chi}, we introduce
\begin{equation}\label{intro:eq-z-L}
z := \Duh \big[ \zeta \big] \qquad \text{and} \qquad \Lin[ll] \chi := 2i \Duh \big[ \partial_\alpha \big( A^\alpha \parall \chi \big) \big]. 
\end{equation}
We emphasize that both the Picard iterate $z$ and operator $\Lin[ll] $ are random. Since $z$ depends only on the complex-valued space-time white noise $\zeta$ and $\Lin[ll]$ depends only on the space-time white noise current $(J^\beta)_{\beta=0}^2$, however, $z$ and $\Lin[ll]$ are probabilistically independent. Equipped with \eqref{intro:eq-z-L}, we can then write 
\begin{equation}\label{intro:eq-chi-integral}
\chi = \big( 1 + \Lin[ll][] \big)^{-1} z. 
\end{equation}
Since the analysis of the random Picard iterate $z$ is rather elementary, the main step in the analysis of $\chi$ lies in the analysis of the random operator $\Lin[ll][]$ and its resolvent 
$( 1 + \Lin[ll][] )^{-1}$. This is done using the random tensor estimates from \cite{DNY20}, which are combined with problem-specific lattice point counting estimates (Lemma \ref{prelim:lem-basic-counting}). 
We mention that the para-controlled approach to stochastic wave equations from \cite{GKO24}, which has also been used in \cite{B24,BDNY24,OOT20,OOT21}, corresponds to the Neumann series approximation
\begin{equation*}
\big( 1 + \Lin[ll][] \big)^{-1} z = \sum_{n=0}^\infty (-1)^n \big( \Lin[ll][] \big)^n z \approx z - \Lin[ll][] z. 
\end{equation*}
Since $ \Lin[ll][]$ is not smoothing, $( \Lin[ll][])^n z $ is not smoother than $z$ or $\Lin[ll][] z$, and therefore the para-controlled approach cannot be applied to our model \eqref{intro:eq-covariant-Lorenz-e1}-\eqref{intro:eq-covariant-Lorenz-e4}.\\ 

While there is no nonlinear smoothing for low$\times$high-interactions in \eqref{intro:eq-bilinear}, both high$\times$high and high$\times$low-interactions exhibit nonlinear smoothing. For the high$\times$high-interactions, the nonlinear smoothing estimate relies heavily on the random structure of $\chi$ (see Proposition \ref{product:prop-main}). In particular, we rely on the probabilistic independence of the vector potential $A^\alpha$ and Picard iterate $z$. For the high$\times$low-interaction, the nonlinear smoothing estimate utilizes the Lorenz gauge condition, which allows us to push derivatives from $A^\alpha$ onto $\phi$ (see e.g. \eqref{intro:eq-phi-Lorenz}). As a consequence of the nonlinear smoothing estimates for high$\times$high and high$\times$low-interactions, it can be shown that the difference $\phi-\chi$ has regularity $1/4-$ and is therefore smoother than $\phi$. 

\begin{remark}
In \cite{B21}, the first author proved the probabilistic well-posedness of a quadratic derivative nonlinear wave equation on $\mathbb{R}^{1+3}$. Similar as for \eqref{intro:eq-covariant-Lorenz-e1}-\eqref{intro:eq-covariant-Lorenz-e4}, the derivative nonlinearity in \cite{B21} prevents nonlinear smoothing. However, there are significant differences between the methods of this article and \cite{B21}: First, the Ansatz in this article is much simpler than in \cite{B21}, which is possible since the random operator $\Lin[ll]$ is explicit. Second, the dispersive estimates in this article (see e.g. Proposition \ref{operator:prop-main}) are much more involved than in \cite{B21}. The reason is that the high$\times$high-product of $A^\alpha$ and $\phi$ is quite delicate, whereas the high$\times$high-products in \cite[(10)]{B21} are harmless. 
\end{remark}

\begin{remark}
In Theorem \ref{intro:thm-main}, the stochastic forcing terms  $\zeta$ and $(J^\alpha)_{\alpha=0}^2$ are chosen as a space-time white noise and space-time white noise current, respectively. While the choice of $\zeta$ is dictated by gauge-invariance, $(J^\alpha)_{\alpha=0}^2$ only has to satisfy the continuity equation \eqref{intro:eq-continuity-equation} and can therefore be chosen more freely. In Definition \ref{intro:def-white-noise-current}, we chose $(J^j)_{j=1}^2$ as a vector-valued space-time white noise, but  $(J^j)_{j=1}^2$  could also have been chosen as the image of a vector-valued space-time white noise under $\langle \nabla_x \rangle^{-\nu}$, where $\nu\in \R$. In the case $\nu>0$, the well-posedness of \eqref{intro:eq-covariant-Lorenz-truncated-e1}-\eqref{intro:eq-covariant-Lorenz-truncated-e4} can be shown using basic product estimates, since the vector potential $(A^\alpha)_{\alpha=0}^2$ has positive regularity. In the case $\nu<0$, the well-posedness of \eqref{intro:eq-covariant-Lorenz-truncated-e1}-\eqref{intro:eq-covariant-Lorenz-truncated-e4} is more difficult than in the setting of Theorem \ref{intro:thm-main}, and it poses an interesting problem for future research.
\end{remark}

\begin{remark}
In our main theorem, we prove the probabilistic well-posedness of \eqref{intro:eq-covariant-Lorenz-e1}-\eqref{intro:eq-covariant-Lorenz-e4} in  \mbox{$(1+2)$-dimensions.} Of course, \eqref{intro:eq-covariant-Lorenz-e1}-\eqref{intro:eq-covariant-Lorenz-e4} can also be studied in $(1+3)$-dimensions, but proving probabilistic well-posedness in this case seems rather challenging. In $(1+3)$-dimensions, the linear stochastic objects in $A^\alpha$ and $\phi$ have regularity $-1/2-$. Based on the lattice point counting estimates in $(1+3)$-dimensions from \cite[Lemma 4.15]{B24}, one further expects that the Duhamel integral of the high$\times$high-interactions in $\partial_\alpha (A^\alpha \phi)$ also has regularity $-1/2-$. Thus, the problem is critical with respect to the probabilistic scaling introduced in \cite{DNY19}.
\end{remark}

\begin{remark}
Using finite speed of propagation, a similar result as in Theorem~\ref{intro:thm-main} can also be obtained if the spatial domain $\T^d_x$ is replaced with the Euclidean space $\R^d_x$.
In this case, the convergence takes place in 
\begin{equation*}
\Cs_t^0 \mathscr{H}_x^{-\delta} \big( [-T,T]\times [-R,R]^2 \rightarrow \R^3) \times 
\Cs_t^0 \mathscr{H}_x^{-\delta} \big( [-T,T]\times [-R,R]^2  \rightarrow \C)
\end{equation*}
for all $T>0$ and $R>0$. We emphasize that it is important that Theorem \ref{intro:thm-main} is global-in-time since, if Theorem \ref{intro:thm-main} were only local-in-time, the time $T$ would have to be chosen as sufficiently small depending on $R$ (see e.g. \cite[Theorem 1.4]{BLS24}).
\end{remark}

\subsection{Maxwell-Klein-Gordon and the failure of a probabilistic null-form estimate}\label{section:intro-null-form} 

We recall that the $(1+d)$-dimensional (periodic) Maxwell-Klein-Gordon system is given by
\begin{alignat}{3}
\partial_\alpha F^{\alpha \beta} 
&= - \Im \big( \phi \overline{D^\beta \phi} \big)    
&\hspace{10ex}& (t,x) \in \R_t \times \T_x^d, \label{intro:eq-MKG-1} \\
D_\alpha D^\alpha \phi 
&= 0 
&& (t,x) \in \R_t \times \T_x^d. \label{intro:eq-MKG-2}
\end{alignat}
The well-posedness of \eqref{intro:eq-MKG-1}-\eqref{intro:eq-MKG-2} has been extensively studied\footnote{Most of the results on the Maxwell-Klein-Gordon equations are formulated on $\mathbb{R}^d$ instead of $\mathbb{T}^d$. Despite the elliptic aspects of the analysis, which enter through the gauge conditions, the local well-posedness theory can mostly be carried over from the Euclidean to the periodic setting.} in dimension $d=2$ in \cite{CP14,Pecher18},  dimension $d=3$ in \cite{KM94,MS04}, the energy-critical dimension $d=4$ in \cite{KST15,KL15,OT16}, and high dimensions in \cite{RT04,Pecher20}.  In order to understand probabilistic aspects of geometric wave equations, one would like to understand a stochastic version of \eqref{intro:eq-MKG-1}-\eqref{intro:eq-MKG-2}, which is given by 
\begin{alignat}{3}
\partial_\alpha F^{\alpha \beta} 
&= - \Im \big( \phi \overline{D^\beta \phi} \big)    + J^\beta 
&\hspace{10ex}& (t,x) \in \R_t \times \T_x^d, \label{intro:eq-sMKG-1} \\
D_\alpha D^\alpha \phi 
&= \zeta 
&& (t,x) \in \R_t \times \T_x^d. \label{intro:eq-sMKG-2}
\end{alignat}
Similar as in Subsection \ref{section:covariant}, gauge-invariance encourages us to choose $\zeta$ as a complex-valued space-time white noise. Due to the low regularity of space-time white noise, we again restrict ourselves to the $(1+2)$-dimensional setting, i.e., $d=2$. 
As before, the choice of the space-time current $(J^\beta)_{\beta=0}^2$ offers more flexibility. Since Theorem \ref{intro:thm-failure} below is not concerned with the space-time currents $(J^\beta)_{\beta=0}^2$, we make the simplest choice and simply set $J^\beta =0$. \\

The only difference between the stochastic Maxwell-Klein-Gordon equations \eqref{intro:eq-sMKG-1}-\eqref{intro:eq-sMKG-2} and covariant wave equation \eqref{intro:eq-covariant} is the term $\Im ( \phi \overline{D^\beta \phi})$. However, this difference turns out to be substantial. 
As we will show in Theorem \ref{intro:thm-failure} below,
 $\Im ( \phi \overline{D^\beta \phi})$ diverges as a space-time distribution at the level of the first Picard iterate. 
 This divergence is different from divergences which can be removed using renormalizations, since it concerns the probabilistically non-resonant component (see Remark \ref{intro:rem-failure} below).  In Theorem \ref{intro:thm-failure}, we actually prove a stronger claim, which not only concerns $\Im ( \phi \overline{D^\beta \phi})$ but also a certain null-form. In order to state the stronger claim, we first introduce additional notation. \\

For all $1\leq j,k\leq 2$ and $\phi,\psi\colon \R_t \times \T_x^2 \rightarrow \C$, we define the null-form
\begin{equation}\label{intro:eq-Qjk}
Q_{jk}\big( \phi, \psi \big) = \partial_j \phi \, \partial_k \psi - \partial_k \phi \,  \partial_j \psi. 
\end{equation}
We note that $Q_{11}(\phi,\psi)=Q_{22}(\phi,\psi)=0$, which implies that only $Q_{12}(\phi,\psi)$ is non-trivial. We also introduce the Leray-Projection $\Leray$, which is defined by 
\begin{equation}\label{intro:eq-Leray}
\Leray B :=  B - \Delta^{-1} \nabla \big( \nabla \cdot B \big) 
\end{equation}
for all smooth $B\colon \R_t \times \T_x^2 \rightarrow \R^2$. From a direct calculation, it follows that 
\begin{equation}\label{intro:eq-null-identity}
\Leray \Im \big( \phi \overline{\nabla \phi} \big)_j 
= \frac{1}{(2\pi)^2} \int_{\T^2} \dx  \, \Im \big( \phi \overline{\nabla \phi} \big)_j - \Delta^{-1} \partial^k \Im \big( Q_{jk}(\phi, \overline{\phi}) \big). 
\end{equation}
We note that the constant term in \eqref{intro:eq-null-identity} only appears because we are working on the torus $\T_x^2$. The identity \eqref{intro:eq-null-identity}, which rewrites the Leray-projection of $\Im \big( \phi \overline{\nabla \phi} \big)$ in terms of the null-forms, is often used when working in the Coulomb gauge (see e.g. \cite{KM94,MS04}). This is because the Coulomb condition allows the introduction of the Leray-projection into the evolution equations for the vector potential. Finally, we let $z\colon \R_t \times \T_x^2 \rightarrow \C$ be the first Picard iterate of \eqref{intro:eq-sMKG-2}, i.e., the solution of
\begin{equation}\label{intro:eq-z}
\partial_\alpha \partial^\alpha z = \zeta, \qquad z[0]=0. 
\end{equation}
Equipped with the null-forms $Q_{jk}$ and the inhomogeneous linear wave $z$, we can now state the failure of a probabilistic null-form estimate. 

\begin{theorem}[Failure of a probabilistic null-form estimate]\label{intro:thm-failure}
Let $d=2$, let $1\leq j \leq 2$, let $0<c\ll 1$ be sufficiently small, let $A=0$, and let $z$ be as in \eqref{intro:eq-z}. Then, there exist two deterministic functions $\varphi,\psi\in \Cs^\infty_c((0,\infty) \times \T_x^2 \rightarrow [-1,1])$  such that, for all $N\in \dyadic$, 
\begin{equation}\label{intro:eq-failure-1} 
 \operatorname{Var} \bigg( 
\int_\R \dt \int_{\T^2} \dx \, \varphi(t,x)  \Im \big( P_{\leq N} z \, \overline{D^j P_{\leq N} z} \big) \bigg) \geq c \log(N) - c^{-1}
\end{equation}
and 
\begin{equation}\label{intro:eq-failure-2} 
 \operatorname{Var} \bigg( 
\int_\R \dt \int_{\T^2} \dx \, \psi(t,x) \Im \big( Q_{12}\big( P_{\leq N} z, \overline{P_{\leq N}z} \big) \big) \bigg) \geq c \log(N) - c^{-1}.
\end{equation}
\end{theorem}

\begin{remark}\label{intro:rem-failure}
As mentioned above, the divergence in \eqref{intro:eq-failure-1} and \eqref{intro:eq-failure-2} is different from divergences which are often removed using renormalizations (such as the divergence of $A_{\leq N,\alpha} A^{\alpha}_{\leq N}$ discussed in Subsection \ref{section:covariant}).
This is because \eqref{intro:eq-failure-1} and \eqref{intro:eq-failure-2} involve the variance and not the expectation of the space-time integrals. 
\end{remark}

\begin{remark}
From the proof of Theorem \ref{intro:thm-failure}, one can see that the logarithmic divergence in \eqref{intro:eq-failure-2} is not present if, in \eqref{intro:eq-z},  $\zeta$ is replaced with $\langle \nabla_x \rangle^{-\nu} \zeta$, where $\nu>0$. However, if the same replacement is made in \eqref{intro:eq-sMKG-2}, then the stochastic Maxwell-Klein-Gordon equations  \eqref{intro:eq-sMKG-1}-\eqref{intro:eq-sMKG-2} are not gauge-invariant.
\end{remark}

We note that Theorem \ref{intro:thm-failure} does not strictly imply the ill-posedness of the stochastic Maxwell-Klein-Gordon equation \eqref{intro:eq-sMKG-1}-\eqref{intro:eq-sMKG-2}. However, it does expose a potential obstruction towards proving probabilistic well-posedness and prevents us from using Picard iteration schemes (without additional ingredients). \\

\textbf{Acknowledgements:} The authors thank Ilya Chevyrev, Sung-Jin Oh, and Daniel Tataru for interesting discussions during the preparation of this work. 
The authors also thank the anonymous referees for helpful comments that improved the quality of the manuscript. B.B. was partially supported by the NSF under Grant No. DMS-1926686. I.R. is partially supported by a Simons 
Investigator Award.

\section{Ansatz}\label{section:ansatz}

In this section, we describe our Ansatz for solving \eqref{intro:eq-covariant-Lorenz-truncated-e1}-\eqref{intro:eq-covariant-Lorenz-truncated-e4}, i.e., the covariant wave equation with space-time white noise in the Lorenz gauge. We first rigorously define all terms in our Ansatz and state the main proposition of this section (Proposition \ref{ansatz:prop-ansatz}). At the end of this section, we then describe the heuristic motivation behind our Ansatz. 

We first define the Duhamel integral and the wave propagator by  
\begin{align}
  \Duh \big[ G \big] 
  &:= - \int_0^t \ds \, \frac{\sin\big( (t-s)|\nabla|\big)}{|\nabla|} G(s), \label{ansatz:eq-duhamel} \\   
\Wp(t) (\phi_0,\phi_1) &:= \cos\big( t |\nabla|\big) \phi_0 + \frac{\sin\big(t |\nabla|\big)}{|\nabla|} \phi_1, \label{ansatz:eq-wp}
\end{align}
for all $t \in \R$, smooth $G\colon \R_t \times \T_x^2 \rightarrow \C$, and smooth $\phi_0,\phi_1 \colon \T_x^2 \rightarrow \C$. If $(J^\alpha)_{\alpha=0}^2$ is the space-time white noise current from Definition \ref{intro:def-white-noise-current} and $\zeta$ is a complex-valued space-time white noise, we define the vector potential $A$, its frequency-truncation $A_{\leq N}$, and the complex-valued field $z$ by 
\begin{align}
A^\alpha &:= \Duh \big[  J^\alpha \big] \qquad \text{for all } 0 \leq \alpha \leq 2, \label{ansatz:eq-A} \\
A^\alpha_{\leq N}&:= \Duh \big[ P_{\leq N} J^\alpha \big] \qquad \text{for all } 0 \leq \alpha \leq 2, \label{ansatz:eq-A-truncated} \\ 
z &:= \Duh \big[ \zeta \big] \label{ansatz:eq-z}. 
\end{align}
We note that, since the evolution equations for $A^\alpha$ and $A^\alpha_{\leq N}$ are linear, constant-coefficient wave equations, it holds that $A_{\leq N}^\alpha= P_{\leq N} A^\alpha$. In the same spirit, we now simplify our notation by writing
\begin{equation}
A^\alpha_K := P_K A^\alpha. 
\end{equation}
As we will see in Section \ref{section:vector} below, both the vector potential $A^\alpha_{\leq N}$ and the complex-valued field $z$ have regularity $0-$. Using the notation from \eqref{ansatz:eq-A}-\eqref{ansatz:eq-z}, we can rewrite \eqref{intro:eq-covariant-Lorenz-truncated-e3} as 
\begin{equation}\label{ansatz:eq-phiN-a} 
\phi_{\leq N} + 2 i \Duh \Big[ \partial_\alpha \big( A^\alpha_{\leq N} \phi_{\leq N} \big) \Big] = 
z + 
\Duh \Big[ \Big( A_{\leq N,\alpha} A_{\leq N}^\alpha - \scrm^{\hspace{-0.4ex}2}_{\leq N}\Big) \phi_{\leq N} \Big] 
+ \Wp(t) (\phi_0,\phi_1). 
\end{equation}
As will be discussed at the end of this section, \eqref{ansatz:eq-phiN-a} cannot be treated as a perturbation of the constant-coefficient linear wave equation. To be slightly more precise, the low$\times$high-interactions in $\partial_\alpha (A^\alpha_{\leq N} \phi_{\leq N})$ do not exhibit nonlinear smoothing and the high$\times$high-interactions in $\partial_\alpha (A^\alpha_{\leq N} \phi_{\leq N})$ cannot be defined without additional structural information on $\phi_{\leq N}$. In order to overcome these difficulties, we let $\parall$, $\parasim$, and $\paragg$ be the low$\times$high, high$\times$high, and high$\times$low-paraproducts from Definition \ref{prelim:def-para-products} below. We then define the three random operators 
\begin{align}
\Lin[ll][\leq N] \phi 
&:= 2i \Duh \Big[ \partial_\alpha \Big( A_{\leq N}^\alpha \parall \phi \Big) \Big] \label{ansatz:eq-Lin-ll}, \\ 
\Lin[sim][\leq N] \phi 
&:= 2i \Duh \Big[ \partial_\alpha \Big( A_{\leq N}^\alpha \parasim \phi \Big) \Big] \label{ansatz:eq-Lin-sim}, \\ 
\Lin[gg][\leq N] \phi
&:= 2i \Duh \Big[ \partial_\alpha \Big( A_{\leq N}^\alpha \paragg \phi \Big) \Big] \label{ansatz:eq-Lin-gg}.
\end{align}
We note that, due to the definitions in \eqref{ansatz:eq-Lin-ll}, \eqref{ansatz:eq-Lin-sim}, and \eqref{ansatz:eq-Lin-gg}, it holds that 
\begin{equation}\label{ansatz:eq-Lin-sum}
\Big( \Lin[ll][\leq N] + \Lin[sim][\leq N] + \Lin[gg][\leq N] \Big) \phi = 2i \Duh \Big[ \partial_\alpha \Big( A^\alpha_{\leq N} \phi \Big) \Big].  
\end{equation}
Equipped with our definitions, we can now state our Ansatz in the form of a proposition. 

\begin{proposition}[Ansatz]\label{ansatz:prop-ansatz}
Let $N\in \dyadic$, let $A_{\leq N}$ be as in \eqref{ansatz:eq-A-truncated}, let $\chi_{\leq N}, \psi_{\leq N}\colon \R_t \times \T_x^2 \rightarrow \C$, and let 
$\phi_{\leq N} = \chi_{\leq N} + \psi_{\leq N}$. Furthermore, assume that $(\chi_{\leq N},\psi_{\leq N})$ is a solution of 
\begin{align}
\chi_{\leq N} + \Lin[ll][\leq N] \chi_{\leq N} &= z \label{ansatz:eq-chi}
\end{align}
and 
\begin{align}
\psi_{\leq N} + \Lin[ll][\leq N] \psi_{\leq N}  
=& \, - \big( \Lin[sim][\leq N] + \Lin[gg][\leq N]\big) \big( \chi_{\leq N} + \psi_{\leq N} \big) \label{ansatz:eq-psi-1} \\
+& \,  \Duh \Big[ \Big( A_{\leq N,\alpha} A_{\leq N}^\alpha - \scrm^{\hspace{-0.4ex}2}_{\leq N}\Big) \big(\chi_{\leq N}+\psi_{\leq N}\big) \Big] 
+ \Wp(t) (\phi_0,\phi_1).\label{ansatz:eq-psi-2}
\end{align}
Then, $(A_{\leq N},\phi_{\leq N})$ solves the covariant wave equation with space-time white noise in the Lorenz gauge, i.e., \eqref{intro:eq-covariant-Lorenz-truncated-e1}-\eqref{intro:eq-covariant-Lorenz-truncated-e4}. 
\end{proposition}

\begin{proof}
We have to verify that \eqref{intro:eq-covariant-Lorenz-truncated-e1}-\eqref{intro:eq-covariant-Lorenz-truncated-e4} are satisfied. We first note that the initial conditions in \eqref{intro:eq-covariant-Lorenz-truncated-e4} can be verified directly from \eqref{ansatz:eq-A-truncated}, \eqref{ansatz:eq-chi}, and \eqref{ansatz:eq-psi-1}-\eqref{ansatz:eq-psi-2}. 

The wave equation \eqref{intro:eq-covariant-Lorenz-truncated-e1} is equivalent to its Duhamel integral formulation \eqref{ansatz:eq-A-truncated}. The Lorenz condition \eqref{intro:eq-covariant-Lorenz-truncated-e2} follows directly from \eqref{ansatz:eq-A-truncated}, the continuity equation for the current \eqref{intro:eq-continuity-equation}, and the initial condition $A_{\leq N}[0]=0$. Finally, the wave equation \eqref{intro:eq-covariant-Lorenz-truncated-e3} follows directly from \eqref{ansatz:eq-Lin-sum} and the evolution equations for $\chi_{\leq N}$ and $\psi_{\leq N}$, i.e., \eqref{ansatz:eq-chi} and \eqref{ansatz:eq-psi-1}-\eqref{ansatz:eq-psi-2}. 
\end{proof}

At the end of this section, we motivate our Ansatz heuristically. To this end, we first recall that $A_{\leq N}$ and $z$ (and therefore also $\phi_{\leq N}$) have regularity at most $0-$. In the following, it is useful to employ a heuristic which will be partially justified by Lemma \ref{prelim:lem-basic-counting} below:  
If $u,v\colon \R_t \times \T_x^2 \rightarrow \C$ are random waves and $K,L,M \in \dyadic$, then multi-linear dispersive effects gain a factor of $\min(K,L,M)^{-1/4}$ over trivial estimates of
\begin{equation*}
\Big\| P_M \int_0^t \ds \sin\big( (t-s) |\nabla|\big) P_K u(s) \, P_L v(s) \Big\|_{\mathcal{N}}  
\end{equation*}
for all relevant norms $\mathcal{N}$. Equipped with this heuristic, we now separately discuss low$\times$high, high$\times$high, and high$\times$low-interactions. \\ 

\emph{Low$\times$high-interactions:} 
    The corresponding contribution is given by 
    \begin{equation}\label{intro:eq-motivation-low-high}
    \Lin[ll][\leq N] \phi_{\leq N} = -2i  \int_0^t \ds \, \sin\big((t-s)|\nabla|\big) \frac{\partial_\alpha}{|\nabla|} \big( A^\alpha_{\leq N} \parall \phi_{\leq N}\big).  
    \end{equation}    
    The $\partial_\alpha$-derivatives are compensated by the inverse gradient $|\nabla|^{-1}$. Since multi-linear dispersive effects do not gain derivatives of the high-frequency input $\phi_{\leq N}$, we expect that the regularity of \eqref{intro:eq-motivation-low-high} is no better than the regularity of $\phi$ and hence given by $0-$. In particular, there is no probabilistic nonlinear smoothing. In order to address this difficulty, we view $\Lin[ll][\leq N]$ as a random operator and estimate the resolvent $(1+\Lin[ll][\leq N])^{-1}$. \\
    
\emph{High$\times$high-interactions:}
    The corresponding contribution is given by 
    \begin{equation}\label{intro:eq-motivation-high-high}
    \Lin[sim][\leq N] \phi_{\leq N} = -2i \int_0^t \ds \, \sin\big((t-s)|\nabla|\big) \frac{\partial_\alpha}{|\nabla|}  \big( A^\alpha_{\leq N} \parasim \phi_{\leq N} \big).  
    \end{equation}   
    As above,  the $\partial_\alpha$-derivatives are compensated by the inverse gradient $|\nabla|^{-1}$. Since high$\times$high$\rightarrow$high-interactions gain one quarter of a derivative from multi-linear dispersive effects, we expect that \eqref{intro:eq-motivation-high-high} has regularity $1/4-$. However, since high$\times$high$\rightarrow$low-interactions only gain in the low frequency-scale, and $A^\alpha_{\leq N}$ and $\phi_{\leq N}$ have negative regularity, \eqref{intro:eq-motivation-high-high} cannot be defined without additional information on $A^\alpha_{\leq N}$ and $\phi_{\leq N}$. This additional information consists of the probabilistic independence of $A_{\leq N}$ and $z$. \\ 
    
\emph{High$\times$low-interactions:}
    The corresponding contribution is given by 
    \begin{equation}\label{intro:eq-motivation-high-low}
    \Lin[gg][\leq N] \phi_{\leq N} = -2i  \int_0^t \ds \, \sin\big((t-s)|\nabla|\big) \frac{\partial_\alpha}{|\nabla|}  \big( A^\alpha_{\leq N} \paragg \phi_{\leq N} \big).  
    \end{equation}    
    As above,  the $\partial_\alpha$-derivatives are compensated by the inverse gradient $|\nabla|^{-1}$. From our discussion of low$\times$high-interactions above, one may expect that \eqref{intro:eq-motivation-high-low} only has regularity $0-$. However, due to the Lorenz gauge condition $\partial_\alpha A^{\alpha}_{\leq N}=0$, it holds that
    \begin{equation*}
    \partial_\alpha \big( A^\alpha_{\leq N} \paragg \phi_{\leq N} \big) 
    = A^\alpha_{\leq N} \paragg \partial_\alpha \phi_{\leq N}.
    \end{equation*}
    This allows us to push the derivatives onto the low-frequency input $\phi_{\leq N}$. By also utilizing multi-linear dispersive effects, we expect that it is possible to prove that  \eqref{intro:eq-motivation-high-low} has regularity $1/4-$.

\section{Notation and preliminaries}\label{section:preliminaries} 

In this section, we recall definitions, estimates, and notation from the previous literature. 

\subsection{General notation}\label{section:notation}  
In all sections except for the introduction, we restrict to the spatial dimension $d=2$. 
Throughout this article, Greek indices such as $\alpha$ and $\beta$ take values in $\{0,1,2\}$ and Roman indices such as $a$ and $b$ take values in $\{1,2\}$. For example, we write
\begin{equation*}
\partial_\alpha A^\alpha = 
\partial_0 A^0 + \partial_{1} A^1 + \partial_{2} A^2
\quad \text{and} \quad 
\partial_a A^a = \partial_{1} A^1 + \partial_{2} A^2. 
\end{equation*}
Let $\delta>0$ be as in the statement of Theorem \ref{intro:thm-main}. We introduce fixed parameters $\delta_0,\delta_1,\delta_2,\kappa>0$, 
$b_-\in (0,1/2)$, and $b_0,b_+\in (1/2,1)$ satisfying
\begin{equation}\label{prelim:eq-parameter-condition}
\kappa \ll 1/2-b_- \ll b_0 -1/2 \ll b_+ - 1/2 \ll \delta_0 \ll \delta_1 \ll \delta_2^2 \ll  \delta_2 \ll \delta. 
\end{equation}
We also introduce an implicit parameter $\theta=\theta(b_-,b_0,b_+,\delta_0,\delta_1,\delta_2,\kappa)>0$, whose value is allowed to change from line to line. \\ 

For any $A,B>0$, we write 
\begin{equation}\label{prelim:eq-lesssim}
A\lesssim B \qquad \text{if} \qquad A \leq CB,
\end{equation}
where $C$ is a constant depending only on the parameters in \eqref{prelim:eq-parameter-condition}. Similarly, we write
\begin{align}
A \gtrsim B \qquad &\text{if} \qquad B\lesssim A, \label{prelim:eq-sim}\\
A \sim  B \qquad &\text{if} \qquad A\lesssim B \quad \text{and} \quad B \lesssim A. \label{prelim:eq-gtrsim}
\end{align}
With a slight abuse of notation, we deviate from \eqref{prelim:eq-lesssim}, \eqref{prelim:eq-sim}, and \eqref{prelim:eq-gtrsim} when the quantities are frequency-scales $M,N\in \dyadic$. In that case, we write
\begin{align*}
M \lesssim N \qquad &\text{if} \quad M \leq 2^{10} N, \\ 
M \gtrsim N \qquad &\text{if} \quad M \geq 2^{-10} N, \\ 
M \sim N \qquad &\text{if} \quad 2^{-10} N \leq M \leq 2^{10} N. 
\end{align*}
We also write 
\begin{align*}
M \ll N \qquad &\text{if} \quad M < 2^{-10} N, \\ 
M \gg N \qquad &\text{if} \quad M > 2^{10} N. 
\end{align*}
The precise meaning of $\lesssim$, $\sim$, and $\gtrsim$ will always be clear from context and should not cause any confusion. 

\subsection{Probability theory}\label{section:probability}
We let $(\Omega,\mathcal{E},\mathbb{P})$ be an abstract probability space. Throughout this article, all random variables will be defined on $\Omega$, all events will be contained in $\mathcal{E}$, and all probabilities will be measured with respect to $\mathbb{P}$. \\

We first recall a few basic facts regarding maxima, moments, and tails of random variables. For a more detailed treatment (and the proofs), we refer to \cite{V18}. 

\begin{definition}\label{prelim:def-Psi}
Let $\gamma>0$ and let $X$ be a random variable. Then, we define
\begin{equation}
\big\| X \big\|_{\Psi_\gamma} := \sup_{p\geq 1} \frac{\E \big[ |X|^p \big]^{1/p}}{p^{1/\gamma}}. 
\end{equation}
\end{definition}

In the next lemma, we show that the finiteness of the $\Psi_\gamma$-norm is equivalent to a stretched-exponential tail. 

\begin{lemma}[On moments and tails]
Let $X$ be a random variable and let $\gamma>0$. Furthermore, let $C_\gamma\geq 1$ and $0<c_\gamma\leq 1$ be sufficiently large and sufficiently small, respectively. Then, the following implications hold: 
\begin{enumerate}[label=(\roman*)]
    \item (From moments to tails) Let $B>0$ and assume that $\| X\|_{\Psi_\gamma}\leq B$. Then, we have for all $\lambda \geq 0$ that 
    \begin{equation*}
    \mathbb{P} \big( |X| \geq \lambda \big) \leq 2 \exp \Big( - c_\gamma \Big(\frac{\lambda}{B}\Big)^\gamma \Big). 
    \end{equation*}
    \item (From tails to moments) Let $B>0$ and assume that 
    \begin{equation*}
    \mathbb{P} \big( |X| \geq \lambda \big) \leq 2  \exp \Big( -  \Big(\frac{\lambda}{B}\Big)^\gamma \Big)
    \end{equation*}
    for all $\lambda \geq 0$. Then, it holds that $\| X \|_{\Psi_\gamma}\leq C_\gamma B$. 
\end{enumerate}
\end{lemma}

In the next lemma, we estimate the $\Psi_\gamma$-norm of maxima of random variables. 

\begin{lemma}[Maxima of random variables]\label{prelim:lem-maxima-random}
Let $J\in \mathbb{N}$ and let $(X_j)_{j=1}^J$ be a family of random variables. Then, it holds that 
\begin{equation*}
\big\| \max_{j=1,\hdots, J} |X_j| \big\|_{\Psi_{\gamma}} \lesssim_\gamma \log ( 2+J )^{\frac{1}{\gamma}} \max_{j=1,\hdots, J} \big\| X_j \big\|_{\Psi_\gamma}. 
\end{equation*}
\end{lemma}

Finally, we recall a Gaussian hypercontractivity estimate, which is phrased in the language of Definition \ref{prelim:def-Psi}. 

\begin{lemma}[Gaussian hypercontractivity]\label{prelim:lem-hypercontractivity}
Let $J\in \mathbb{N}$, let $(g_j)_{j=1}^J$ be a family of Gaussians, and let $X$ be a polynomial in $(g_j)_{j=1}^J$ of degree less than or equal to $m$ . Then, it holds that 
\begin{equation}\label{prelim:eq-hypercontractivity}
\big\| X \big\|_{\Psi_{1/m}} \lesssim_m \E \big[ |X|^2 \big]^{1/2}. 
\end{equation}
\end{lemma}
We emphasize that the implicit constant in \eqref{prelim:eq-hypercontractivity} only depends on the degree and not on $J$, i.e., the number of Gaussians. \\ 

We now introduce the notation associated with the space-time white noise current $(J^\alpha)_{\alpha=0}^2$ and the complex-valued space-time white noise $\zeta$. 
In order to define the space-time white noise current $(J^\alpha)_{\alpha=0}^2$, we let $(W_t(n))_{n\in \Z^2}$ be a sequence of Gaussian processes satisfying the following properties: 
\begin{enumerate}[label=(\roman*)]
    \item $W_t(0)$ is a standard, real-valued, two-sided Brownian motion and, for all $n\in \Z^2\backslash \{0\}$, $W_t(n)$ is a standard, complex-valued, two-sided Brownian motion. 
    \item For all $m,n\in \Z^2 \backslash \{0\}$ satisfying $m\neq \pm n$, $W_t(m)$ and $W_t(n)$ are independent processes.
    \item For all $n\in \Z^2$, $\overline{W_t(n)}=W_t(-n)$. 
\end{enumerate}
Then, we let $(W_t^j(n))_{n\in \Z^2}$, where $j=1,2$, be two independent copies of $(W_t(n))_{n\in \Z^2}$ and define 
\begin{equation}\label{prelim:eq-Jj}
J^j(t,x) = \sum_{n\in \Z^2} e^{i\langle n,x\rangle} \partial_t W^j_t(n).
\end{equation}
We note that $J^1$ and $J^2$ are well-defined as space-time distributions. The time-current $J^0$ is then defined via $J^1$ and $J^2$ as in Definition \ref{intro:def-white-noise-current}. \\

In order to define the complex-valued space-time white noise, we let $(Z_t(n))_{n\in \Z^2}$ be a sequence of standard, complex-valued, two-sided Brownian motions. Similar as in \eqref{prelim:eq-Jj}, we then define
\begin{equation}\label{prelim:eq-zeta}
\zeta(t,x) = \sum_{n\in \Z^2} e^{i\langle n ,x \rangle} \partial_t Z_t(n). 
\end{equation}

Furthermore, we define $\sigma_A$ and $\sigma_z$ as the $\sigma$-algebras generated by the stochastic processes $A$ and $z$ from \eqref{ansatz:eq-A} and \eqref{ansatz:eq-z}, respectively. 

\subsection{Harmonic analysis}
For any smooth $f\colon \T_x^2 \rightarrow \C$, we define its Fourier transform by 
\begin{equation}
\widehat{f}(n) := \frac{1}{2\pi} \int_{\T^2} \dx \, f(x) e^{-i\langle n ,x\rangle}. 
\end{equation}
For a smooth, compactly supported function $f\colon \R_t \times \T_x^2\rightarrow \C$, we define its space-time Fourier transform by 
\begin{equation}
\widetilde{f}(\lambda,n) := \frac{1}{(2\pi)^{3/2}}
\int_\R \dt \int_{\T^2} \dx \, f(t,x) e^{-i t\lambda - i \langle n ,x \rangle}
\end{equation}
for all $\lambda \in \R$ and $n \in \Z^2$.
We let $\rho \colon \R \rightarrow [0,1]$ be a smooth, even function satisfying 
$\rho(\xi)=1$ for all $\xi \in [-7/8,7/8]$ and $\rho(\xi)=0$ for all $\xi \not \in [-9/8,9/8]$. For all $N \in \dyadic$, we define $\rho_{\leq N}(\xi):= \rho(\xi/N)$. Furthermore, we define
\begin{equation*}
\rho_{1}:= \rho_{\leq 1} \quad \text{and} \quad \rho_{N} := \rho_{\leq N} - \rho_{\leq N/2} \quad \text{for all } N \geq 2. 
\end{equation*}
We then define the associated Littlewood-Paley operators by 
\begin{equation}\label{prelim:eq-littlewood-paley}
\widehat{P_{\leq N} f}(n) = \rho_{\leq N}(n) \widehat{f}(n) \quad \text{and} \quad 
\widehat{P_{N} f}(n) = \rho_{N}(n) \widehat{f}(n)
\end{equation}
for all smooth $f\colon \T^2\rightarrow \C$ and all $n\in \Z^2$. We also define the fattened Littlewood-Paley operators by 
\begin{equation}\label{prelim:eq-fattened-littlewood-paley}
\widetilde{P}_N := \sum_{\substack{M \in \dyadic\colon \\ M \sim N }} P_M.     
\end{equation}
Equipped with our Littlewood-Paley operators, we can now introduce our paraproducts. 

\begin{definition}[Paraproducts]\label{prelim:def-para-products}
For any smooth $f,g\colon \T^2 \rightarrow \C$, we define the low$\times$high, high$\times$high, and high$\times$low-paraproduct operators by 
\begin{align}
f \parall g &:= 
\sum_{\substack{K,L \in \dyadic \colon \\ K\ll L}} P_K f \, P_L g, \\
f \parasim g &:= 
\sum_{\substack{K,L \in \dyadic \colon \\ K\sim L}} P_K f \, P_L g, \\
f \paragg g &:= 
\sum_{\substack{K,L \in \dyadic \colon \\ K\gg L}} P_K f \, P_L g.
\end{align}
\end{definition}

For any $\nu \in \R$, we define the Sobolev space $H_x^\nu(\T^2)$ and the Hölder space $\Cs_x^\nu(\T^2)$ as the completion of $C^\infty(\T^2)$ with respect to the norms
\begin{align}
\| f \|_{H_x^\nu(\T^2)} &:= \Big( \sum_{K} K^{2\nu} \big\| P_K f \big\|_{L^2_x(\T^2)}^2 \Big)^{1/2}, \label{prelim:eq-Sobolev} \\
\| f \|_{\Cs_x^\nu(\T^2)} &:= \sup_K K^\nu \big\| P_K f \big\|_{L^\infty(\T^2)}. 
\end{align}

We recall from \eqref{intro:eq-Sobolev} that 
\begin{equation*}
\mathscr{H}_x^\nu(\T^2) := H_x^\nu(\T^2)  \times H_x^{\nu-1}(\T^2). 
\end{equation*}
For any interval $I\subseteq \R_t$ and $f\colon I \times \T_x^2 \rightarrow \C$, we further define
\begin{equation*}
\| f[t]\|_{\Cs_t^0 \mathscr{H}_x^\nu(I \times \T^2)} 
:= \| f \|_{\Cs_t^0 H_x^\nu(I\times \T^2)}
+  \| \partial_t f \|_{\Cs_t^0 H_x^{\nu-1}(I\times \T^2)}. 
\end{equation*}

\subsection{\protect{$X^{s,b}$-spaces and tensors}}\label{section:xsb-tensor} 

In this subsection, we discuss $X^{\nu,b}$-spaces, tensors, and related basic estimates. For more detailed treatments, we refer to \cite{Tao06} and \cite{DNY20}. 

\begin{definition}[$X^{\nu,b}$-spaces]\label{prelim:def-Xnub}
For all $\nu \in \R$, $b\in \R$, and $u\colon \R_t \times \T_x^2 \rightarrow \C$, we define
\begin{equation}
\big\| u \big\|_{X^{\nu,b}(\R)} := \big\| \langle n \rangle^\nu \langle |\lambda| - |n| \rangle^b \big(\mathcal{F}_{t,x} u\big)(\lambda,n) \big\|_{L_\lambda^2 \ell_n^2 (\R \times \Z^2)}. 
\end{equation}
Furthermore, for any closed interval $I\subseteq \R$ and $v\colon I \times \T_x^2 \rightarrow \C$, we define
\begin{equation}
\big\| v \big\|_{X^{\nu,b}(I)} := \inf \big\{ \big\| u \big\|_{X^{\nu,b}(\R)} \colon u \big|_{I\times \T^2} = v \big\}. 
\end{equation}
\end{definition}

We remark that the variable $\nu$ is used for the regularity since the more common choice $s$ will be used as a second time variable. 
In the following lemma, we list basic estimates involving $X^{\nu,b}$-norms.

\begin{lemma}[Basic properties of $X^{\nu,b}$]\label{prelim:lem-Xnub}
Let $b,b^\prime \in (1/2,1)$ satisfy $b<b^\prime$ and let $\nu \in \R$. Then, we have the following estimates:
\begin{enumerate}[label=(\roman*)]
\item (Linear estimate) For all $(\phi_0,\phi_1)\in \mathscr{H}_x^\nu(\T^2)$ and $T_0>0$, it holds that 
\begin{equation*}
\big\| \mathcal{W}(t) (\phi_0,\phi_1) \big\|_{\Cs_t^0 \mathscr{H}_x^\nu([0,T_0])} + \big\| \mathcal{W}(t) (\phi_0,\phi_1) \big\|_{X^{\nu,b}([0,T_0])} \lesssim (1+T_0^2) \big\| (\phi_0,\phi_1) \big\|_{\mathscr{H}_x^{\nu}(\T^2)}. 
\end{equation*}
\item (Continuity) For all $\phi \in X^{\nu,b}([0,T_0])$, it holds that 
\begin{equation*}
\| \phi \|_{C_t^0 H_x^\nu([0,T_0]\times \T^2)} \lesssim \| \phi \|_{X^{\nu,b}([0,T_0])}. 
\end{equation*}
\item (Hölder-continuity) For all $0<\gamma<b-1/2$, it holds that 
\begin{equation*}
\| \phi \|_{C_t^\gamma H_x^\nu([0,T_0]\times \T^2)} \lesssim \| \phi \|_{X^{\nu,b}([0,T_0])}. 
\end{equation*}
\item (Time-localization) For all $T_0>0$, $\tau \in (0,1)$, and $\phi \in X^{\nu,b^\prime}([0,T_0+\tau])$, it holds that
\begin{equation*}
\big\| \phi \big\|_{X^{\nu,b}([0,T_0+\tau])} 
\lesssim \big\| \phi \big\|_{X^{\nu,b}([0,T_0])} 
+ \tau^{b^\prime-b} \big\| \phi \big\|_{X^{\nu,b^\prime}([0,T_0+\tau])}.
\end{equation*}
\item (Energy estimate) For all $\nu \in \R$, $T_0 >0$, and $F\in X^{\nu-1,b-1}([0,T_0])$, it holds that 
\begin{equation*}
\big\| \Duh \big[ F \big] \big\|_{\Cs_t^0 \mathscr{H}_x^{\nu}([0,T_0]\times \T^2)} + 
\big\| \Duh \big[ F \big] \big\|_{X^{\nu,b}([0,T_0])}
\lesssim (1+T_0^2) \big\| F \big\|_{X^{\nu-1,b-1}([0,T_0])}. 
\end{equation*}
\end{enumerate}
\end{lemma}

In the next lemma, we state a technical estimate concerning operator norms.  This estimate is essentially the opposite direction of a typical transference principle for $X^{\nu,b}$-norms. 

\begin{lemma}[\protect{From $X^{\nu,b}$ to $\ell^2$}]\label{prelim:lem-Xnub-to-ell2}
Let $\nu_1,\nu_2,b_1,b_2\in \R$, let $T\geq 1$, and let 
\begin{equation*}
\Lc \colon X^{\nu_1,b_1}([0,T]) \rightarrow X^{\nu_2,b_2}([0,T]). 
\end{equation*}
Furthermore, let $\varphi_m \colon \R_t \rightarrow \C$, where $m\in \Z^2$, be a sequence of functions and define the functions $\Lc_m^{(\varphi)} \in X^{\nu_2,b_2}([0,T])$ and operator $\Lc^{(\varphi)}\colon \ell^2 \rightarrow X^{\nu_2,b_2}([0,T])$ by 
\begin{equation*}
\Lc^{(\varphi)}_m 
:= \Lc \big( e^{i \langle m ,x \rangle} \varphi_m \big) \qquad \text{and} \qquad 
\Lc^{(\varphi)} v := 
\sum_{m\in \Z^2} v_m \Lc^{(\varphi)}_m \quad \textup{for all } v\in \ell^2. 
\end{equation*}
Then, it holds that 
\begin{equation}
\big\| \Lc^{(\varphi)} \big\|_{\ell^2 \rightarrow X^{\nu_2,b_2}([0,T])} \lesssim \big\| \Lc \big\|_{X^{\nu_1,b_1}([0,T]) \rightarrow X^{\nu_2,b_2}([0,T])} \sup_{m\in \Z^2} \big\| e^{i\langle m,x \rangle} \varphi_m(t) \big\|_{X^{\nu_1,b_1}([0,T])}. 
\end{equation}
\end{lemma}

\begin{proof}
This follows directly from the orthogonality of 
$(e^{i\langle m,x\rangle} \varphi_m(t))_{m\in \Z^2}$ in $X^{\nu_1,b_1}$. 
\end{proof}

Finally, we state an $X^{\nu,b}$-estimate for  It\^{o}-integrals, which will be used to combine Lemma \ref{prelim:lem-Xnub-to-ell2} with stochastic estimates. 

\begin{lemma}[$X^{\nu,b}$-estimate for It\^{o}-integrals]\label{prelim:lem-Xnub-ito}
Let $M \in \dyadic$,  let $p\geq 1$, let $T\geq 1$, and let $Z$ be as in Subsection \ref{section:probability}. Then, it holds that 
\begin{equation}
\E \Big[ \sup_{\substack{m\in \Z^2 \colon \\ |m|\sim M}} \Big\| e^{i\langle m,x \rangle} \int_0^t \mathrm{d}Z_s(m) \, \sin\big( (t-s) |m| \big) \Big\|_{X^{0,b_+}([0,T])}^p \Big]^{1/p} \lesssim \sqrt{p} T^\theta M^{4(b_+-1/2)}. 
\end{equation}
\end{lemma}

\begin{proof}
We first decompose 
\begin{equation*}
 \int_0^t \mathrm{d}Z_s(m) \, \sin\big( (t-s) |m| \big) 
 = \frac{1}{2i} \Big( 
 e^{it|m|} \int_0^t \mathrm{d}Z_s(m) \, e^{-is|m|} 
 -  e^{-it|m|} \int_0^t \mathrm{d}Z_s(m) \, e^{is|m|} \Big).
\end{equation*}
For all $m\in \Z^2$, it holds that
\begin{equation*}
\operatorname{Law}\Big( t \mapsto 
\int_0^t \mathrm{d}Z_s(m) \, e^{\pm is|m|} \Big) 
= \operatorname{Law} \Big( t \mapsto Z_t(m) \Big).
\end{equation*}
Using the usual $\Cs^{1/2-}-$estimate for Brownian motions and our estimate of maxima of random variables (Lemma \ref{prelim:lem-maxima-random}), it follows for all $\epsilon>0$ that
\begin{equation*}
\E \Big[ \sup_{\substack{m\in \Z^2 \colon \\ |m|\sim M}} \Big\| e^{i\langle m,x \rangle} \int_0^t \mathrm{d}Z_s(m) \, \sin\big( (t-s) |m| \big) \Big\|_{X^{0,b_-}([0,T])}^p \Big]^{1/p} \lesssim \sqrt{p} T^\theta M^\epsilon.
\end{equation*}
In order to increase the $b$-parameter from $b_-$ to $b_+$, we note that 
\begin{align*}
\int_0^t \mathrm{d}Z_s(m) \, \sin\big( (t-s) |m| \big)  
&= \sin\big( (t-s)|m| \big) Z_s(m) \Big|_{s=0}^t - |m| \int_0^t \ds \cos\big( (t-s) |m| \big) Z_s(m)  \\
&= - |m| \int_0^t \ds \cos\big( (t-s) |m| \big) Z_s(m) . 
\end{align*}
Thus, we can increase the $b$-parameter from $b_-$ to $b_+$ by paying a factor of $M^{(b_+-b_-)}$, which is acceptable. 
\end{proof}

This completes our discussion of $X^{\nu,b}$-spaces and we now turn to tensors and their estimates. 

\begin{definition}[Tensors and tensor norms]\label{prelim:def-tensors}
Let $\Jc$ be a finite index set. A tensor $h=h_{k_\Jc}$ is a function from $(\Z^2)^\Jc$ into $\C$. A partition of $\Jc$ is a pair of sets $(\Xc,\Yc)$ such that $\Jc = \Xc \medcup \Yc$ and $\Xc \medcap \Yc = \emptyset$. For any partition $(\Xc,\Yc)$, we define
\begin{equation}\label{prelim:eq-tensor-operator-norm}
\big\| h \big\|_{k_\Xc \rightarrow k_\Yc}^2 := \sup \Big\{ \sum_{k_\Yc} \Big| \sum_{k_\Xc} h_{k_\Jc} z_{k_\Xc}\Big|^2 \colon \sum_{k_\Xc} \big| z_{k_\Xc} \big|^2 = 1 \Big\}. 
\end{equation}
We also define the Hilbert-Schmidt norm 
\begin{equation}
\big\| h \big\|_{k_\Jc}^2 := \sum_{k_\Jc} \big| h_{k_\Jc} \big|^2,
\end{equation}
which corresponds to the special cases $\Xc = \emptyset$ or $\Yc = \emptyset$ in \eqref{prelim:eq-tensor-operator-norm}. 
\end{definition}

We now  recall a special case of a random tensor estimate which was obtained in \cite{DNY20}. For the version stated below, which involves It\^{o}-integrals instead of Gaussians, we also refer to \cite[Lemma C.3]{OWZ21}.  
\begin{lemma}[Moment method]\label{prelim:lem-moment-method} 
Let $J\in \mathbb{N}$ be a positive integer and let $\Jc := \{1,\hdots, J\}$. Let $h=h_{k_0 k_{\Jc}}$ be a tensor, let $K\in \dyadic$, and assume that $h$ is supported on frequency vectors satisfying $|k_j| \leq K$ for all $0\leq j \leq J$. Furthermore, for each $s\in \R$, let 
$f(s):= f_{k_0 k_{\Jc}}(s)$
be an additional tensor. Finally, let $H=H_{k_{\Jc}}$ be a random tensor defined as 
\begin{equation}
H_{k_{\Jc}} = \sum_{k_0 \in \Z^2} h_{k_0 k_{\Jc}}\int_{\R} \mathrm{d}W_s(k_0)\,  f_{k_0 k_{\Jc}}(s),
\end{equation}
where $(W_s(k))_{k\in \Z^2}$ is as in Subsection \ref{section:probability}. Then, it holds for all partitions $(\Xc,\Yc)$ of $\Jc$, all $\epsilon>0$, and all $p\geq 1$ that 
\begin{equation}
\E \Big[ \big\| H \big\|_{k_\Xc \rightarrow k_\Yc}^p \Big]^{1/p} \lesssim_{J,\epsilon} \sqrt{p} K^\epsilon \max \Big( \big\| \, h \big\|_{k_0 k_\Xc \rightarrow k_\Yc}, 
\big\| \, h \big\|_{ k_\Xc \rightarrow k_0 k_\Yc} \Big)
\big\| f_{k_0 k_{\Jc}} \big\|_{\ell^\infty_{k_0}\ell^\infty_{k_\Jc} L^2_s}. 
\end{equation}
\end{lemma}

\subsection{Lattice point counting estimates} 

In this subsection, we present the two-dimensional version of the basic lattice point counting estimate from \cite[Lemma 4.15]{B24}. Combined with the tensor estimates from Subsection \ref{section:xsb-tensor}, the lattice point counting estimates will be used in Section \ref{section:operator} to estimate random linear operators. 

\begin{lemma}[Basic lattice point counting estimate]\label{prelim:lem-basic-counting} 
Let $K,L \in \dyadic$ and let $l \in \Z^2$ satisfy $|l|\sim L$. Then, it holds that
\begin{align}
\sup_{\mu \in \R} \# \Big\{ k \in \Z^2 \colon |k| \sim K, \, \big| |k+l| - |k| -\mu \big| \leq 1 \Big\} & \lesssim \min(K,L)^{-1/2} K^2, \label{prelim:eq-basic-counting-minus} \\ 
\sup_{\mu \in \R} \# \Big\{ k \in \Z^2 \colon |k| \sim K, \, \big| |k+l| + |k| -\mu \big| \leq 1 \Big\} & \lesssim K^{-1/2} K^2, \label{prelim:eq-basic-counting-plus} \\
\sup_{\mu \in \R} \# \Big\{ k \in \Z^2 \colon |k| \sim K, \, \big| |k+l|  -\mu \big| \leq 1 \Big\} & \lesssim K^{-1} K^2. \label{prelim:eq-basic-counting-zero}
\end{align}
Furthermore, for all $\sigma \in \{ -1,0,1 \}$, it holds that 
\begin{equation}\label{prelim:eq-basic-counting-linear}
\sup_{\mu \in \R} \sup_{\unormal \in \mathbb{S}^1} \# \Big\{ 
k \in \Z^2 \colon |k|\sim K, \, \big| \unormal \cdot k + \sigma |k| - \mu \big| \leq 1 \Big\} \lesssim K^{-1/2} K^2. 
\end{equation}
\end{lemma}

On the right-hand sides of \eqref{prelim:eq-basic-counting-minus}, \eqref{prelim:eq-basic-counting-plus}, \eqref{prelim:eq-basic-counting-zero} and \eqref{prelim:eq-basic-counting-linear}, we always isolate the factor $K^2$, which corresponds to the cardinality of $\{ k \in \Z^2 \colon |k|\sim K\}$. 

\begin{remark}
While the precise form of Lemma \ref{prelim:lem-basic-counting} is most closely related to \cite[Lemma 4.15]{B24}, similar estimates have been used for wave equations since the earliest bilinear and null-form estimates (see e.g. \cite{FK20,KM93,KM95D}). In the language of \cite{FK20}, \eqref{prelim:eq-basic-counting-minus} and \eqref{prelim:eq-basic-counting-plus} essentially correspond to weighted surface integrals of hyperboloids and ellipsoids, which were estimated in \cite[Proposition 4.3 and 4.5]{FK20}.
\end{remark}

\begin{proof}
Since the estimates are (essentially) available in the literature, we only sketch the argument. 

Similar as in \cite[Lemma 4.15]{B24} (and \cite[Lemma 5.1]{BDNY24} for \eqref{prelim:eq-basic-counting-plus}), the first two estimates \eqref{prelim:eq-basic-counting-minus} and \eqref{prelim:eq-basic-counting-plus} can be reduced to proving that 
\begin{equation}\label{prelim:eq-basic-counting-p1}
\sup_{\mu_1,\mu_2\in \R} \Leb \Big( \Big\{ \xi \in \R^2 \colon |\xi| \sim K, \, |\xi| = \mu_1 + \mathcal{O}(1), \, \Big| \xi + |l| e_1 \Big| = \mu_2 + \mathcal{O}(1) \Big\} \Big) \lesssim \min(K,L)^{-1/2} K. 
\end{equation}
In order to estimate \eqref{prelim:eq-basic-counting-p1}, we use an angular decomposition and therefore let $\varphi:= \angle ( \xi, e_1 ) \in [0,\pi]$. Then, we dyadically decompose 
\begin{align*}
 &\Leb \Big( \Big\{ \xi \in \R^2 \colon |\xi| \sim K, \, |\xi| = \mu_1 + \mathcal{O}(1), \, \Big| \xi + |l| e_1 \Big| = \mu_2 + \mathcal{O}(1) \Big\} \Big) \\
 \lesssim& \, \sum_{\Phi \in 2^{-\Nzero}} 
 \Leb \Big( \Big\{ \xi \in \R^2 \colon |\xi| \sim K, \, |\xi| = \mu_1 + \mathcal{O}(1), \, \Big| \xi + |l| e_1 \Big| = \mu_2 + \mathcal{O}(1), \, \min \big( \varphi, \pi - \varphi\big) \sim \Phi \Big\} \Big). 
\end{align*}
Trivially, it holds that 
\begin{equation}\label{prelim:eq-basic-counting-p2}
\begin{aligned}
&\Leb \Big( \Big\{ \xi \in \R^2 \colon |\xi| \sim K, \, |\xi| = \mu_1 + \mathcal{O}(1), \, \Big| \xi + |l| e_1 \Big| = \mu_2 + \mathcal{O}(1), \, \min \big( \varphi, \pi - \varphi\big) \sim \Phi \Big\} \Big) \\ 
\lesssim & \, \Leb \Big( \Big\{ \xi \in \R^2 \colon |\xi| \sim K, \, |\xi| = \mu_1 + \mathcal{O}(1), \,  \min \big( \varphi, \pi - \varphi\big) \sim \Phi \Big\} \Big) 
\lesssim  \, \Phi K. 
\end{aligned}
\end{equation}
By computing the volume using polar coordinates as in \cite[Proof of Lemma 4.15]{B24}, it also holds that 
\begin{equation}\label{prelim:eq-basic-counting-p3}
\begin{aligned}
&\Leb \Big( \Big\{ \xi \in \R^2 \colon |\xi| \sim K, \, |\xi| = \mu_1 + \mathcal{O}(1), \, \Big| \xi + |l| e_1 \Big| = \mu_2 + \mathcal{O}(1), \min \big( \varphi, \pi - \varphi\big) \sim \Phi \Big\} \Big) \\ 
\lesssim & \, \Phi^{-1} \min(K,L)^{-1} K. 
\end{aligned}
\end{equation}
The desired estimate \eqref{prelim:eq-basic-counting-p1} then follows by combining \eqref{prelim:eq-basic-counting-p2} and \eqref{prelim:eq-basic-counting-p3} and summing over $\Phi$. We note that the additional factor of $\Phi^{-1}$, which is not present in \cite{B24}, is due to the differences in the volume elements of polar coordinates in $\mathbb{R}^2$ and $\mathbb{R}^3$. \\

The third estimate \eqref{prelim:eq-basic-counting-zero} can be reduced to proving that
\begin{equation*}
\sup_{\mu \in \R} \Leb \Big( \Big\{ \xi \in \R^2 \colon |\xi| \sim K,  \Big| |\xi + l | - \mu \Big| \leq 1 \Big\} \Big) \lesssim K, 
\end{equation*}
which is trivial. \\

It remains to prove the fourth estimate \eqref{prelim:eq-basic-counting-linear}. Similar as in \eqref{prelim:eq-basic-counting-p1}, it suffices to prove that 
\begin{equation}\label{prelim:eq-basic-counting-p4}
\sup_{\mu \in \R} \Leb \Big( \Big\{ \xi \in \R^2 \colon |\xi| \sim K,  \Big| \xi_1 + \sigma |\xi| - \mu \Big| \leq 1 \Big\} \Big) \lesssim K^{-1/2} K^2. 
\end{equation}
The case $\sigma=0$ is trivial and can even be bounded by $K^{-1} K^2$. Thus, it remains to treat the cases $\sigma=\pm 1$. Now, if the constraint $|\xi_1 + \sigma |\xi| - \mu| \leq 1$ is satisfied and $|\xi| \sim K$, then it follows that 
\begin{equation*}
\xi_2^2 = - \xi_1^2 + (\mu - \xi_1)^2 + \mathcal{O}(K). 
\end{equation*}
Thus, \eqref{prelim:eq-basic-counting-p4} can be obtained by first integrating over $\xi_2$, which contributes $K^{1/2}$, and then integrating over $\xi_1$, which contributes $K$. 
\end{proof}

\section{The vector potential}\label{section:vector}

In this section, we primarily examine the vector potential $(A^\alpha)_{\alpha=0}^2$, which is a solution of the system of wave equations
\begin{align}
\partial_\alpha \partial^\alpha A^\beta &= J^\beta, \\ 
A^\beta[0] &=0. 
\end{align}
We recall that the spatial components $J^1,J^2\colon \R_t\times \T_x^2 \rightarrow \R$ of the space-time white noise current  are given by 
\begin{equation}
J^j(t,x) = \sum_{k\in \Z^2} e^{i \langle k,x \rangle} \partial_t W_t^j(k),
\end{equation}
where $(W_t^1(k))_{k\in \Z^2}$ and $(W_t^2(k))_{k\in \Z^2}$ are independent sequences of standard complex-valued Brownian motions as in Subsection \ref{section:probability}. In the first lemma of this section, we obtain an explicit formula for the vector potential. 

\begin{lemma}[Explicit formula for the vector potential]\label{vector:lem-explicit} 
For all $t\geq 0$ and $k \in \Z^2$, it holds that 
\begin{equation}\label{vector:eq-explicit-A0}
\widehat{A}^0(t,k) = - 1\big\{ k\neq 0 \big\} \frac{ik_a}{|k|^2} \int_0^t \dW^a_{s}[k] \Big( \cos\big( (t-s)|k|\big)-1\Big). 
\end{equation}
Similarly, for $a=1,2$ and all $k\in \Z^2\backslash\{0\}$, it holds that 
\begin{equation}
\widehat{A}^a(t,k) = - |k|^{-1} \int_0^t \dW^a_{s}[k] \sin\big((t-s) |k|\big). \label{vector:eq-explicit-Aa}
\end{equation}
In \eqref{vector:eq-explicit-Aa}, we implicitly set $\sin((t-s)|0|)/|0|:=(t-s)$. 
\end{lemma}

\begin{proof}
The second explicit formula \eqref{vector:eq-explicit-Aa} follows directly from the definition of $A^a$ and $J^a$. To be precise, it holds that  
\begin{align*}
A^a(t,x) = - \int_0^t \ds \, \frac{\sin\big( (t-s) |\nabla|\big)}{|\nabla|} J^a(s,x) 
= - \sum_{k \in \Z^2} 
\bigg( e^{i \langle k,x\rangle} \int_0^t \dW^a_s[k] \frac{\sin\big((t-s)|k|\big)}{|k|} \bigg), 
\end{align*}
which implies the desired identity. Thus, it remains to prove the first explicit formula \eqref{vector:eq-explicit-A0}. Using the definition of $A^0$, integration by parts, and $J^0\big|_{t=0}=0$, it holds that 
\begin{align*}
A^0(t,x) 
&= - \int_0^t \ds \, \frac{\sin\big((t-s)|\nabla|\big)}{|\nabla|} J^0(s,x) \\
&= - \frac{\cos\big((t-s)|\nabla|\big)}{|\nabla|^2} J^0(s,x) \Big|_{s=0}^t 
+ \int_0^t \ds \,  \frac{\cos\big( (t-s) |\nabla|\big)}{|\nabla|^2} (\partial_s J^0)(s,x) \\
&= - |\nabla|^{-2} J^0(t,x) + \int_0^t \ds \,  \frac{\cos\big( (t-s) |\nabla|\big)}{|\nabla|^2} (\partial_s J^0)(s,x) \\
&= |\nabla|^{-2} \int_0^t \ds \, \Big( \cos\big( (t-s) |\nabla|\big) -1 \Big) (\partial_s J^0)(s,x). 
\end{align*}
After inserting $\partial_t J^0 = - \partial_a J^a$ and using the definition of $J^a$, it follows that 
\begin{align*}
& |\nabla|^{-2} \int_0^t \ds \Big( \cos\big( (t-s) |\nabla|\big) -1 \Big) (\partial_s J^0)(s,x) \\ 
=& \, - |\nabla|^{-2} \int_0^t \ds \Big( \cos\big( (t-s) |\nabla|\big) -1 \Big) (\partial_a J^a)(s,x) \\
=& - \sum_{k \in \Z^2 \backslash \{0\}} \bigg( \frac{ik_a}{|k|^2} \Big( \int_0^t \dW^a_s[k] \big( \cos\big( (t-s) |k|\big) - 1\big) \Big)  e^{i \langle k,x \rangle}. 
\end{align*}
This yields the desired identity \eqref{vector:eq-explicit-A0}. 
\end{proof}

Equipped with Lemma \ref{vector:lem-explicit}, we now examine the space-time covariances of the vector potential. 

\begin{lemma}[Space-time covariances of the vector potential]\label{vector:lem-covariance} Let $t,t^\prime \geq 0$ and let $k,l \in \Z^2$. Then, we have the the following two identities: 
\begin{enumerate}[label=(\roman*)]
    \item \label{vector:item-covariance-time} (Time-component) It holds that 
    \begin{align*}
    &\E \Big[ \widehat{A}_0(t,k) \widehat{A}^0(t^\prime,l) \Big] \\
    =&  \mathbf{1} \big\{ k+l =0 \neq k \big\} \,  \frac{1}{|k|^2}  
    \bigg( 
    \big( t \wedge t^\prime \big) 
    \Big( 1+ \frac{1}{2} \cos \big( (t-t^\prime) |k| \big) \Big)  \\
    &+ \frac{1}{4|k|} \Big( 
    \sin\big( (t+t^\prime) |k| \big) + 4 \sin\big( t |k| \big) + 4 \sin\big( t^\prime |k| \big) + 3 \sin \big( |t-t^\prime| |k| \big) 
    \Big) \bigg). 
    \end{align*}
    \item \label{vector:item-convariance-spatial} (Spatial components) For all $a,b=1,2$, it holds that 
    \begin{align*}
    &\E \Big[ \widehat{A}^a(t,k) \widehat{A}^b(t^\prime,l) \Big] \\
    =& \mathbf{1}\big\{ k+l =0 \neq k \big\} \,  \delta^{ab}  \frac{1}{2|k|^2}  \bigg( 
    \big( t \wedge t^\prime \big) \cos \big( (t-t^\prime) |k| \big) 
    \\
    &\, + \frac{1}{2|k|} \Big( \sin\big( |t-t^\prime| |k| \big) - \sin\big( (t+t^\prime) |k| \big) \Big) \bigg) \\
    +& \mathbf{1} \big\{ k=l=0 \big\} \delta^{ab} \min(t,t^\prime)^2 \Big( \tfrac{1}{2} \max(t,t^\prime) - \tfrac{1}{6} \min(t,t^\prime) \Big). 
    \end{align*}
\end{enumerate}
\end{lemma}

\begin{remark}
Naturally, Lemma \ref{vector:lem-explicit} also allows us to compute the space-time covariances of $A^0$ and $A^a$, where $a=1,2$, but they will not be needed in this article. 
\end{remark}

\begin{proof}
We prove \ref{vector:item-covariance-time} and \ref{vector:item-convariance-spatial} separately. \\

\emph{Proof of \ref{vector:item-covariance-time}:} If $k=0$ or $l=0$, it follows from Lemma \ref{vector:lem-explicit} that $\widehat{A}_0(t,k)=0$ or $\widehat{A}^0(t^\prime,l)=0$. Thus, it remains to treat the case $k,l\neq 0$. Using Lemma \ref{vector:lem-explicit} and It\^{o}'s isometry, it holds that 
\begin{align*}
&\E \Big[ \widehat{A}_0(t,k) \widehat{A}^0(t^\prime,l) \Big] \\
=&  \E \bigg[ \frac{ik_a}{|k|^2} \int_0^t \dW^a_{s}[k] \Big( \cos\big( (t-s)|k|\big)-1\Big) \times \frac{il_b}{|l|^2} \int_0^{t^\prime} \dW^b_{s}[l] \Big( \cos\big( (t^\prime-s)|l|\big)-1\Big) \bigg] \\
=& \mathbf{1}\big\{ k+ l = 0 \big\}  \frac{k_a k_b \delta^{ab}}{|k|^4} \int_0^{t\wedge t^\prime} \ds \, \big( 1- \cos\big( (t-s) |k| \big) \big)  \big( 1- \cos\big( (t^\prime-s) |k| \big) \big) \\
=& \mathbf{1}\big\{ k+ l = 0 \big\}  \frac{1}{|k|^2}  \int_0^{t\wedge t^\prime} \ds \, \big( 1- \cos\big( (t-s) |k| \big) \big)  \big( 1- \cos\big( (t^\prime-s) |k| \big) \big). 
\end{align*}
Thus, it only remains to compute the $s$-integral, which is elementary. \\ 

\emph{Proof of \ref{vector:item-convariance-spatial}:} We only treat the case $k,l\neq 0$, since the remaining cases are similar.  
Using Lemma \ref{vector:lem-explicit} and It\^{o}'s isometry, it holds that 
\begin{align*}
&\E \Big[ \widehat{A}^a(t,k) \widehat{A}^b(t^\prime,l) \Big]\\
=& \E \bigg[ |k|^{-1} \int_0^t \dW^a_{s}[k] \sin\big((t-s) |k|\big) \times 
 |l|^{-1} \int_0^{t^\prime} \dW^b_{s}[k] \sin\big((t^\prime-s) |l|\big) \bigg] \\
 =& \, \mathbf{1}\big\{ k+ l = 0 \big\} \delta^{ab} |k|^{-2} \int_0^{t\wedge t^\prime} \ds \,   \sin\big((t-s) |k|\big)  \sin\big((t^\prime-s) |k|\big). 
\end{align*}
Thus, it only remains to compute the $s$-integral, which is elementary. 
\end{proof}

As a consequence of Lemma \ref{vector:lem-covariance}, we obtain the following regularity estimates. 

\begin{corollary}[Regularity estimates for the vector potential]\label{vector:cor-regularity}
For all $K\in \dyadic$, $T\geq 1$, $p\geq 1$, $\epsilon>0$, and $0\leq \alpha \leq 2$, it holds that 
\begin{align}
\E \Big[ \big\| A^\alpha_K \big\|_{\Cs_t^0 \Cs_x^{-\epsilon}([0,T]\times \T^2)}^p \Big]^{1/p} 
&\lesssim_\epsilon \sqrt{p} T^\theta K^{-\epsilon/2},  \label{vector:eq-regularity-1} \\
\E \Big[ \big\| \partial_t A^\alpha_K \big\|_{\Cs_t^0 \Cs_x^{-1-\epsilon}([0,T]\times \T^2)}^p \Big]^{1/p} 
&\lesssim_\epsilon \sqrt{p} T^\theta K^{-\epsilon/2}. \label{vector:eq-regularity-2}
\end{align}
\end{corollary}

\begin{proof}
Using Gaussian hypercontractivity (Lemma \ref{prelim:lem-hypercontractivity}) and translation-invariance (similar as in  \cite[Proof of Lemma 7.4]{BDNY24}), the proof of the first estimate  \eqref{vector:eq-regularity-1} can be reduced to
\begin{equation}\label{vector:eq-regularity-p1}
\sup_{t\in [0,T]} \E \Big[ \big\| P_K A^\alpha (t) \big\|_{H_x^{-\epsilon}}^2 \Big] \lesssim K^{-2\epsilon} T^\theta.
\end{equation}
In order to proof \eqref{vector:eq-regularity-p1}, we use Lemma \ref{vector:lem-covariance}, which yields
\begin{align*}
 \E \Big[ \big\| P_K A^\alpha (t) \big\|_{H_x^{-\epsilon}}^2 \Big] 
 = \sum_{k \in \Z^2} \langle k \rangle^{-2\epsilon} \rho_K^2(k) \E \Big[ \big| \widehat{A}^\alpha(t,k) \big|^2 \Big] \lesssim T^\theta \sum_{k\in \Z^2} \langle k \rangle^{-2-2\epsilon} \rho_K^2(k) \lesssim  K^{-2\epsilon} T^\theta. 
\end{align*}
It therefore remains to prove \eqref{vector:eq-regularity-2}. Using Lemma \ref{vector:lem-explicit}, it follows\footnote{The It\^{o}-integrals in Lemma \ref{vector:lem-explicit} are continuously differentiable since the integrand vanishes at $s=t$.} that
\begin{align}
\partial_t \widehat{A}^0(t,k) &= 1\big\{ k\neq 0 \big\} \frac{ik_a}{|k|} \int_0^t \dW^a_{s}[k]  \sin\big( (t-s)|k|\big), \label{vector:eq-regularity-p2} \\ 
\partial_t \widehat{A}^a(t,k) &= -  \int_0^t \dW^a_{s}[k] \cos\big((t-s) |k|\big). \label{vector:eq-regularity-p3} 
\end{align}
Using the explicit formulas in \eqref{vector:eq-regularity-p2} and \eqref{vector:eq-regularity-p3}, the second estimate \eqref{vector:eq-regularity-2} can then be proven using similar arguments as in the proof of \eqref{vector:eq-regularity-1} and we omit the details.
\end{proof}

We now treat the quadratic expression 
\begin{equation}
A_{\leq N,\alpha} A^{\alpha}_{\leq N} - \scrm^{\hspace{-0.4ex}2}_{\leq N}. 
\end{equation}
To this end, we first provide the precise definition of the renormalized mass $\scrm_{\leq N}$, which was previously left undefined. 

\begin{definition}[The renormalized mass]\label{vector:def-mass}
For all $N \in \dyadic$ and $t\in \R$, we define
\begin{equation*}
\scrm_{\leq N}(t) := \sqrt{ \frac{5}{2} \bigg( \sum_{n\in \Z^2\backslash \{ 0 \}} \frac{\rho_{\leq N}^2(n)}{|n|^2} \bigg) |t|}. 
\end{equation*}
\end{definition}

Equipped with Definition \ref{vector:def-mass}, we can now state the estimate of the quadratic expression.

\begin{lemma}[The renormalized quadratic term in the vector potential]\label{vector:lem-quadratic}
For all $M,N \in \dyadic$, $T\geq 1$, $p\geq 1$, and $\epsilon>0$, it holds that 
\begin{equation}\label{vector:eq-quadratic-1}
\E \Big[ \sup_N \Big\| A_{\leq N,\alpha} A^{\alpha}_{\leq N}- \scrm^{\hspace{-0.4ex}2}_{\leq N} \Big\|_{L_t^\infty \Cs^{-\epsilon}_x([0,T]\times \T_x^2)}^p \Big]^{1/p} \lesssim_\epsilon \, p T^\theta. 
\end{equation}
Furthermore, it holds that
\begin{equation}\label{vector:eq-quadratic-2}
\begin{aligned}
&\E \Big[ \sup_N \Big\| 
\Big( A_{\leq M,\alpha} A^{\alpha}_{\leq M}- \scrm^{\hspace{-0.4ex}2}_{\leq M} \Big) - 
\Big(   A_{\leq N,\alpha} A^{\alpha}_{\leq N}- \scrm^{\hspace{-0.4ex}2}_{\leq N}\Big) 
\Big\|_{L_t^\infty \Cs^{-\epsilon}_x([0,T]\times \T_x^2)}^p \Big]^{1/p} \\
\lesssim_\epsilon&\,  p T^\theta \min(M,N)^{-\kappa}. 
\end{aligned}
\end{equation}
\end{lemma}

\begin{proof}
We only prove \eqref{vector:eq-quadratic-1}, since \eqref{vector:eq-quadratic-2} follows from a minor modification of the same argument. 

We first decompose $A_{\leq N,\alpha} A^{\alpha}_{\leq N}$ into a resonant and non-resonant term. To be more precise, we decompose
\begin{align}
&A_{\leq N,\alpha} A_{\leq N}^\alpha \notag \\ 
=& \sum_{k,l\in \Z^2} \rho_{\leq N}(k) \rho_{\leq N}(l) \, 
\Big( \widehat{A}_\alpha(t,k) \widehat{A}^\alpha(t,l) 
- \E \big[ \widehat{A}_\alpha(t,k) \widehat{A}^\alpha(t,l) \big] \Big) 
e^{i\langle k+l,x \rangle} \label{vector:eq-quadratic-p1} \\ 
+& \sum_{k,l\in \Z^2} \rho_{\leq N}(k) \rho_{\leq N}(l) \,  
\E \big[ \widehat{A}_\alpha(t,k) \widehat{A}^\alpha(t,l) \big]  
e^{i\langle k+l,x \rangle} \label{vector:eq-quadratic-p2}. 
\end{align}
We now treat the non-resonant term \eqref{vector:eq-quadratic-p1} and resonant term \eqref{vector:eq-quadratic-p2} separately. \\ 

\emph{Contribution of the non-resonant term \eqref{vector:eq-quadratic-p1}:} We use the dyadic decomposition
\begin{align*}
&\sum_{k,l\in \Z^2} \rho_{\leq N}(k) \rho_{\leq N}(l) \, 
\Big( \widehat{A}_\alpha(t,k) \widehat{A}^\alpha(t,l) 
- \E \big[ \widehat{A}_\alpha(t,k) \widehat{A}^\alpha(t,l) \big] \Big) 
e^{i\langle k+l,x \rangle} \\
=& \sum_{\substack{K,L\colon \\ K,L \leq N}} 
\sum_{k,l\in \Z^2} \rho_{K}(k) \rho_{L}(l) \, 
\Big( \widehat{A}_\alpha(t,k) \widehat{A}^\alpha(t,l) 
- \E \big[ \widehat{A}_\alpha(t,k) \widehat{A}^\alpha(t,l) \big] \Big) 
e^{i\langle k+l,x \rangle}. 
\end{align*}
Due to Gaussian hypercontractivity (Lemma \ref{prelim:lem-hypercontractivity}) and translation-invariance (see e.g. \cite[Proof of Lemma 7.4]{BDNY24}), it suffices to prove that 
\begin{equation}\label{vector:eq-quadratic-p3}
\begin{aligned}
&\sup_{t\in [0,T]} \E \Big[ \Big\| 
\sum_{k,l\in \Z^2} \rho_{K}(k) \rho_{L}(l) \, 
\Big( \widehat{A}_\alpha(t,k) \widehat{A}^\alpha(t,l) 
- \E \big[ \widehat{A}_\alpha(t,k) \widehat{A}^\alpha(t,l) \big] \Big) 
e^{i\langle k+l,x \rangle} \Big\|_{H_x^{-\epsilon}}^2 \Big] \\ 
\lesssim& \,  T^\theta \max(K,L)^{-2\epsilon}. 
\end{aligned}
\end{equation}
We now note that, up to permutations, the family of random variables
\begin{equation*}
\Big( \widehat{A}_\alpha(t,k) \widehat{A}^\alpha(t,l) 
- \E \big[ \widehat{A}_\alpha(t,k) \widehat{A}^\alpha(t,l) \big] \Big)_{k,l \in \Z^2}
\end{equation*}
is orthogonal in $L^2(\Omega,\mathbb{P})$. Together with Lemma \ref{vector:lem-covariance} and Gaussian hypercontractivity\footnote{Gaussian hypercontractivity is used only for converting the second-moment bound from Lemma \ref{vector:lem-covariance} into a fourth-moment bound.} (Lemma \ref{prelim:lem-hypercontractivity}), it follows that 
\begin{align*}
&\E \Big[ \Big\| 
\sum_{k,l\in \Z^2} \rho_{K}(k) \rho_{L}(l) \, 
\Big( \widehat{A}_\alpha(t,k) \widehat{A}^\alpha(t,l) 
- \E \big[ \widehat{A}_\alpha(t,k) \widehat{A}^\alpha(t,l) \big] \Big) 
e^{i\langle k+l,x \rangle} \Big\|_{H_x^{-\epsilon}}^2 \Big]\\
\lesssim& \, \sum_{k,l\in \Z^2} \rho_K^2(k) \rho_L^2(l) \langle k + l \rangle^{-2\epsilon} \E \Big[ \Big|  \widehat{A}_\alpha(t,k) \widehat{A}^\alpha(t,l) 
- \E \big[ \widehat{A}_\alpha(t,k) \widehat{A}^\alpha(t,l) \big] \Big|^2 \Big] \\
\lesssim& \,  T^\theta \sum_{k,l\in \Z^2} \rho_K^2(k) \rho_L^2(l) \langle k + l \rangle^{-2\epsilon} \langle k\rangle^{-2} \langle l \rangle^{-2} \\ 
\lesssim& \, T^\theta \max(K,L)^{-2\epsilon}. 
\end{align*}
This completes the proof of \eqref{vector:eq-quadratic-p3}. \\ 

\emph{Contribution of the resonant term \eqref{vector:eq-quadratic-p2}:} By relabeling $k\in \Z^2$ as $n\in \Z^2$, using Lemma \ref{vector:lem-covariance}, and using Definition \ref{vector:def-mass}, we obtain that 
\begin{align}
& \sum_{k,l\in \Z^2} \rho_{\leq N}(k) \rho_{\leq N}(l) \,  
\E \big[ \widehat{A}_\alpha(t,k) \widehat{A}^\alpha(t,l) \big]  
e^{i\langle k+l,x \rangle} \notag \\
=& \, \frac{5}{2} \Big( \sum_{n\in \Z^2 \backslash \{0 \} } \frac{\rho_{\leq N}^2(n)}{|n|^2} \Big) t 
+ \sum_{n\in \Z^2 \backslash \{0\}} \frac{\rho_{\leq N}^2(n)}{|n|^3} \Big( -\frac{1}{4} \sin\big( 2t |n|\big) + 2 \sin\big( t|n|\big) \Big) + \frac{t^3}{3}. \notag
\end{align}
Here, the last summand comes from the contribution for $n=0$. 
Since the first summand coincides with $\scrm^{\hspace{-0.4ex}2}_{\leq N}$, it easily follows for all $\nu \in \R$  that 
\begin{equation*}
\Big\| \eqref{vector:eq-quadratic-p2} - \scrm^{\hspace{-0.4ex}2}_{\leq N} \Big\|_{L_t^\infty \Cs_x^\nu ([0,T]\times \T_x^2)} \lesssim T^3.  \qedhere
\end{equation*}
\end{proof}

In the next lemma, we utilize dispersive effects in order to bound a certain time-integral involving the vector potential $A$. This lemma will be used in the proof of Proposition \ref{operator:prop-main} below. 
\begin{lemma}[Dispersive estimate for a time-integral]\label{vector:lem-dispersive-time-integral}
Let $K \in \dyadic$ and let $0\leq \alpha \leq 2$. Then, it holds for all $T\geq 1$ and $p\geq 1$ that
\begin{equation*}
\E \Big[ \max_{\unormal \in \mathbb{S}^1} \sup_{\substack{\lambda \in \R \colon \\ |\lambda| \lesssim K^{10}}} \Big\| \int_0^t \dt^\prime e^{it^\prime \lambda} e^{-t^\prime \unormal \cdot \nabla} A^\alpha_K(t^\prime,x) \Big\|_{L_x^\infty H_t^{b_+}(\T_x^2 \times [0,T])}^p \Big]^{1/p} \lesssim \sqrt{p} T^\theta K^{-1/4+\delta}. 
\end{equation*}
\end{lemma}

\begin{proof}
We first estimate
\begin{equation}\label{vector:eq-dispersive-p1}
 \Big\| \int_0^t \dt^\prime e^{it^\prime \lambda} e^{-t^\prime \unormal \cdot \nabla} A^\alpha_K(t^\prime,x) \Big\|_{L_x^\infty H_t^{b_+}(\T_x^2 \times [0,T])} 
 \lesssim  \Big\|e^{it \lambda} e^{-t \unormal \cdot \nabla} A^\alpha_K(t,x) \Big\|_{L_x^\infty H_t^{b_+-1}(\T_x^2 \times [0,T])}. 
\end{equation}
In order to estimate the right-hand side of \eqref{vector:eq-dispersive-p1}, we first want to replace $b_+-1$ by $b_--1$ (where $b_-$ is as in Subsection \ref{section:notation}). To this end, we note that 
\begin{align*}
\Big\|e^{it \lambda} e^{-t \unormal \cdot \nabla} A^\alpha_K(t,x) \Big\|_{L_x^\infty H_t^{0}(\T_x^2 \times [0,T])} 
&\lesssim  T^{1/2} \Big\|e^{it \lambda} e^{-t \unormal \cdot \nabla} A^\alpha_K(t,x) \Big\|_{L_t^\infty L_x^\infty( [0,T] \times \T_x^2)} \\
&\lesssim T^{1/2} \big\| A_K^\alpha \big\|_{L_t^\infty L_x^\infty( [0,T] \times \T_x^2)}. 
\end{align*}
Using Corollary \ref{vector:cor-regularity}, it follows for all $\epsilon>0$ that 
\begin{equation}\label{vector:eq-dispersive-p2}
\E \Big[ \max_{\unormal \in \mathbb{S}^1} \sup_{\substack{\lambda \in \R \colon \\ |\lambda| \lesssim K^{10}}} \Big\|  e^{it \lambda} e^{-t \unormal \cdot \nabla} A^\alpha_K(t,x) \Big\|_{L_x^\infty H_t^{0}(\T_x^2 \times [0,T])}^p \Big]^{1/p} \lesssim \sqrt{p} T^\theta K^\epsilon.
\end{equation}
Equipped with \eqref{vector:eq-dispersive-p2}, we may then prove the desired estimate at $b_+-1$ by interpolation between $b_--1$ and $0$, and it then only remains to prove that 
\begin{equation}\label{vector:eq-dispersive-p3}
\E \Big[ \max_{\unormal \in \mathbb{S}^1} \sup_{\substack{\lambda \in \R \colon \\ |\lambda| \lesssim K^{10}}} \Big\|  e^{it \lambda} e^{-t \unormal \cdot \nabla} A^\alpha_K(t,x) \Big\|_{L_x^\infty H_t^{b_--1}(\T_x^2 \times [0,T])}^p \Big]^{1/p} \lesssim \sqrt{p} T^\theta K^{-1/4+\delta/2}. 
\end{equation}
Using a standard meshing argument (see e.g. \cite[Section 5.7]{BDNY24}) and Gaussian hypercontractivity (Lemma \ref{prelim:lem-hypercontractivity}), \eqref{vector:eq-dispersive-p3} can be further reduced to the estimate
\begin{equation}\label{vector:eq-dispersive-p4}
\max_{\unormal \in \mathbb{S}^1} \sup_{\substack{\lambda \in \R \colon \\ |\lambda| \lesssim K^{10}}} \max_{x\in \T^2} \E \Big[  \Big\|  e^{it \lambda} e^{-t \unormal \cdot \nabla} A^\alpha_K(t,x) \Big\|_{ H_t^{b_--1}( [0,T])}^2 \Big]^{1/2} \lesssim  T^\theta K^{-1/4+\delta/4}.
\end{equation}
Finally, since $b_- <1/2$, the $H_t^{b_--1/2}$-norm can be estimated by the sup-norm of the Fourier transform. Thus,  \eqref{vector:eq-dispersive-p4} can be further reduced to the estimate
\begin{equation}\label{vector:eq-dispersive-p5}
\max_{\unormal \in \mathbb{S}^1} \sup_{\substack{\lambda \in \R \colon \\ |\lambda| \lesssim K^{10}}} \max_{x\in \T^2} \sup_{\xi \in \R} \E \Big[  \Big|  \int_{\R} \dt \, \chi\big( t/T\big) e^{it\xi} e^{it \lambda} e^{-t \unormal \cdot \nabla} A^\alpha_K(t,x) \Big|^2 \Big]^{1/2} \lesssim  T^\theta K^{-1/4+\delta/4},
\end{equation}
where $\chi$ is a smooth, nonnegative cut-off function. 
After all of the reductions above, we now prove \eqref{vector:eq-dispersive-p5} using the covariance identity (Lemma \ref{vector:lem-covariance}) and lattice point counting estimates (Lemma \ref{prelim:lem-basic-counting}). We have that 
\begin{align}
&\E \Big[  \Big|  \int_{\R} \dt \, \chi\big( t/T\big) e^{it\xi} e^{it \lambda} e^{-t \unormal \cdot \nabla} A^\alpha_K(t,x) \Big|^2 \notag \\
=& \int_{\R} \dt \int_{\R} \dt^\prime \chi\big(t/T\big) \chi\big(t^\prime/T) e^{i (\xi+\lambda) (t-t^\prime)} \sum_{k,k^\prime \in \Z^2} \bigg( \rho_K(k) \label{vector:eq-dispersive-p6} \\
&\times \rho_K(k^\prime) e^{-i t \unormal \cdot k + i t^\prime \unormal \cdot k^\prime} 
\E \Big[ \widehat{A}_K^\alpha(t,k) \overline{\widehat{A}^\alpha_K(t^\prime,k)} \Big] e^{i\langle k-k^\prime, x \rangle} \bigg).   \notag 
\end{align}
To avoid confusion, we note that the index $0\leq \alpha \leq 2$ in \eqref{vector:eq-dispersive-p6} is fixed and not summed over. 
Using Lemma \ref{vector:eq-dispersive-p6}, we can write
\begin{equation}\label{vector:eq-dispersive-p7}
\E \Big[ \widehat{A}_K^\alpha(t,k) \overline{\widehat{A}^\alpha_K(t^\prime,k)} \Big]
= \mathbf{1}\big\{ k=k^\prime \big\} \bigg( \frac{\varphi(t,t^\prime)}{\langle k \rangle^2} \sum_{\sigma=0,\pm 1} c_\sigma e^{i\sigma (t-t^\prime) |k|} + \mathcal{O} \Big( \frac{T}{\langle k\rangle^3} \Big) \bigg), 
\end{equation}
where $\varphi \colon \R \times \R \rightarrow \R$ is continuous with bounded first-order derivatives, $c_{-1}$, $c_0$, and $c_1$ are constants, and $\mathcal{O}$ is the usual Landau symbol. Due to the estimate 
\begin{equation*}
\int_{\R} \dt \int_{\R} \dt^\prime \chi\big(t/T\big) \chi\big(t^\prime/T) \sum_{k\in \Z^2} \rho_K^2(k) \langle k \rangle^{-3} \lesssim T^2 K^{-1}, 
\end{equation*}
the contribution of the $\mathcal{O}$-term in \eqref{vector:eq-dispersive-p7} is (better than) acceptable. Thus, it remains to treat the contribution of the main term in \eqref{vector:eq-dispersive-p7}. To this end, we estimate 
\begin{align}
&\bigg| \int_{\R} \dt \int_{\R} \dt^\prime   \chi\big(t/T\big) \chi\big(t^\prime/T) e^{i (\xi+\lambda) (t-t^\prime)} \varphi(t,t^\prime)  \sum_{k\in \Z^2}  \bigg( \frac{\rho_K^2(k)}{\langle k \rangle^2}   e^{-i (t-t^\prime) \unormal \cdot k} e^{i\sigma (t-t^\prime) |k|} \bigg) \bigg| \notag \\
\lesssim& \, T^\theta \sum_{k\in \Z^2} \frac{\rho_K^2(k)}{\langle k \rangle^2} \Big( 1+ \big| \xi + \lambda - \unormal \cdot k + \sigma |k| \big| \Big)^{-1}. 
\end{align}
By using a level-set decomposition and the lattice point counting estimate (Lemma \ref{prelim:lem-basic-counting}), we have that 
\begin{align*}
&\sum_{k\in \Z^2} \frac{\rho_K^2(k)}{\langle k \rangle^2} \Big( 1+ \big| \xi + \lambda - \unormal \cdot k + \sigma |k| \big| \Big)^{-1} \\
\lesssim& \sum_{\substack{\mu\in \Z \colon \\ |\mu|\lesssim K}}
\Big( 1+ \big| \xi + \lambda - \mu \big| \Big)^{-1} 
\times K^{-2} \sup_{\mu \in \Z^2}  \sum_{k\in \Z^2}  \rho_K^2(k)
\mathbf{1} \big\{ \big| - \unormal \cdot k + \sigma |k| - \mu \big| \leq 1\big\} \\
\lesssim& \log(K) K^{-1/2}. 
\end{align*}
This yields an acceptable contribution to \eqref{vector:eq-dispersive-p4} and therefore completes the proof. 
\end{proof}

At the end of this section, we shift our attention from the vector potential $A$ to the linear stochastic object $z$ from \eqref{ansatz:eq-z}. Just like the vector potential $A$, $z$ solves a constant-coefficient wave equation with stochastic forcing. Using similar arguments as above, we also obtain the following regularity estimate for $z$.

\begin{corollary}[Regularity estimate for $z$]\label{vector:cor-regularity-z}
For all $K\in \dyadic$, $T\geq 1$, $p\geq 1$, and $\epsilon>0$, it holds that 
\begin{equation}\label{vector:eq-z-cs}
\E \Big[ \big\| P_K z \big\|_{\Cs_t^0 \Cs_x^{-\epsilon}([0,T] \times \T^2)}^p \Big]^{1/p} \lesssim \sqrt{p} T^\theta K^{-\epsilon/2}. 
\end{equation}
Furthermore, for all $\epsilon \geq 10 (b_+ - 1/2)$, it also holds that 
\begin{equation}\label{vector:eq-z-xnub}
\E \Big[ \big\| P_K z \big\|_{X^{-\epsilon,b_+}([0,T])}^p \Big]^{1/p} \lesssim \sqrt{p} T^\theta K^{-\epsilon/2}.
\end{equation}
\end{corollary}

\begin{proof}
The first estimate \eqref{vector:eq-z-cs} follows essentially as in Corollary \ref{vector:cor-regularity}. The second estimate \eqref{vector:eq-z-xnub} follows easily from Lemma \ref{prelim:lem-Xnub-ito}. 
\end{proof}
\section{Random operator estimates}\label{section:operator}

In this section, we prove random operator estimates for $\Lin[ll][\leq N]$, $\Lin[sim][\leq N]$, and $\Lin[gg][\leq N]$ from \eqref{ansatz:eq-Lin-ll}-\eqref{ansatz:eq-Lin-gg}. To this end, we first use a dyadic decomposition of the vector potential and define
\begin{align}
\Lin[ll][K] \phi &:=  2i \Duh \Big[ \partial_\alpha \Big( A_K^\alpha \parall \phi \Big) \Big] ,\label{operator:eq-LK-ll} \\
\Lin[sim][K] \phi &:=  2i \Duh \Big[ \partial_\alpha \Big( A_K^\alpha \parasim \phi \Big) \Big] ,\label{operator:eq-LK-sim} \\
\Lin[gg][K] \phi &:=  2i \Duh \Big[ \partial_\alpha \Big( A_K^\alpha \paragg \phi \Big) \Big] . \label{operator:eq-LK-gg}
\end{align}
Our main estimates, which address all three operators in \eqref{operator:eq-LK-ll}, \eqref{operator:eq-LK-sim}, and \eqref{operator:eq-LK-gg}, are collected in the next proposition. 

\begin{proposition}[Random operator estimates]\label{operator:prop-main}
Let $b_0,b_+$, and $\delta_1$ be as in Subsection \ref{section:notation}, let $p\geq 1$, let $T\geq 1$, and let $K\in \dyadic$. Then, we have the following three estimates: 
\begin{enumerate}[label=(\roman*)]
\item (High$\times$high-estimate) \label{operator:item-high-high} 
For all $\nu\geq \delta_1$, it holds that 
\begin{equation*}
\E \Big[ \big\| \Lin[sim][K] \big\|_{X^{\nu,b_0}([0,T])\rightarrow X^{\nu+1/4-\delta_1,b_+}([0,T])}^p \Big]^{1/p}  \lesssim \sqrt{p}\, T^\theta K^{-\kappa}. 
\end{equation*}
\item (High$\times$low-estimate) \label{operator:item-high-low} 
For all $\nu\leq 0 $, it holds that 
\begin{equation*}
\E \Big[ \big\| \Lin[gg][K] \big\|_{X^{\nu,b_0}([0,T])\rightarrow X^{\nu,b_+}([0,T])}^p \Big]^{1/p}  \lesssim \sqrt{p}\, T^\theta K^{\delta_1 - 1/4-\kappa}. 
\end{equation*}
\item (Low$\times$high-estimate) \label{operator:item-low-high} 
For all $\nu\in \R$, it holds that 
\begin{equation*}
\E \Big[ \big\| \Lin[ll][K] \big\|_{X^{\nu,b_0}([0,T])\rightarrow X^{\nu,b_+}([0,T])}^p \Big]^{1/p}  \lesssim \sqrt{p}\, T^\theta K^{\delta_1 - 1/4-\kappa}. 
\end{equation*}
\end{enumerate}
\end{proposition}

We remark that the three estimates in Proposition \ref{operator:prop-main} are stated in increasing order of difficulty.

\begin{remark}[High$\times$high-estimate]
We emphasize that the condition $\nu>0$, which is slightly weaker than $\nu\geq \delta_1$, is likely necessary for the high$\times$high-estimate. The reason is that the vector potential $A^\alpha$ is a linear wave with negative spatial regularity, and therefore the high$\times$high-product $A^\alpha \parasim \phi$ cannot be defined for all elements of $X^{0,b}$. \\
If the high$\times$high-estimate was satisfied for any $\nu <0$, then the proof of Theorem \ref{intro:thm-main} would likely be much simpler. At least for the derivative nonlinearity in \eqref{ansatz:eq-phiN-a}, one could then close all estimates directly in $X^{\nu,b}$. 
\end{remark}

\begin{remark}[Null-structure]
In the low$\times$high-estimate, we exhibit a $1/4$-gain in the lowest frequency-scale, i.e., a gain of $K^{-1/4}$. By using null structures (and potentially imposing a different gauge condition), it may be possible to replace $K^{-1/4}$ by $K^{-1/2}$. The reason for this is as follows: 

The proof of the low$\times$high-estimate crucially relies on the lattice point counting estimate from Lemma \ref{prelim:lem-basic-counting}. In the proof of Lemma \ref{prelim:lem-basic-counting}, the main contribution comes from angles $\Phi \sim K^{-1/2}$. Therefore, the worst contributions come from nearly parallel interactions, which are weakened by null forms. By utilizing a null structure, it may be possible to restrict to angles $\Phi \sim 1$, in which case the $K^{-1/2}$-factor in  Lemma \ref{prelim:lem-basic-counting} can be replaced by $K^{-1}$. This improved lattice point counting estimate would then lead to the aforementioned improvement in the low$\times$high-estimate. 
Using similar heuristics, one may hope to improve the high$\times$high-estimate by replacing $X^{\nu+1/4-\delta_1,b_+}$ with $X^{\nu+1/2-\delta_1,b_+}$ and the high$\times$low-estimate by replacing   $K^{\delta_1 - 1/4-\kappa}$ in with  $K^{\delta_1 - 1/2-\kappa}$.
However, since Theorem \ref{intro:thm-main} can be proven without any of these further improvements, we did not pursue this direction here.  
\end{remark}

\begin{remark}[Comparison of bilinear and operator estimates]
In the deterministic literature, bilinear estimates for wave equations are often stated in the form 
\begin{equation}\label{operator:eq-bilinear}
\big\| \Duh \big[ uv \big] \big\|_{X^{\nu_0,b_0}} 
\lesssim \| u \|_{X^{\nu_1,b_1}} \| v \|_{X^{\nu_2,b_2}}. 
\end{equation}
Of course, the bilinear estimate \eqref{operator:eq-bilinear} is equivalent to the operator estimate
\begin{equation}\label{operator:eq-operator}
\big\| v \mapsto \Duh \big[ uv \big] \big\|_{X^{\nu_2,b_2} \rightarrow X^{\nu_0,b_0}} \lesssim \| u \|_{X^{\nu_1,b_1}}.  
\end{equation}
In deterministic settings, the operator formulation \eqref{operator:eq-operator} is unnecessarily complicated. In random settings, however, it is more convenient to state estimates in the form of operator bounds.  The reason is that random objects, such as our vector potential $A^\alpha$, are often explicit and not just any element of a given function space. 
\end{remark}

\begin{remark}[Comparison with a deterministic bilinear estimate]\label{operator:rem-comparison-bilinear}
While the literature contains several deterministic bilinear estimates for wave equations (see e.g. \cite{DFS10,DFS12,FK20,KS02}), they generally require more regularity than is available in Proposition \ref{operator:prop-main}. For example, consider the null-form estimate from \cite[Corollary 13.4]{FK20}, which implies\footnote{In the notation of \cite{FK20}, the null-form estimate corresponds to $\beta_0=-3/4+\delta$, $\beta_+=0$, $\beta_-=-1/4+\delta$, $\alpha_1=5/4+\delta$, and $\alpha_2=1/4+\delta$.}
\begin{equation}\label{operator:eq-nullform}
\big\| \Duh \big[ Q_{12}(|\nabla|^{-1} A_\alpha , \phi) \big] 
\big\|_{X^{1/4+\delta,3/4+\delta}}
\lesssim \big\| A_\alpha \big\|_{X^{1/4+\delta,b}}
\big\| \phi \big\|_{X^{1/4+\delta,b}},
\end{equation}
where $0<\delta\ll 1$ and $0<b-1/2\ll 1$. While \eqref{operator:rem-comparison-bilinear} requires that $A$ and $\phi$ have spatial regularity $1/4+$, Proposition \ref{operator:prop-main} concerns a vector potential $A$ with regularity $0-$ and scalar fields $\phi$ with regularity $0+$.
\end{remark}

Once Proposition \ref{operator:prop-main} has been established, it is relatively easy to establish the following two lemmas. In the first lemma,  we control the resolvent $\big(1+\Lin[ll][\leq N]\big)^{-1}$ and the structured component $\chi_{\leq N}$. In the statement, $\sigma_A$ and $\sigma_z$ are the $\sigma$-Algebras generated by $A$ and $z$, which were previously introduced in Subsection \ref{section:probability}.

\begin{lemma}[Resolvent estimate]\label{operator:lem-resolvent}
Let $C=C(b_0,b_+,\delta_1)\geq 1$ be sufficiently large and $0<c\leq 1$ be sufficiently small. Then, for all $\lambda\geq 1$, there exist events $E^{(A)}_\lambda \in \sigma_A$ and $E^{(z)}_\lambda \in \sigma_z$ which satisfy 
\begin{equation}\label{operator:eq-event-probability}
\mathbb{P} \big( E^{(A)}_\lambda \big), \mathbb{P}\big( E^{(z)}_\lambda \big) \geq 1 - c \exp(-\lambda)
\end{equation}
and such that the following properties are satisfied: 
\begin{enumerate}[label=(\roman*)]
    \item (Resolvent estimate) On the event $E^{(A)}_\lambda$, we have for all $M,N\in \dyadic$, $T\geq 1$, and $\nu \in \R$ that 
    \begin{align}
    \Big\| \big( 1 + \Lin[ll][\leq N] \big)^{-1} \Big\|_{X^{\nu,b_0}([0,T]) \rightarrow X^{\nu,b_0}([0,T])} &\leq \exp \Big( C (T\lambda)^C\Big), \label{operator:eq-resolvent-1} \\
     \Big\| \big( 1 + \Lin[ll][\leq M] \big)^{-1} - \big( 1 + \Lin[ll][\leq N] \big)^{-1} \Big\|_{X^{\nu,b_0}([0,T]) \rightarrow X^{\nu,b_0}([0,T])} &\leq \exp \Big( C (T\lambda)^C\Big) \min(M,N)^{-\kappa}. \label{operator:eq-resolvent-2} 
    \end{align}
\item (Bound on $\chi$) On the event $E^{(A)}_\lambda \medcap E^{(z)}_\lambda$, we have for all $M,N\in \dyadic$ and $T\geq 1$ that
\begin{align}
\big\| \chi_{\leq N} \big\|_{X^{-\delta_1,b_0}([0,T])} 
&\leq \exp \Big( C (T\lambda)^C\Big), 
\label{operator:eq-chi-1} \\ 
\big\| \chi_{\leq M} - \chi_{\leq N} \big\|_{X^{-\delta_1,b_0}([0,T])} 
&\leq \exp \Big( C (T\lambda)^C\Big) \min(M,N)^{-\kappa}. 
\label{operator:eq-chi-2} 
\end{align}
\end{enumerate}
\end{lemma}

In the second lemma, we obtain a commutator estimate, which will be useful in the proof of Proposition \ref{product:prop-main} below.

\begin{lemma}[Commutator estimate]\label{operator:lem-commutator}
Let $C=C(b_0,b_+,\delta_1)\geq 1$ be sufficiently large and $0<c\leq 1$ be sufficiently small. Then, for all $\lambda \geq 1$, there exists an event $E^{(A)}_\lambda \in \sigma_A$ which satisfies
\begin{equation}\label{operator:eq-commutator-probability}
\mathbb{P} \big( E^{(A)}_\lambda \big) \geq 1 - c \exp(-\lambda)
\end{equation}
and such that, on this event, the following estimates hold: For all $L,M,N \in \dyadic$, $T\geq 1$, and $\nu \in \R$, it holds that 
\begin{align}
&\Big\| \big[ P_L , (1+\Lin[ll][\leq N])^{-1}\big] \Big\|_{X^{\nu,b_0}([0,T])\rightarrow X^{\nu,b_0}([0,T])} 
\leq L^{-1/4+\delta_1} \exp\big( C (T\lambda)^C \big), \label{operator:eq-commutator-1} \\ 
& \Big\| \big[ P_L , (1+\Lin[ll][\leq M])^{-1} - (1+\Lin[ll][\leq N])^{-1}\big] \Big\|_{X^{\nu,b_0}([0,T])\rightarrow X^{\nu,b_0}([0,T])}  \label{operator:eq-commutator-2} \\
&\leq L^{-1/4+\delta_1} \exp\big( C (T\lambda)^C \big) \min(M,N)^{-\kappa}.\notag
\end{align}
\end{lemma} 

The proofs of Lemma \ref{operator:lem-resolvent}  and Lemma \ref{operator:lem-commutator} will be presented at the end of Subsection \ref{section:operator-proof} and we first focus on the proof of Proposition \ref{operator:prop-main}. 

We split the proof of Proposition \ref{operator:prop-main} over the following two subsections. In Subsection \ref{section:operator-basic}, we first prove a basic random operator estimate which will serve as the main ingredient in the proof of Proposition \ref{operator:prop-main}. In Subsection \ref{section:operator-proof}, we then present the proof of Proposition \ref{operator:prop-main}. In addition to using the basic random operator estimate, we need to carefully treat cases in which the frequency-scales of the vector potential $A$ and argument $\phi$ are far apart. These cases are very delicate since, as mentioned in the introduction and further discussed in Section \ref{section:ansatz}, there is no nonlinear smoothing. 

\subsection{Basic random operator estimate}\label{section:operator-basic}

In this subsection, we state and proof a basic random operator estimate which will be the main ingredient in the proof of Proposition \ref{operator:prop-main}. 

\begin{lemma}[Basic random operator estimate]\label{operator:lem-basic}
Let $M,K,L\in \dyadic$, let $(c_{kl})_{k,l\in \Z^2}$ be a deterministic sequence, let $\sigma_1 \in \{-1,1\}$, and let $\sigma_2 \in \{ -1,0,1\}$. Define a random operator $\mathcal{R}$ by 
\begin{equation}
\begin{aligned}
\mathcal{R} \phi :=& 
\sum_{k,l \in \Z^2} \bigg( \rho_M(k+l) \rho_{K}(k) \rho_{L}(l) c_{kl} e^{i\langle k+l, x \rangle} \\
&\times \int_0^t \dt^\prime \, 
e^{i\sigma_1 (t-t^\prime) |k+l|} \bigg( \int_0^{t^\prime} \mathrm{d}W_s(k) \, e^{i\sigma_2 (t^\prime-s)|k|} \bigg) \widehat{\phi}(t^\prime,l) \bigg), 
\end{aligned}
\end{equation}
where $(W_s(k))_{k\in \Z^2}$ is as in Subsection \ref{section:probability}. Then, for all $T\geq 1$ and $p\geq 1$, it holds that 
\begin{equation}\label{operator:eq-basic-1}
\E \Big[ \big\| \, \mathcal{R} \, \big\|_{X^{0,b_0}([0,T]) \rightarrow X^{0,b_+}([0,T])}^p \Big]^{1/p} \lesssim \sqrt{p} T^\theta \max(K,L,M)^{\delta_0} \min(K,L,M)^{-1/4} K \big\| c_{kl} \big\|_{\ell^\infty_k \ell^\infty_l}. 
\end{equation}
In the case $\sigma_1 \neq \sigma_2$, we have the stronger estimate
\begin{equation}\label{operator:eq-basic-2}
\E \Big[ \big\| \, \mathcal{R} \, \big\|_{X^{0,b_0}([0,T]) \rightarrow X^{0,b_+}([0,T])}^p \Big]^{1/p} \lesssim \sqrt{p} T^\theta \max(K,L,M)^{\delta_0} \min(K,M)^{-1/4} K \big\| c_{kl} \big\|_{\ell^\infty_k \ell^\infty_l}, 
\end{equation}
in which $\min(K,L,M)$ has been replaced by $\min(K,M)$. 
\end{lemma}

The proof is based on the lattice point counting estimates from Lemma \ref{prelim:lem-basic-counting} and the random tensor estimates from \cite{DNY20}, which we recalled in Lemma \ref{prelim:lem-moment-method} above. 

\begin{proof}[Proof of Lemma \ref{operator:lem-basic}:]
The proof consists of three main steps. In the first step, we reduce the desired estimates \eqref{operator:eq-basic-1} and \eqref{operator:eq-basic-2} to a random tensor estimate. In the second step, we then reduce the random tensor estimate to a deterministic tensor estimate (using the moment method from Lemma \ref{prelim:lem-moment-method}). In the last step, we then prove the deterministic tensor estimate. \\

To unify and simplify the notation, we denote the main factor in the right-hand sides of \eqref{operator:eq-basic-1} and \eqref{operator:eq-basic-2} as $\frakC$, i.e., 
\begin{equation}
\frakC := K \| c_{kl} \|_{\ell^\infty_k \ell^\infty_l} 
\begin{cases}
\begin{tabular}{ll}
$\min(K,L,M)^{-1/4}$ &if $\sigma_1=\sigma_2$ \\ 
$\min(K,M)^{-1/4}$ &if $\sigma_1 \neq \sigma_2$
\end{tabular}
\end{cases}.
\end{equation}

\emph{Step 1: Reduction to a random tensor estimate.} 
Due to the definition of restricted $X^{\nu,b}$-norms, $\phi\in X^{0,b_0}([0,T])$ can be replaced by $\phi \in X^{0,b_0}(\R)$. Then, we can write
\begin{equation}
\widehat{\phi}(t^\prime,l) = \sum_{\sigma^\prime = \pm 1} \int_{\R} \dlambda^\prime \, e^{it^\prime \sigma^\prime |l|} e^{it^\prime \lambda^\prime} \widetilde{\phi}^{(\sigma^\prime)}(\lambda^\prime,l),
\end{equation}
where 
\begin{equation*}
\max_{\sigma^\prime = \pm 1} \big\| \langle \lambda^\prime \rangle^{b_0} \widetilde{\phi}^{(\sigma^\prime)}(\lambda^\prime,l) \big\|_{L_{\lambda^\prime}^2 \ell_l^2}  \sim \big\| \phi \big\|_{X^{0,b_0}(\R)}. 
\end{equation*}
We now write 
\begin{equation}
\Rc \phi = \sum_{\sigma^\prime = \pm 1} \int_\R \dlambda^\prime \Rc_{\sigma^\prime, \lambda^\prime} \widetilde{\phi}^{(\sigma^\prime)}(\lambda^\prime),
\end{equation}
where the operators
$\Rc_{\sigma^\prime,\lambda^\prime}\colon \ell^2(\Z^2) \rightarrow X^{0,b_+}([0,T])$ are defined by 
\begin{equation}
\begin{aligned}
\Rc_{\sigma^\prime, \lambda^\prime} v :=& \sum_{k,l \in \Z^2} \bigg[ \rho_M(k+l) \rho_K(k) \rho_L(l) c_{kl}\, e^{i\langle k+ l , x \rangle} 
\int_0^t \dt^\prime \, \bigg(  e^{i\sigma_1 (t-t^\prime) |k+l|}
\\
&\times  \Big( \int_0^{t^\prime} \mathrm{d}W_s(k) \, e^{i\sigma_2 (t^\prime-s)|k|}\Big) e^{i\sigma^\prime t^\prime |l|} e^{it^\prime \lambda^\prime} \bigg) \,  v_l \bigg]
\end{aligned}
\end{equation}
for all $v\in \ell^2(\Z^2)$. Using Cauchy-Schwarz in $\lambda^\prime$ and $b_0>1/2$,  we can then reduce the desired estimates \eqref{operator:eq-basic-1} and \eqref{operator:eq-basic-2} of the random operator $\Rc$ to the estimate
\begin{equation}\label{operator:eq-basic-p1}
\begin{aligned}
\E \bigg[ \max_{\sigma^\prime =\pm 1} \sup_{\lambda^\prime \in \R} \langle \lambda^\prime \rangle^{-(b_0-b_-) p} \big\| \Rc_{\sigma^\prime, \lambda^\prime} \big\|_{\ell^2 \rightarrow X^{0,b_+}([0,T])}^p \bigg]^{1/p} 
\lesssim \sqrt{p} T^\theta \max(K,L,M)^{\delta_0} \frakC. 
\end{aligned}
\end{equation}
Using Lemma \ref{prelim:lem-maxima-random} (see also the reductions in \cite[Subsection 5.7]{BDNY24}), \eqref{operator:eq-basic-p1} can be further reduced to proving that  
\begin{equation}\label{operator:eq-basic-p2}
\begin{aligned}
\max_{\sigma^\prime =\pm 1} \sup_{\lambda^\prime \in \R} \E \bigg[  \big\| \Rc_{\sigma^\prime, \lambda^\prime} \big\|_{\ell^2 \rightarrow X^{0,b_+}([0,T])}^p \bigg]^{1/p} 
\lesssim \sqrt{p} T^\theta \max(K,L,M)^{\delta_0/2} \frakC. 
\end{aligned}
\end{equation}
We now want to decrease the $b_+$-parameter in \eqref{operator:eq-basic-p2}. To this end, we first use the crude estimate
\begin{align*}
\big\| \Rc_{\sigma^\prime, \lambda^\prime} v \big\|_{X^{0,1}([0,T])} 
&\lesssim 
\big\| \partial_t \Rc_{\sigma^\prime, \lambda^\prime} v \big\|_{L_t^2 L_x^2([0,T]\times \T^2)}
+ \big\| \nabla_x \Rc_{\sigma^\prime, \lambda^\prime} v \big\|_{L_t^2 L_x^2([0,T]\times \T^2)}
+ \big\| \Rc_{\sigma^\prime, \lambda^\prime} v \big\|_{L_t^2 L_x^2([0,T]\times \T^2)} \\
&\lesssim T^\theta (KL)^3 \max_{\substack{\,\,\,k \in \Z^2\colon \\ |k|\sim K}} \max_{t^\prime \in [0,T]}
\bigg| \int_0^{t^\prime} \mathrm{d}W_s(k) \, e^{i\sigma_2 (t^\prime-s)|k|} \bigg| \, \max_{\substack{\,\,\, l \in \Z^2 \colon \\ |l|\sim L}} |v_l|, 
\end{align*}
which follows directly from the definition of $\Rc_{\sigma^\prime, \lambda^\prime}$. As a result, Lemma \ref{prelim:lem-maxima-random} and Doob's maximal inequality imply that 
\begin{align*}
    \max_{\sigma^\prime =\pm 1} \sup_{\lambda^\prime \in \R} \E \bigg[  \big\| \Rc_{\sigma^\prime, \lambda^\prime} \big\|_{\ell^2 \rightarrow X^{0,1}([0,T])}^p \bigg]^{1/p} 
    &\lesssim T^\theta (KL)^3 \E \bigg[  \max_{\substack{\,\,\,k \in \Z^2\colon \\ |k|\sim K}} \max_{t^\prime \in [0,T]}
\bigg| \int_0^{t^\prime} \mathrm{d}W_s(k) \, e^{i\sigma_2 (t^\prime-s)|k|} \bigg|^p \bigg]^{1/p} \\
&\lesssim \sqrt{p} T^\theta (KL)^4. 
\end{align*}
By interpolation, \eqref{operator:eq-basic-p2} can therefore be reduced to the estimate 
\begin{equation}\label{operator:eq-basic-p3}
\begin{aligned}
\max_{\sigma^\prime =\pm 1} \sup_{\lambda^\prime \in \R} \E \bigg[  \big\| \Rc_{\sigma^\prime, \lambda^\prime} \big\|_{\ell^2 \rightarrow X^{0,b_-}([0,T])}^p \bigg]^{1/p} 
\lesssim \sqrt{p} T^\theta \max(K,L,M)^{\delta_0/4} \frakC. 
\end{aligned}
\end{equation}
Comparing \eqref{operator:eq-basic-p3} with \eqref{operator:eq-basic-p2}, we decreased the $b$-parameter from $b_+$ to $b_-$, but this cost us a factor of  $\max(K,L,M)^{\delta_0/4}$. 

We now let $\varphi \in C^\infty_c(\R\rightarrow [0,1])$ be a smooth, compactly supported cut-off function satisfying $\varphi|_{[-1,1]}=1$. For any $\lambda \in \R$, $m\in \Z^2$,  and $v\in \ell^2(\Z^2)$, we then write 
\begin{equation}\label{operator:eq-basic-p4}
\int_{\R} \dt e^{-it\lambda} e^{-i t \sigma_1 |m|} \varphi(t/T) \widehat{(\Rc_{\sigma^\prime,\lambda^\prime} v)}(t,m) 
= \sum_{\mu \in \Z} \big(\Rc_{\lambda,\mu,\sigma^\prime,\lambda^\prime} v\big)  (m),
\end{equation}
where the operators $\Rc_{\lambda,\mu,\sigma^\prime,\lambda^\prime}\colon \ell^2(\Z^2)\rightarrow \ell^2(\Z^2)$ are defined as
\begin{equation}
\begin{aligned}
&\big( \Rc_{\lambda,\mu,\sigma^\prime,\lambda^\prime} v \big)(m) \\
:=& \, \sum_{\substack{k,l\in \Z^2\colon \\ k+l=m }}
\bigg[ \rho_M(k+l) \rho_K(k) \rho_L(l) c_{kl}  \, 
\mathbf{1} \big\{ - \sigma_1 |m| + \sigma_2 |k| + \sigma^\prime |l| \in [\mu,\mu+1) \big\} \\ 
&\times \bigg( \int_\R \mathrm{d}W_s(k) \, \bigg( \int_{\R} \dt \int_{\R} \dt^\prime \mathbf{1} \big\{ 0 \leq s \leq t^\prime \leq t \big\} \varphi(t/T) e^{-it\lambda} \\
&\times 
e^{it^\prime ( - \sigma_1 |m| + \sigma_2 |k| + \sigma^\prime |l| + \lambda^\prime)} e^{-i\sigma_2 s |k|} \bigg) \bigg) v_l
\bigg]. 
\end{aligned}
\end{equation}
In order to prove \eqref{operator:eq-basic-p3}, it therefore remains to prove 
\begin{equation}\label{operator:eq-basic-p5}
\max_{\sigma^\prime =\pm 1} 
\sup_{\lambda^\prime \in \R} 
\bigg\| \langle \lambda \rangle^{b_-} \, 
\E \Big[  \big\| \Rc_{\lambda,\mu,\sigma^\prime, \lambda^\prime} \big\|_{\ell^2 \rightarrow \ell^2}^p \Big]^{1/p} \bigg\|_{L_\lambda^2 \ell_\mu^1}
\lesssim \sqrt{p} T^\theta \max(K,L,M)^{\delta_0/4} \frakC.
\end{equation}
\emph{Step 2: From a random to a deterministic tensor estimate.} Using the moment method (Lemma \ref{prelim:lem-moment-method}), we now reduce the random tensor estimate \eqref{operator:eq-basic-p5} to a deterministic tensor estimate. Using Lemma \ref{prelim:lem-moment-method}, it holds that 
\begin{equation}\label{operator:eq-basic-p6}
\begin{aligned}
&\E \Big[ \big\| \Rc_{\lambda,\mu,\sigma^\prime, \lambda^\prime} \big\|_{\ell^2 \rightarrow \ell^2}^p \Big]^{1/p} \\
\lesssim& \, \sqrt{p}  \max(K,L,M)^{\delta/8} \max \Big( \| h_{klm} \|_{kl\rightarrow m}, \| h_{klm} \|_{l \rightarrow km} \Big) \max_{k,l,m\in \Z^2} \big\| f_{klm} (s) \big\|_{L_s^2},
\end{aligned}
\end{equation}
where 
\begin{equation}\label{operator:eq-basic-p7}
\begin{aligned}
h_{klm} &:= \big| c_{kl} \big| \, 
\mathbf{1}\big\{ m= k+l \big\}  
\mathbf{1}\big\{ |k| \sim K\big\} \mathbf{1}\big\{ |l| \sim L\big\} \mathbf{1}\big\{ |m| \sim M\big\} \\
&\times \mathbf{1} \big\{ - \sigma_1 |m| + \sigma_2 |k| + \sigma^\prime |l| \in [\mu,\mu+1) \big\}
\end{aligned}
\end{equation}
and 
\begin{equation}\label{operator:eq-basic-p8}
\begin{aligned}
f_{klm}(s) &:=
\mathbf{1}\big\{ m= k+l \big\}  
\mathbf{1}\big\{ |k| \sim K\big\} \mathbf{1}\big\{ |l| \sim L\big\} \mathbf{1}\big\{ |m| \sim M\big\} \\
&\times \mathbf{1} \big\{ - \sigma_1 |m| + \sigma_2 |k| + \sigma^\prime |l| \in [\mu,\mu+1) \big\} \\ 
&\times \int_{\R} \dt \int_{\R} \dt^\prime \mathbf{1} \big\{ 0 \leq s \leq t^\prime \leq t \big\} \varphi(t/T) e^{-it\lambda} 
e^{it^\prime ( - \sigma_1 |m| + \sigma_2 |k| + \sigma^\prime |l| + \lambda^\prime)} e^{-i\sigma_2 s |k|}.
\end{aligned}
\end{equation}
We note that $h$ depends on $\mu$ and $\sigma^\prime$ and $f$ depends on $\lambda$, $\mu$, $\sigma^\prime$, and $\lambda^\prime$, but we do not reflect this in our notation. In order to estimate $f_{klm}(s)$, we first integrate in $t^\prime$ and then integrate in $t$, which yields
\begin{equation}\label{operator:eq-basic-p9}
\big| f_{klm}(s) \big| 
\lesssim T^\theta 
\mathbf{1}\big\{ 0 \leq s \lesssim T \big\} \mathbf{1}\big\{ |\mu| \lesssim \max(K,L,M) \big\} 
\langle \mu + \lambda^\prime \rangle^{-1} 
\Big( \langle \lambda \rangle^{-1} + \langle \lambda - \mu - \lambda^\prime \rangle^{-1} \Big).
\end{equation}
As a result, it follows that
\begin{equation}\label{operator:eq-basic-p10}
\begin{aligned}
&\Big\| \langle \lambda \rangle^{b_-} \max_{k,l,m\in \Z^2} \big\| f_{klm}(s) \big\|_{L_s^2} \Big\|_{L_\lambda^2\ell_\mu^1} 
 \\
\lesssim& \, T^\theta \, \Big\| 
\mathbf{1}\big\{ |\mu| \lesssim \max(K,L,M) \big\}
\langle \lambda \rangle^{b_-} 
\langle \mu + \lambda^\prime \rangle^{-1} 
\big( \langle \lambda \rangle^{-1} + \langle \lambda - \mu - \lambda^\prime \rangle^{-1} \big) \Big\|_{L_\lambda^2 \ell_\mu^1}
\\ 
\lesssim& \,  \, T^\theta \log\big(2+\max(K,L,M)\big)
\big\| \langle \lambda \rangle^{b_- - 1} \big\|_{L_\lambda^2} \\
\lesssim&\, T^\theta \max(K,L,M)^{\delta_0/16}. 
\end{aligned}
\end{equation}
By inserting \eqref{operator:eq-basic-p10} into \eqref{operator:eq-basic-p6}, it follows that 
\begin{equation}\label{operator:eq-basic-p11}
\begin{aligned}
& \max_{\sigma^\prime =\pm 1} 
\sup_{\lambda^\prime \in \R}  \bigg\| \langle \lambda \rangle^{b_-} \, 
\E \Big[  \big\| \Rc_{\lambda,\mu,\sigma^\prime, \lambda^\prime} \big\|_{\ell^2 \rightarrow \ell^2}^p \Big]^{1/p} \bigg\|_{L_\lambda^2 \ell_\mu^1} \\
\lesssim& \, \sqrt{p} \max(K,L,M)^{\delta_0/8}
 \max_{\sigma^\prime =\pm 1}  \sup_{\mu \in \Z} \max \Big( \| h_{klm} \|_{kl\rightarrow m}, \| h_{klm} \|_{l \rightarrow km} \Big)\\
&\times \max_{\sigma^\prime =\pm 1} 
\sup_{\lambda^\prime \in \R}  \, \Big\| \langle  \lambda \rangle^{b_-} \max_{k,l,m\in \Z^2} \big\| f_{klm} \big\|_{L_s^2(\R)} \Big\|_{L_\lambda^2 \ell_\mu^1} \\
\lesssim&\, \sqrt{p} T^\theta \max(K,L,M)^{\delta_0/8+\delta_0/16}  \max_{\sigma^\prime =\pm 1}  \sup_{\mu \in \Z} \max \Big( \| h_{klm} \|_{kl\rightarrow m}, \| h_{klm} \|_{l \rightarrow km} \Big). 
\end{aligned}
\end{equation}
In order to prove \eqref{operator:eq-basic-p5}, it therefore remains to prove that 
\begin{equation}\label{operator:eq-basic-p12}
\begin{aligned}
  \max_{\sigma^\prime =\pm 1}  \sup_{\mu \in \Z} \max \Big( \| h_{klm} \|_{kl\rightarrow m}, \| h_{klm} \|_{l \rightarrow km} \Big) 
  \lesssim \max(K,L,M)^{\delta_0/16} \frakC. 
\end{aligned}
\end{equation}
\emph{Step 3: A deterministic tensor estimate.}
In the last step of this argument, we prove the deterministic tensor estimate in \eqref{operator:eq-basic-p12}.   We estimate the two arguments in the maximum in \eqref{operator:eq-basic-p12} separately. Using Schur's test and the lattice point counting estimate (Lemma \ref{prelim:lem-basic-counting}), the first argument is estimated by 
\begin{align*}
&\big\| h_{klm} \big\|_{kl\rightarrow m}^2\\
\lesssim& \, \| c_{kl}\|_{\ell^\infty_k \ell^\infty_l}^2 
\sup_{\substack{m \in \Z^2 \colon \\ |m| \sim M }} 
\sum_{\substack{k,l\in \Z^2}} \mathbf{1} \big\{ k+l=m \big\} \,  \mathbf{1} \big\{ |k| \sim K \big\}  \, 
\mathbf{1} \big\{ - \sigma_1 |m| + \sigma_2 |k| + \sigma^\prime |l| \in [\mu,\mu+1) \big\} \\
=& \, \| c_{kl}\|_{\ell^\infty_k \ell^\infty_l}^2 
\sup_{\substack{m \in \Z^2 \colon \\ |m| \sim M }} 
\sum_{\substack{k\in \Z^2  }}  \mathbf{1} \big\{ |k| \sim K \big\} \,  
\mathbf{1} \big\{ - \sigma_1 |m| + \sigma_2 |k| + \sigma^\prime |k-m| \in [\mu,\mu+1) \big\} \\
\lesssim&\, \min(K,M)^{-1/2} K^2 \| c_{kl}\|_{\ell^\infty_k \ell^\infty_l}^2  \lesssim \frakC^2. 
\end{align*}
We note that the first argument always obeys the better bound in \eqref{prelim:eq-basic-counting-plus}, i.e., the bound involving $\min(K,M)$ instead of $\min(K,L,M)$. 
Using Schur's test, the second argument is estimated by 
\begin{align}
&\big\| h_{klm} \big\|_{l\rightarrow km}^2 \notag \\
\lesssim& \, \| c_{kl}\|_{\ell^\infty_k \ell^\infty_l}^2 
\sup_{\substack{l \in \Z^2 \colon \\ |l| \sim L }} 
\sum_{\substack{k,m\in \Z^2}} \mathbf{1} \big\{ k+l=m \big\} \,  \mathbf{1} \big\{ |k| \sim K \big\}  \, 
\mathbf{1} \big\{ - \sigma_1 |m| + \sigma_2 |k| + \sigma^\prime |l| \in [\mu,\mu+1) \big\} \notag \\
\lesssim& \, \| c_{kl}\|_{\ell^\infty_k \ell^\infty_l}^2 
\sup_{\substack{l \in \Z^2 \colon \\ |l| \sim L }} 
\sum_{\substack{k\in \Z^2}}  \mathbf{1} \big\{ |k| \sim K \big\}  \, 
\mathbf{1} \big\{ - \sigma_1 |k+l| + \sigma_2 |k| + \sigma^\prime |l| \in [\mu,\mu+1) \big\}. \label{operator:eq-basic-p13}
\end{align}
We now estimate \eqref{operator:eq-basic-p13} using \eqref{prelim:eq-basic-counting-minus} when $\sigma_1 = \sigma_2$ and using either \eqref{prelim:eq-basic-counting-plus} or \eqref{prelim:eq-basic-counting-zero} when $\sigma_1 \neq \sigma_2$, which yields 
\begin{equation*}
\, \| c_{kl}\|_{\ell^\infty_k \ell^\infty_l}^2 
\sup_{\substack{l \in \Z^2 \colon \\ |l| \sim L }} 
\sum_{\substack{k\in \Z^2}}  \mathbf{1} \big\{ |k| \sim K \big\}  \, 
\mathbf{1} \big\{ - \sigma_1 |k+l| + \sigma_2 |k| + \sigma^\prime |l| \in [\mu,\mu+1) \big\} 
\lesssim \frakC^2. 
\end{equation*}
This completes the proof of \eqref{operator:eq-basic-p12} and hence the proof of this lemma.
\end{proof}

\subsection{Proof of Proposition \ref{operator:prop-main} and its consequences}\label{section:operator-proof}

In the main part of this subsection, we prove Proposition \ref{operator:prop-main} using the basic random operator estimate (Lemma \ref{operator:lem-basic}). At the end of this subsection, we then use Proposition \ref{operator:prop-main} (and its proof) to prove Lemma \ref{operator:lem-resolvent} and Lemma \ref{operator:lem-commutator}.

\begin{proof}[Proof of Proposition \ref{operator:prop-main}:]
We prove the high$\times$high, high$\times$low, and low$\times$high-estimate separately. \\

\emph{Proof of \ref{operator:item-high-high}: The high$\times$high-estimate.} We first use the dyadic decomposition 
\begin{equation}
(2i)^{-1} \Lin[sim][K] \phi = 
\sum_{\substack{L,M\in \dyadic \colon \\ K \sim L \gtrsim M}}
P_M \Duh \Big[ \partial_\alpha \Big( A^\alpha_K P_L \phi \Big)\Big]. 
\end{equation}
We then write the contribution of the spatial components as 
\begin{equation}\label{operator:eq-hh-p1}
\begin{aligned}
&P_M \Duh \Big[ \partial_a \Big( A^a_K P_L \phi \Big)\Big] \\
=& \, - i \sum_{k,l \in \Z^2} \bigg( \rho_K(k) \rho_L(l) \rho_{M}(k+l) (k+l)_a e^{i\langle k + l , x \rangle} \\
&\times \int_0^t \dt^\prime \, \frac{\sin\big( (t-t^\prime) |k+l|\big)}{|k+l|} \widehat{A}^a(t^\prime,k) \widehat{\phi}(t^\prime,l) \bigg). 
\end{aligned}
\end{equation}
Furthermore, using integration by parts and $A^0(0)=0$, we write the contribution of the temporal component as 
\begin{equation}\label{operator:eq-hh-p2}
\begin{aligned}
&P_M \Duh \Big[ \partial_0 \Big( A^0_K P_L \phi \Big)\Big] \\
=& \, - P_M \int_0^t \dt^\prime \, \frac{\sin\big( (t-t^\prime) |\nabla|\big)}{|\nabla|} \partial_{t^\prime} \big( A^0_K(t^\prime) P_L \phi(t^\prime) \big) \\
=&  P_M \int_0^t \dt^\prime \, \cos \big( (t-t^\prime) |\nabla| \big) A_K^0(t^\prime) P_L \phi (t^\prime) \\ 
=&  \sum_{k,l\in \Z^2} \bigg(  \rho_K(k) \rho_L(l) \rho_{M}(k+l)  e^{i\langle k + l , x \rangle} 
\int_0^t \dt^\prime \, \cos\big( (t-t^\prime) |k+l|\big) \widehat{A}^0(t^\prime,k) \widehat{\phi}(t^\prime,l) \bigg). 
\end{aligned}
\end{equation}
By using \eqref{operator:eq-hh-p1}, \eqref{operator:eq-hh-p2}, the explicit formula from Lemma \ref{vector:lem-explicit}, and the basic random operator estimate (Lemma \ref{operator:lem-basic}), it follows that 
\begin{align*}
&\E \Big[ \Big\| P_M \Duh \Big[ \partial_\alpha \Big( A^\alpha_K P_L \phi \Big) \Big] \Big\|_{X^{\nu,b_0}([0,T])\rightarrow X^{1/4+\nu-\delta_1,b_+}([0,T])}^p \Big]^{1/p} \\
\lesssim& \, \sqrt{p} T^\theta M^{\nu+1/4-\delta_1} L^{-\nu} \max(K,L,M)^\delta \min(K,L,M)^{-1/4}. 
\end{align*}
Since $K \sim L \gtrsim M$ and $\nu \neq \delta_1$, it holds that 
\begin{equation*}
M^{\nu+1/4-\delta_1} L^{-\nu}  \max(K,L,M)^\delta \min(K,L,M)^{-1/4} \sim  M^{\nu-\delta_1} K^{\delta-\nu} \lesssim K^{\delta-\delta_1}. 
\end{equation*}
Since $\kappa \ll \delta \ll \delta_1$, this implies the desired estimate. \\

\emph{Proof of \ref{operator:item-high-low}: The high$\times$low-estimate.} We first use the dyadic decomposition 
\begin{equation*}
(2i)^{-1} \Lin[gg][K] \phi = \sum_{\substack{L,M \in \dyadic \colon \\ M \sim K \gg L}} P_M \Duh \Big[ \partial_\alpha \Big( A^\alpha_K P_L \phi \Big) \Big]. 
\end{equation*}
Using the Lorenz condition, it holds that 
\begin{equation*}
\partial_\alpha \Big( A^\alpha_K P_L \phi \Big) 
= A_K^\alpha \partial_\alpha P_L \phi = A^0_K \partial_0 P_L \phi + A^a_K \partial_a P_L \phi. 
\end{equation*}
We now treat the temporal component ($\alpha=0$) and the spatial components ($\alpha=1,2$) separately. \\

\emph{The spatial components:} To treat the contributions of the spatial components, we write 
\begin{align*}
&P_M \Duh \Big[ A^a_K \partial_a P_L \phi \Big] \\
=& \, - i \sum_{k,l \in \Z^2} \bigg( \rho_K(k) \rho_L(l) \rho_{M}(k+l)\,  l_a \, e^{i\langle k+l ,x \rangle}  \int_0^t \dt^\prime \, \frac{\sin\big( (t-t^\prime) |k+l|\big)}{|k+l|} \widehat{A}^a(t^\prime,k) \widehat{\phi}(t^\prime,l) \bigg). 
\end{align*}
By inserting the explicit formula (Lemma \ref{vector:lem-explicit}) and using the basic random operator estimate (Lemma \ref{operator:lem-basic}), it follows that 
\begin{align*}
& \E \Big[ \Big\| P_M \Duh \Big[ A_K^a \partial_a P_L \phi \Big] \Big\|_{X^{\nu,b_0}([0,T]) \rightarrow X^{\nu,b_+}([0,T])}^p \Big]^{1/p} \\
\lesssim&\, \sqrt{p} T^\theta M^{\nu} L^{-\nu} \max(K,L,M)^\delta \min(K,L,M)^{-1/4} L M^{-1}. 
\end{align*}
Since $M \sim K \gg L$  and $\nu \leq 0$, it holds that 
\begin{equation*}
    M^{\nu} L^{-\nu} \max(K,L,M)^\delta \min(K,L,M)^{-1/4} L M^{-1} 
    \lesssim K^{-1+\delta} L^{3/4} \lesssim K^{-1/4+\delta}. 
\end{equation*}
Since $\delta \ll \delta_1$, this yields an acceptable contribution. \\

\emph{The temporal component:} This case is slightly technical. Since there is not much room in the $b$-parameter, we will prevent the time-derivative from hitting $\phi$. Using integration by parts and $A^0(0)=0$, we first write
\begin{equation}\label{operator:eq-hl-p1}
\begin{aligned}
&\Duh \Big[ A^0_K \partial_0 P_L \phi \Big] \\
=& \, - \sum_{k,l \in \Z^2} \bigg( \rho_K(k) \rho_L(l) \rho_{M}(k+l) e^{i \langle k+ l , x\rangle} 
\int_0^t \dt^\prime \, \frac{\sin\big( (t-t^\prime) |k+l|\big)}{|k+l|} 
\widehat{A}^0(t^\prime,k) \partial_{t^\prime} \widehat{\phi}(t^\prime,l) \bigg) \\
=&   \, \sum_{k,l \in \Z^2} \left( \rho_K(k) \rho_L(l) \rho_{M}(k+l) e^{i \langle k+ l , x\rangle} 
\int_0^t \dt^\prime \,  \partial_{t^\prime} \bigg( \frac{\sin\big( (t-t^\prime) |k+l|\big)}{|k+l|} 
\widehat{A}^0(t^\prime,k) \bigg)  \widehat{\phi}(t^\prime,l) \right). 
\end{aligned}
\end{equation}
Using Lemma \ref{vector:lem-explicit}, we further compute\footnote{We note that the It\^{o}-integral below is differentiable in $t^\prime$ since the $t^\prime$-dependent integrand vanishes at $s=t^\prime$.}
\begin{equation}\label{operator:eq-hl-p2}
\begin{aligned}
&\partial_{t^\prime} \left( \frac{\sin\big( (t-t^\prime) |k+l|\big)}{|k+l|} 
\widehat{A}^0(t^\prime,k) \right) \\
=& \, - i \frac{k_a}{|k|^2 |k+l|} 
\partial_{t^\prime} \left( \int_0^{t^\prime} \mathrm{d}W_s^a(k) \, \sin\big((t-t^\prime) |k+l|\big) \Big( \cos \big( (t^\prime-s) |k|\big)-1\Big) \right) \\
=&  \, - i \frac{k_a}{|k|^2 |k+l|}  \int_0^{t^\prime} \mathrm{d}W^a_s(k)\,  \partial_{t^\prime} \bigg( \sin\big((t-t^\prime) |k+l|\big) \Big( \cos \big( (t^\prime-s) |k|\big)-1\Big) \bigg). 
\end{aligned}
\end{equation}
Using trigonometric identities, it follows that 
\begin{align}
&\partial_{t^\prime} \bigg( \sin\big((t-t^\prime) |k+l|\big) \Big( \cos \big( (t^\prime-s) |k|\big)-1\Big) \bigg) \notag \\
=& -\frac{1}{2} \Big( |k+l| - |k| \Big) \cos \Big( (t-t^\prime) |k+l| + (t^\prime-s) |k| \Big) \label{operator:eq-hl-p3} \\
-& \frac{1}{2} \Big( |k+l| + |k| \Big) \cos \Big( (t-t^\prime) |k+l| - (t^\prime-s) |k| \Big) \label{operator:eq-hl-p4} \\
+& |k+l| \cos \Big( (t-t^\prime) |k+l|\Big). \label{operator:eq-hl-p5}
\end{align}
We now insert \eqref{operator:eq-hl-p3}-\eqref{operator:eq-hl-p5} into \eqref{operator:eq-hl-p2}, then insert \eqref{operator:eq-hl-p2} into \eqref{operator:eq-hl-p1}, and finally use the basic random operator estimate (Lemma \ref{operator:lem-basic}). For all three terms \eqref{operator:eq-hl-p3}, \eqref{operator:eq-hl-p4}, and \eqref{operator:eq-hl-p5},  this yields an acceptable contribution. For \eqref{operator:eq-hl-p3}, the reason is that $\big| |k+l|-|k| \big| \lesssim |l|$. For \eqref{operator:eq-hl-p4} and \eqref{operator:eq-hl-p5}, the reason is that we can use the improvement for $\sigma_1\neq \sigma_2$ in Lemma \ref{operator:lem-basic}. \\ 

\emph{Proof of \ref{operator:item-low-high}: The low$\times$high-estimate.} It suffices to treat the case $\nu=0$, since the estimates for $\nu\in \R$ are all equivalent. Since there is no nonlinear smoothing, i.e., no gain in the spatial frequency of $\phi$, the following argument is rather delicate. In particular, we cannot directly apply the basic random operator estimate (Lemma \ref{operator:lem-basic}), since this would lead to a $\delta$-loss in the highest frequency scale. We therefore split the proof into several steps. \\

\emph{Step (A): Reduction to much higher frequency-scales.} In the first step, we reduce to the case when the spatial frequency of $\phi$ is much higher than the spatial frequency of $A_K$. To this end, we decompose
\begin{align}
\Duh \Big[ \partial_\alpha \Big( A_K^\alpha \parall \phi \Big) \Big] 
&= \Duh \Big[ \partial_\alpha \Big( A_K^\alpha \parall P_{< K^{100}} \phi \Big) \Big] \label{operator:eq-lh-p1}\\
&+  \Duh \Big[ \partial_\alpha \Big( A_K^\alpha  P_{\geq K^{100}} \phi \Big) \Big] \label{operator:eq-lh-p2}. 
\end{align}
Similar as in the proof of \ref{operator:item-high-high}, it follows from the basic random operator estimate (Lemma \ref{operator:lem-basic}) and $\kappa,\delta_0 \ll \delta_1$ that 
\begin{align*}
  \E \Big[ \Big\| \Duh \Big[ \partial_\alpha \Big( A_K^\alpha  \parall P_{< K^{100}} \phi \Big) \Big] \Big\|_{X^{0,b_0}([0,T]) \rightarrow X^{0,b_+}([0,T])}^p \Big]^{1/p} 
  &\lesssim \sqrt{p} T^\theta \big( K^{100}\big)^{\delta} K^{-1/4} \\
  &\lesssim \sqrt{p} T^\theta K^{-1/4+\delta_1-\kappa}. 
\end{align*}
Thus, it remains to treat the contribution of \eqref{operator:eq-lh-p2}.\\

\emph{Step (B): Reduction to low-modulations.} 
Using the definition of the $X^{0,b_0}([0,T])$-norm, we can replace $\phi \in X^{0,b_0}([0,T])$ by $\phi \in X^{0,b_0}(\R)$. We then employ the usual decomposition of $\phi$ into time-modulated half-waves. To be precise, we write 
\begin{equation}\label{operator:eq-lh-p3}
\phi = \sum_{\sigma = \pm 1 } \int_{\R} \dlambda \, e^{i\lambda t} e^{i\sigma t|\nabla|} \phi^{(\sigma)}(\lambda),
\end{equation}
where, for each $\sigma\in \{\pm 1\}$ and $\lambda \in \R$, $\phi^{(\sigma)}(\lambda) \in L^2(\T^2_x)$ and 
\begin{equation*}
\| \phi \|_{X^{0,b_0}(\R)} \sim \max_{\sigma = \pm 1 }
\big\| \langle \lambda \rangle^{b_0} \phi^{(\sigma)}(\lambda) \big\|_{L_\lambda^2 L_x^2}. 
\end{equation*}
We then further decompose $\phi$ into a high and low-modulation term. To this end, we define\footnote{We introduce the parameter $\Lambda$ for expository purposes. Indeed, there is a-priori no reason for choosing  $\Lambda$ as being equal to $K$, but this choice can be justified using the numerology below.} $\Lambda := K$ and decompose
\begin{align}
\phi_{\hi} :=  \sum_{\sigma = \pm 1 } \int_{\R} \dlambda \,  \mathbf{1} \big\{ |\lambda| \geq \Lambda \big\} e^{i\lambda t} e^{i\sigma t|\nabla|} \phi^{(\sigma)}(\lambda), \label{operator:eq-lh-p4} \\ 
\phi_{\lo} :=  \sum_{\sigma = \pm 1 } \int_{\R} \dlambda \,  \mathbf{1}\big\{ |\lambda| \leq \Lambda \big\} e^{i\lambda t} e^{i\sigma t|\nabla|} \phi^{(\sigma)}(\lambda). \label{operator:eq-lh-p5}
\end{align}
For the high-modulation term, we have that 
\begin{align*}
&\Big\| \Duh \Big[ \partial_\alpha \Big( A^\alpha_K P_{\geq K^{100}} \phi_{\hi} \Big) \Big] \Big\|_{X^{0,b_+}([0,T])} \\
\lesssim& \, T^\theta \max_{\alpha=0,1,2} \Big\| A^\alpha_K  P_{\geq K^{100}} \phi_{\hi}  \Big\|_{L_t^2 L_x^2([0,T]\times \T^2)} \\
\lesssim& \, T^\theta \max_{\alpha=0,1,2} \Big\| A^\alpha_K \Big\|_{L_t^\infty L_x^\infty([0,T]\times \T^2)} 
\Big\| P_{\geq K^{100}} \phi_{\hi}  \Big\|_{L_t^2 L_x^2([0,T]\times \T^2)} \\
\lesssim& \,  T^\theta  \Lambda^{-b_0}  \Big\| A^\alpha_K \Big\|_{L_t^\infty L_x^\infty([0,T]\times \T^2)}  \big\| \phi \big\|_{X^{0,b_0}(\R)}. 
\end{align*}
Using Corollary \ref{vector:cor-regularity}, this yields an acceptable contribution. Thus, it remains to control the contribution of the low-modulation term in \eqref{operator:eq-lh-p5}. \\

\emph{Step (C): Simplifying the Duhamel integrals.} For the temporal component, we obtain from integration by parts and $A^0(0)=0$ that 
\begin{align*}
\Duh \Big[ \partial_0 \Big( A^0_K P_{\geq K^{100}} \phi_{\lo} \Big) \Big] = - \int_0^t \dt^\prime \, \cos \big( (t-t^\prime) |\nabla| \big) \Big( A^0_K P_{\geq K^{100}} \phi_{\lo} \Big). 
\end{align*}
For the spatial components, we write 
\begin{align*}
\Duh \Big[ \partial_a \Big( A^a_K P_{\geq K^{100}} \phi_{\lo} \Big) \Big] 
=  \frac{\partial_a}{|\nabla|}  
\int_0^t \dt^\prime \, \sin \big( (t-t^\prime) |\nabla| \big) \Big( A^a_K P_{\geq K^{100}} \phi_{\lo} \Big)
\end{align*}
Due to the boundedness of $\partial_a / |\nabla|$ on $L^2(\T_x^2)$, it therefore suffices to estimate 
\begin{equation}\label{operator:eq-lh-p6}
e^{i\sigma t |\nabla|} \int_0^t \dt^\prime \, e^{-i\sigma t^\prime |\nabla|} \Big(  A^\alpha_K \, P_{\geq K^{100}} \phi_{\lo} \Big)
\end{equation}
for all $\sigma = \pm 1$ and $0\leq \alpha \leq 2$. \\

\emph{Step (D): Angular decomposition.} For any $\varphi \in (-\pi,\pi]$, we define the projection operator $\Qphi$ by 
\begin{equation}\label{operator:eq-lh-p7}
\widehat{\Qphi f}(m) = \mathbf{1} \big\{ \big| \angle (m , e_1) - \varphi\big| \leq 2\pi K^{-99} \big\} \widehat{f}(m)
\end{equation}
for all $m\in \Z^2\backslash \{0\}$ and\footnote{In our argument, the operator $\Qphi$ is only applied to high-frequency functions, and therefore the definition for $m=0$ is irrelevant.} $\widehat{\Qphi f}(0)=0$. In \eqref{operator:eq-lh-p7}, the distance between $\angle (m,e_1)$ and $\varphi$ is measured in $\T_x$, i.e., modulo multiples of $2\pi$. For future use, we also define the unit normal corresponding to the angle $\varphi$ as
\begin{equation}\label{operator:eq-lh-p8}
\nphi := \big( \cos \varphi , \,   \sin \varphi \big).  
\end{equation}
For any $k\in \Z^2$ and $l\in \Z^2$ satisfying $|k|\sim K$ and $|l| \gtrsim K^{100}$, it holds that 
\begin{equation*}
\Big| \angle ( k+l , e_1 ) - \angle (l, e_1) \Big| \lesssim K^{-99}. 
\end{equation*}
Using almost orthogonality and Lemma \ref{prelim:lem-maxima-random}, it therefore suffices to estimate 
\begin{equation}\label{operator:eq-lh-p9}
e^{i\sigma t |\nabla|} \int_0^t \dt^\prime \, e^{-i\sigma t^\prime |\nabla|} \Big(  A^\alpha_K  \cdot \Qphi P_{\geq K^{100}} \phi_{\lo} \Big)
\end{equation}
for a fixed $\varphi \in (-\pi,\pi]$. We now simplify the action of $e^{-i\sigma t^\prime |\nabla|}$ on the product of $A^\alpha_K$ and  $\Qphi P_{\geq K^{100}} \phi_{\lo}$. For any $k,l\in \Z^2$ satisfying $|k|\sim K$, $|l|\gtrsim K^{100}$, and $|\angle (l, e_1 )  -\varphi|\lesssim K^{-99}$, it holds that 
\begin{equation}\label{operator:eq-lh-p10}
\Big| |k+ l| - \big( \nphi \cdot k + |l| \big) \Big| \lesssim K^{-98}. 
\end{equation}
Then, it easily follows that 
\begin{equation}\label{operator:eq-lh-p11}
\begin{aligned}
&\Big\| 
e^{-i\sigma t^\prime |\nabla|} \Big(  A^\alpha_K  \cdot \Qphi P_{\geq K^{100}} \phi_{\lo} \Big) 
-   \Big( e^{-\sigma t^\prime \nphi \cdot \nabla}   A^\alpha_K  \Big) \cdot \Big( e^{-i\sigma t^\prime |\nabla|} \Qphi P_{\geq K^{100}} \phi_{\lo} \Big) \Big\|_{L_x^2} \\
\lesssim& \, K^{-96} \big\| A_K^\alpha(t^\prime) \big\|_{L_x^2} \big\| \phi_{\lo} \big\|_{L_x^2}. 
\end{aligned}
\end{equation} 
Due to \eqref{operator:eq-lh-p11}, it then remains to control 
\begin{equation}\label{operator:eq-lh-p12}
e^{i\sigma t |\nabla|} \int_0^t \dt^\prime \, 
\Big( e^{-\sigma t^\prime \nphi \cdot \nabla}   A^\alpha_K  \Big)
\cdot 
\Big( e^{-i\sigma t^\prime |\nabla|} \Qphi P_{\geq K^{100}} \phi_{\lo} \Big) . 
\end{equation}

\emph{Step (D): Inserting the half-wave decomposition of $\phi_{\lo}$.}
Using \eqref{operator:eq-lh-p5}, it holds that 
\begin{align}
&e^{i\sigma t |\nabla|} \int_0^t \dt^\prime \, 
\Big( e^{-\sigma t^\prime \nphi \cdot \nabla}   A^\alpha_K  \Big)
\cdot 
\Big( e^{-i\sigma t^\prime |\nabla|} \Qphi P_{\geq K^{100}} \phi_{\lo} \Big) \notag  \\
=& \, \sum_{\sigma^\prime=\pm 1} \int_{\R} \dlambda \mathbf{1} \big\{ |\lambda| \leq \Lambda \big\} e^{i\sigma t |\nabla|} \int_0^t \dt^\prime \, \bigg[  \Big( e^{-\sigma t^\prime \nphi \cdot \nabla}   A^\alpha_K  \Big) \label{operator:eq-lh-p13} \\
&\times \, \Big( e^{it^\prime \lambda} e^{i(\sigma^\prime-\sigma)t^\prime |\nabla|} \Qphi P_{\geq K^{100}}  \phi_{\lo}^{(\sigma^\prime)}(\lambda) \Big) \bigg]. \notag 
\end{align}
If $\sigma^\prime \neq \sigma$, integration by parts in $t^\prime$ gains a factor of 
\begin{equation*}
K^{-100} \big( \Lambda + K \big) \lesssim K^{-99}, 
\end{equation*}
 so that the corresponding contribution can be controlled using crude arguments. Thus, it remains to treat the case $\sigma^\prime=\sigma$.  In this case, we can simplify \eqref{operator:eq-lh-p13}, which can be written as 
 \begin{equation}\label{operator:eq-lh-p14}
\begin{aligned}
\int_{\R} \dlambda \mathbf{1} \big\{ |\lambda| \leq \Lambda \big\} \,  e^{i\sigma t |\nabla|} \bigg[  \bigg(\int_0^t \dt^\prime \, e^{it^\prime\lambda}  e^{-\sigma t^\prime \nphi \cdot \nabla}   A^\alpha_K  \bigg) \cdot
\Qphi P_{\geq K^{100}}  \phi_{\lo}^{(\sigma)}(\lambda) \bigg].
\end{aligned}
 \end{equation}
 The most important aspect of \eqref{operator:eq-lh-p14} is that there no longer are any $t^\prime$-dependent operators acting on  $\phi^{(\sigma)}_{\lo}(\lambda)$, which can therefore be pulled outside of the $t^\prime$-integral. \\
 
 \emph{Step (E): Finishing up the proof.} 
 It remains to estimate \eqref{operator:eq-lh-p14}. We have that 
 \begin{align}
&\bigg\| \int_{\R} \dlambda \mathbf{1} \big\{ |\lambda| \leq \Lambda \big\} \,  e^{i\sigma t |\nabla|} \bigg[  \bigg(\int_0^t \dt^\prime \, e^{it^\prime\lambda}  e^{-\sigma t^\prime \nphi \cdot \nabla}   A^\alpha_K  \bigg) \cdot
\Qphi P_{\geq K^{100}}  \phi_{\lo}^{(\sigma)}(\lambda) \bigg] \bigg\|_{X^{0,b_+}([0,T])} \notag \\
\leq& \, \int_{\R} \dlambda \,  \bigg\| e^{i\sigma t |\nabla|} \bigg[  \bigg(\int_0^t \dt^\prime \, e^{it^\prime\lambda}  e^{-\sigma t^\prime \nphi \cdot \nabla}   A^\alpha_K  \bigg) \cdot
\Qphi P_{\geq K^{100}}  \phi_{\lo}^{(\sigma)}(\lambda) \bigg] \bigg\|_{X^{0,b_+}([0,T])} \notag \\
\lesssim& \, \int_{\R} \dlambda \,   \bigg\|  \bigg(\int_0^t \dt^\prime \, e^{it^\prime\lambda}  e^{-\sigma t^\prime \nphi \cdot \nabla}   A^\alpha_K  \bigg) \cdot
\Qphi P_{\geq K^{100}}  \phi_{\lo}^{(\sigma)}(\lambda)  \bigg\|_{L_x^2 H_t^{b_+}(\T^2 \times [0,T])} \notag \\
\lesssim& \, \bigg\| \bigg(\int_0^t \dt^\prime \, e^{it^\prime\lambda}  e^{-\sigma t^\prime \nphi \cdot \nabla}   A^\alpha_K  \bigg) \bigg\|_{L_x^\infty H_t^{b_+}(\T^2\times [0,T])} \, 
\Big\| \Qphi P_{\geq K^{100}}  \phi_{\lo}^{(\sigma)}(\lambda) \Big\|_{L_\lambda^1 L_x^2(\R\times \T^2)} \label{operator:eq-lh-p15}. 
 \end{align}
 The first factor in \eqref{operator:eq-lh-p15} is now estimated using Lemma \ref{vector:lem-dispersive-time-integral} and the second factor in \eqref{operator:eq-lh-p15} is estimated using Cauchy-Schwarz in $\lambda$ and the definition of the $X^{0,b_0}$-norm. 
\end{proof}

Equipped with Proposition \ref{operator:prop-main}, Lemma \ref{operator:lem-resolvent} follows from Gronwall-type argument. The only slightly non-standard aspect of the argument is that Proposition \ref{operator:prop-main} only yields bounds on intervals $[t_0,t_1]$ when $t_0=0$ and not for general $t_0 \geq 0$. For this reason, we sketch the (otherwise standard) argument.

\begin{proof}[Proof of Lemma \ref{operator:lem-resolvent}:]
We first prove the resolvent estimate \eqref{operator:eq-resolvent-1}, which is the main part of the argument.  To this end, let $C_0=C_0(b_0,b_+)$ be sufficiently large. Using Proposition \ref{operator:prop-main}  and a standard union-bound estimate over time-scales $T\in \dyadic$, we obtain that there exists an event $E^{(A)}_\lambda \in \sigma_A$ which satisfies the probability estimate \eqref{operator:eq-event-probability} and such that, on this event, we have for all $t\geq 0$ that  
\begin{equation}\label{operator:eq-resolvent-p2}
\sup_N \big\| \Lin[ll][\leq N] \big\|_{X^{\nu,b_0}([0,t])\rightarrow X^{\nu,b_+}([0,t])} \leq C_0  \max(1,|t|)^{C_0} \lambda. 
\end{equation}
For the rest of this argument, we restrict ourselves to this event $E^{(A)}_\lambda$. \\

We now let $N\in \dyadic$, let $T\geq 1$, and let $v \in X^{\nu,b_0}([0,T])$. Using soft arguments (due to the finiteness of $N$), it follows that there exists a unique solution $\chi \in X^{\nu,b}([0,T])$ of 
\begin{equation}\label{operator:eq-resolvent-p3}
\chi + \Lin[ll][\leq N] \chi = v. 
\end{equation}
In order to prove the resolvent bound \eqref{operator:eq-resolvent-1}, it remains to prove that 
\begin{equation}\label{operator:eq-resolvent-p4}
\big\| \chi \big\|_{X^{\nu,b_0}([0,T])} \leq \exp\big( C (T\lambda)^C \big) \big\| v \big\|_{X^{\nu,b_0}([0,T])}.
\end{equation}
To this end, we let $\tau=\tau(T,\lambda)\in (0,1)$ be a time-step which remains to be chosen. For all $j\geq 1$, we define 
\begin{equation*}
D_j := \big\| \chi \big\|_{X^{\nu,b_0}([0,j\tau])} + \big\| \Lin[ll][\leq N] \chi \big\|_{X^{\nu,b_0}([0,j\tau])}. 
\end{equation*}
Assuming that $j\in \Nzero$ satisfies $(j+1)\tau \leq T$, we obtain from 
Lemma \ref{prelim:lem-Xnub}  and \eqref{operator:eq-resolvent-p3} that
\begin{align}
D_{j+1} &= \big\| \chi \big\|_{X^{\nu,b_0}([0,(j+1)\tau])} + \big\| \Lin[ll][\leq N] \chi \big\|_{X^{\nu,b_0}([0,(j+1)\tau])} \notag \\
&\leq 2 \big\| \Lin[ll][\leq N] \chi \big\|_{X^{\nu,b_0}([0,(j+1)\tau])}
 + \big\| v \big\|_{X^{\nu,b_0}([0,(j+1)\tau])} \notag \\
&\lesssim_{b_0,b_+} \big\| \Lin[ll][\leq N] \chi \big\|_{X^{\nu,b_0}([0,j\tau])}
+ \tau^{b_+-b_0}  \big\| \Lin[ll][\leq N] \chi \big\|_{X^{\nu,b_+}([0,(j+1)\tau])} + \big\| v \big\|_{X^{\nu,b_0}([0,T])}. \label{operator:eq-resolvent-p5}
\end{align}
Using \eqref{operator:eq-resolvent-p2}, it follows that
\begin{align}
\eqref{operator:eq-resolvent-p5} \, &\lesssim_{b_0,b_+} D_j + \tau^{b_+-b_0} T^{C_0} \lambda \big\| \chi \big\|_{X^{\nu,b_0}([0,(j+1)\tau])} 
+ \big\| v \big\|_{X^{\nu,b_0}([0,T])} \notag \\
&\lesssim_{b_0,b_+} D_j +  \big\| v \big\|_{X^{\nu,b_0}([0,T])} + \tau^{b_+-b_0}  T^{C_0} \lambda D_{j+1}.   \label{operator:eq-resolvent-p6}
\end{align}
We now let $C_1=C_1(b_0,b_+,C_0)\geq 1$ be sufficiently large. Assuming that the time-step $\tau$ satisfies
\begin{equation*}
 C_1 (T\lambda)^{C_1} \tau \leq \frac{1}{2}, 
\end{equation*}
 it follows from \eqref{operator:eq-resolvent-p6} that
\begin{equation}\label{operator:eq-resolvent-p7}
D_{j+1} \leq C_1 \Big( D_j +  \big\| v \big\|_{X^{\nu,b_0}([0,T])} \Big).  
\end{equation}
The desired estimate \eqref{operator:eq-resolvent-p4} now follows by simply iterating \eqref{operator:eq-resolvent-p7} and our condition on $\tau$. This completes the proof of \eqref{operator:eq-resolvent-1}.\\

It remains to prove the second resolvent estimate \eqref{operator:eq-resolvent-2} and the two estimates \eqref{operator:eq-chi-1} and \eqref{operator:eq-chi-2} for the structured component. The second resolvent estimate \eqref{operator:eq-resolvent-2} follows directly from the resolvent identity, the first resolvent estimate \eqref{operator:eq-resolvent-1}, and the decay in $K$ in Proposition \ref{operator:prop-main}. The estimates \eqref{operator:eq-chi-1} and \eqref{operator:eq-chi-2} for the structured component follow directly from the two resolvent estimates \eqref{operator:eq-resolvent-1} and \eqref{operator:eq-resolvent-2}  and the regularity estimate for $z$ (Corollary \ref{vector:cor-regularity-z}).
\end{proof}

At the end of this subsection, we now present the proof of Lemma \ref{operator:lem-commutator}. 

\begin{proof}[Proof of Lemma \ref{operator:lem-commutator}:]
We only prove \eqref{operator:eq-commutator-1}, since \eqref{operator:eq-commutator-2} follows from a minor modification. First, we observe the commutator identity 
\begin{equation}\label{operator:eq-commutator-p1}
\big[ P_L , (1+\Lin[ll][\leq N])^{-1} \big] = 
-  (1+\Lin[ll][\leq N])^{-1} \big[ P_L,  1+\Lin[ll][\leq N] \big]  (1+\Lin[ll][\leq N])^{-1}. 
\end{equation}
Due to \eqref{operator:eq-commutator-p1} and Lemma \ref{operator:lem-resolvent}, it suffices to prove for all $L\in \dyadic$, $p\geq 1$, and $T\geq 1$ that
\begin{equation}\label{operator:eq-commutator-p2}
\E \Big[ \sup_{N \in \dyadic}
\Big\| \Big[  P_L,  \Lin[ll][\leq N] \Big] \Big\|_{X^{\nu,b_0}([0,T])\rightarrow X^{\nu,b_0}([0,T])}^p \Big]^{1/p} \lesssim \sqrt{p} T^\theta L^{-1/4+\delta_1-\kappa}. 
\end{equation}
To this end, we further decompose
\begin{equation}\label{operator:eq-commutator-p3}
\Big[  P_L,  \Lin[ll][\leq N] \Big] 
= 2i \sum_{\substack{K,L^\prime \in \dyadic \colon \\ K \ll L \sim L^\prime}} \Duh \Big[ P_L \Big( \partial_\alpha \big( A^\alpha_K \,  P_{L^\prime} \phi\big) \Big) - \partial_\alpha \big( A^\alpha_K \, P_L P_{L^\prime} \phi \big) \Big]. 
\end{equation}
We now note that, for any $k,l\in \Z^2$, it holds that  
\begin{equation}\label{operator:eq-commutator-p4}
\Big| \rho_L(k+l) \rho_K(k) \rho_{L^\prime}(l) - \rho_L(l)  \rho_K(k) \rho_{L^\prime}(l) \Big| \lesssim K L^{-1} \rho_K(k) \rho_{L^\prime}(l).
\end{equation}
By combining a Fourier expansion of \eqref{operator:eq-commutator-p3}, the estimate \eqref{operator:eq-commutator-p4}, and the basic random operator estimate (Lemma \ref{operator:lem-basic}), it follows that 
\begin{align*}
&\E \bigg[ \bigg\|  \Duh \Big[ P_L \Big( \partial_\alpha \big( A^\alpha_K \,  P_{L^\prime} \phi\big) \Big) - \partial_\alpha \big( A^\alpha_K \, P_L P_{L^\prime} \phi \big) \Big] \bigg\|_{X^{\nu,b_0}([0,T]) \rightarrow X^{\nu,b_+}([0,T])}^p \bigg]^{1/p} \\ 
\lesssim& \,  \sqrt{p} T^\theta  K^{3/4+\delta} L^{-1} 
\lesssim \sqrt{p} T^\theta   L^{-1/4+\delta}.  
\end{align*}
This implies the desired estimate \eqref{operator:eq-commutator-p2}.

\end{proof}

\section{A product estimate and its applications}\label{section:product}

In this section, we control the contributions of 
\begin{equation}\label{product:eq-motivation}
\partial_\alpha \big( A^\alpha_{\leq N} \parasim \chi_{\leq N} \big)  \qquad \text{and} \qquad \Big( A_{\leq N,\alpha} A^\alpha_{\leq N} - \scrm^{\hspace{-0.4ex}2}_{\leq N} \Big) \,  \chi_{\leq N},
\end{equation}
which involve the structured component 
$\chi_{\leq N} = (1-\Lin[ll][\leq N])^{-1} z$ (as introduced in Section \ref{section:ansatz}). 
Due to the negative regularity of $A^\alpha_{\leq N}$ and $\chi_{\leq N}$,  both terms in \eqref{product:eq-motivation} cannot be defined (even as space-time distributions) using purely deterministic arguments. Instead, our argument will rely heavily on the independence of the vector potential $A^\alpha$ and the linear stochastic object $z$. Our main estimates are collected in the following proposition. 

\begin{proposition}\label{product:prop-main} 
Let $C=C(b_0,b_+,\delta_1,\delta_2)\geq 1$ be sufficiently large and let $0<c\leq 1$ be sufficiently small. Then, for all $\lambda \geq 1$, there exists an event $E_\lambda$ which satisfies
\begin{equation}\label{product:eq-main-probability}
\mathbb{P}\big( E_\lambda \big) \geq 1 - c  \exp\big( - \lambda \big) 
\end{equation}
and such that, on this event, the following estimates are satisfied: 
\begin{enumerate}[label=(\roman*)]
\item (Cubic estimate) For all $T\geq 1$ and $M,N \in \dyadic$, it holds that 
\begin{align}
&\Big\|  \Big( A_{\leq N,\alpha} A^\alpha_{\leq N} - \scrm^{\hspace{-0.4ex}2}_{\leq N} \Big) \,  \chi_{\leq N} \Big\|_{L_t^2 H_x^{-\delta_2}([0,T]\times \T^2)}  \leq \exp\big( C(T\lambda)^C\big), \label{product:eq-cubic-1} \\
&\Big\|  \Big( A_{\leq M,\alpha} A^\alpha_{\leq M} - \scrm^{\hspace{-0.4ex}2}_{\leq M} \Big) \,  \chi_{\leq M}  
- \Big( A_{\leq N,\alpha} A^\alpha_{\leq N} - \scrm^{\hspace{-0.4ex}2}_{\leq N} \Big) \,  \chi_{\leq N} \Big\|_{L_t^2 H_x^{-\delta_2}([0,T]\times \T^2)} \label{product:eq-cubic-2} \\
 \leq&\, \exp\big( C(T\lambda)^C\big) \min(M,N)^{-\kappa}. \notag 
\end{align}
\item (High$\times$high-estimate) For all $T\geq 1$ and $M,N\in \dyadic$, it holds that 
\begin{align}
\big\| \Lin[sim][\leq N] \chi_{\leq N} \big\|_{X^{1/4-\delta_2,b_+}([0,T])} 
&\leq \exp\big( C (T\lambda)^C \big), \label{product:eq-hh-1} \\
\big\| \Lin[sim][\leq M] \chi_{\leq M} - \Lin[sim][\leq N] \chi_{\leq N} \big\|_{X^{1/4-\delta_2,b_+}([0,T])} &\leq \exp\big( C (T\lambda)^C \big) \min(M,N)^{-\kappa}.  \label{product:eq-hh-2}
\end{align}
\end{enumerate}
\end{proposition}

\begin{remark}
Proposition \ref{product:prop-main} can likely also be proven by combining Lemma \ref{operator:lem-resolvent} with the tensor merging estimates from \cite{DNY20}. However, we present a more elementary argument, which requires less of the tensor machinery. 
\end{remark}

The main ingredient in the proof of Proposition \ref{product:prop-main} is a product estimate. Due to its importance, the product estimate is captured in its own lemma.

\begin{lemma}[Product estimate]\label{product:lem-product}
Let $T\geq 1$, let $2\leq  q <\infty$, and let $z\colon \R_t \times \T_x^2 \rightarrow \R$ be as in \eqref{ansatz:eq-z}. Furthermore, let $f\colon [0,T] \times \T_x^2 \rightarrow \C$, let $\Mc \colon X^{0,b_0}([0,T]) \rightarrow X^{0,b_0}([0,T])$, and assume that both $f$ and $\Mc$ are probabilistically independent of $z$. Then, it holds for all $K,L \in \dyadic$ satisfying $K\sim L$ and all $p\geq 1$ that 
\begin{equation}
\begin{aligned}
&\E_z \Big[ \Big\| P_K f \, \Mc P_L z \Big\|_{L_t^q H_x^{-1}([0,T]\times \T_x^2)}^p \Big]^{1/p}  \\
\lesssim& \,  \sqrt{p} T^\theta K^{-1+\delta_1} \| f \|_{L_t^q L_x^2([0,T])} \| \Mc \|_{X^{0,b_0}([0,T])\rightarrow X^{0,b_0}([0,T])}. 
\end{aligned}
\end{equation}
Here, $\E_z$ denotes the expectation taken only with respect to the $\sigma$-Algebra $\sigma_z$. 
\end{lemma}

\begin{remark}
Due to the frequency-scales, time-scales, moments, and $X^{0,b}$-spaces, the statement of Lemma \ref{product:lem-product} is much more complicated than the main idea behind its proof. To illustrate the main idea, we state a toy-version of the same estimate. To this end, let $f\in \ell^2(\Z^2)$ be deterministic, let $\Mc \colon \ell^2(\Z^2) \rightarrow \ell^2(\Z^2)$ be deterministic, and let $(g_m)_{m\in \Z^2}$ be a sequence of independent, standard complex-valued Gaussian random variables. Then, it is easy to show that 
\begin{equation}
\sup_{n\in \Z^2} \E \Big[ \Big| \sum_{\substack{k,l\in \Z^2 \colon \\ k+l =n}} f_k \big( \Mc g \big)_l \Big|^2 \Big] \lesssim \big\| \Mc^\ast f \big\|_{\ell^2}^2 \lesssim \big\| \Mc^\ast \big\|_{\ell^2 \rightarrow \ell^2}^2 \big\| f \big\|_{\ell^2}^2 = \big\| \Mc \big\|_{\ell^2 \rightarrow \ell^2}^2 \big\| f \big\|_{\ell^2}^2.
\end{equation}
\end{remark}

\begin{proof}[Proof of Lemma \ref{product:lem-product}:]
In the following, all space-time norms are implicitly restricted to $[0,T]$. Using Gaussian hypercontractivity (Lemma \ref{prelim:lem-hypercontractivity}), it suffices to treat the case $p=q$. Furthermore, since the conclusion is trivial for $K,L\lesssim 1$, we can assume that $K,L \gg 1$.  \\ 

We now make use of a small trick which allows us to easily use Lemma \ref{prelim:lem-Xnub-to-ell2}. To this end, let $(Z_l)_{l\in \Z^2}$ be the Gaussian process from Subsection \ref{section:probability}. Furthermore, let $(\sigma_l)_{l\in \Z^2}$ be  a sequence of independent Bernoulli variables. Due to the rotation invariance of complex-valued Brownian motions, the two sequences $(\sigma_l)_{l\in \Z^2}$ and $(\sigma_l Z_l)_{l\in \Z^2}$ are still independent. We now note that the stochastic object $P_L z$ can be written as 
\begin{equation}\label{product:eq-product-p1}
\begin{aligned}
P_L z 
&= \sum_{\substack{l\in \Z^2 \colon \\ |l|\sim L}} 
e^{i\langle l, x \rangle} \int_0^t \mathrm{d}Z_s(l) \frac{\sin\big((t-s)|l|\big)}{|l|} \\ 
&= \sum_{\substack{l \in \Z^2 \colon \\ |l| \sim L}} 
\sigma_l e^{i\langle l,x \rangle} \int_0^t \mathrm{d}\big( \sigma_l Z_s(l)\big) \frac{\sin\big((t-s)|l|\big)}{|l|} \\
&=:  \sum_{\substack{l \in \Z^2 \colon \\ |l| \sim L}} \sigma_l e^{i\langle l, x \rangle} \varphi_l(t). 
\end{aligned}
\end{equation}
Similar as discussed above, we note that $(\sigma_l)_{l\in \Z^2}$ and $(\varphi_l)_{l\in \Z^2}$ are independent. In the following, we denote the corresponding expectations by $\E_\sigma$ and $\E_\varphi$ and the corresponding $L^p$-spaces by $L_\sigma^p$ and $L_\varphi^p$, respectively. We now define $\Mc^{(\varphi)}$ similar as in Lemma \ref{prelim:lem-Xnub-to-ell2}, i.e., for all $l \in \Z^2$ and $v\in \ell^2(\Z^2)$, we define 
\begin{equation*}
\Mc^{(\varphi)}_l := \Mc \big( e^{i \langle l, x\rangle} \varphi_l(t) \big) \quad \text{and} \qquad 
\Mc^{(\varphi)} v := \sum_{\substack{l \in \Z^2 \colon \\ |l| \sim L}} \Mc^{(\varphi)}_l v_l. 
\end{equation*}

After these preparations, we now proof the product estimate. To this end, let $N\in \dyadic$. Using Minkowski's integral inequality, \eqref{product:eq-product-p1}, and Khintchine's inequality, it holds that 
\begin{equation}\label{product:eq-product-p2} 
\begin{aligned}
\Big\| P_N \Big( P_K f(t) \Mc P_L z(t) \Big) \Big \|_{L_\varphi^q L_\sigma^q L_t^q L_x^2} 
\lesssim& \,  
\Big\| P_N \Big( P_K f(t) \Mc P_L z(t) \Big) \Big \|_{L_\varphi^q  L_t^q L_x^2 L_\sigma^q}  \\
\lesssim_q &\, \bigg\| \Big( \sum_{\substack{l \in \Z^2 \colon \\ |l|\sim L}} \big| P_N \big(P_K f(t) \Mc_l^{(\varphi)} \big) \big|^2 \Big)^{1/2} \bigg\|_{L_\varphi^q L_t^q L_x^2}. 
\end{aligned}
\end{equation}
Using Plancherell's theorem, it holds that 
\begin{equation}\label{product:eq-product-p3}
\begin{aligned}
 \bigg\| \Big( \sum_{\substack{l \in \Z^2 \colon \\ |l|\sim L}} \big| P_N \big(P_K f(t) \Mc_l^{(\varphi)} \big) \big|^2 \Big)^{1/2} \bigg\|_{L_x^2}^2
&=\sum_{\substack{l\in \Z^2 \colon \\ |l| \sim L }} \Big\|  P_N \Big( P_K f(t) \Mc^{(\varphi)}_l(t) \Big)\Big\|_{L_x^2}^2 \\
&\lesssim \sum_{\substack{n \in \Z^2 \colon \\ |n|\sim N}} \sum_{\substack{l\in \Z^2 \colon \\ |l| \sim L }}
\Big| \big\langle e^{i\langle n ,x \rangle} \overline{P_K f}(t), \Mc^{(\varphi)}_l(t) \big\rangle_{L^2_x} \Big|^2. 
\end{aligned}
\end{equation}
Using duality and the embedding $X^{0,b_0}\hookrightarrow L_t^\infty L_x^2$, it holds for all $n\in \Z^2$ that 
\begin{equation}\label{product:eq-product-p4}
\begin{aligned}
&\sum_{\substack{l\in \Z^2 \colon \\ |l| \sim L }}
\Big| \big\langle e^{i\langle n ,x \rangle} \overline{P_K f}(t), \Mc^{(\varphi)}_l(t) \big\rangle_{L^2_x} \Big|^2 \\
 \leq& \,  \big\| e^{i \langle n,x \rangle} \overline{P_K f}(t) \big\|_{L_x^2}^2 \big\| \Mc^{(\varphi)}(t) \big\|_{\ell^2\rightarrow L_x^2}^2 
 \leq \big\| f(t) \big\|_{L_x^2}^2 \big\| \Mc^{(\varphi)}\big\|_{\ell^2 \rightarrow L_t^\infty L_x^2}^2 
 \lesssim \big\| f(t) \big\|_{L_x^2}^2 \big\| \Mc^{(\varphi)}\big\|_{\ell^2 \rightarrow X^{0,b_0}}^2. 
\end{aligned}
\end{equation}
By combining \eqref{product:eq-product-p2}, \eqref{product:eq-product-p3}, and \eqref{product:eq-product-p4}, we obtain that 
\begin{equation}\label{product:eq-product-p5}
\E_\varphi \E_\sigma \Big[ \Big\| P_N \Big( P_K f(t) \, \Mc P_L z (t) \Big) \Big\|_{L_t^q L_x^2}^q \Big]^{1/q} \\
\lesssim N \big\| f \big\|_{L_t^q L_x^2} 
\E_\varphi \Big[ \big\| \Mc^{(\varphi)}\big\|_{\ell^2 \rightarrow X^{0,b_0}}^q \Big]^{1/q}. 
\end{equation}
By using Lemma \ref{prelim:lem-Xnub-to-ell2}, Lemma \ref{prelim:lem-Xnub-ito}, and the definition of $\varphi_l$ from \eqref{product:eq-product-p1}, it follows that 
\begin{align*}
\E_\varphi \Big[ \Big\| \Mc^{(\varphi)} \Big\|_{\ell^2 \rightarrow X^{0,b_0}}^q \Big]^{1/q} 
\lesssim& \, \big\| \Mc \big\|_{X^{0,b_0}\rightarrow X^{0,b_0}} \E_\varphi \Big[ \sup_{\substack{l \in \Z^2 \colon \\ |l|\sim L}} \big\| e^{i\langle l,x \rangle} \varphi_l(t) \big\|_{X^{0,b_0}}^q \Big]^{1/q} \\
\lesssim& \,  L^{-1+\delta_1/2} T^\theta \big\| \Mc \big\|_{X^{0,b_0}\rightarrow X^{0,b_0}} . 
\end{align*}
After inserting this back into \eqref{product:eq-product-p5}, we obtain the desired estimate. 
\end{proof} 

Equipped with Lemma \ref{product:lem-product}, we can now prove the main proposition of this section. 
\begin{proof}[Proof of Proposition \ref{product:prop-main}:]
Let $\lambda\geq 1$ be arbitrary but fixed. In the following argument, we refer to events satisfying the probability estimate \eqref{product:eq-main-probability} simply as events with sufficiently high probability. We also note that, after possibly adjusting the choice of $0<c\ll 1$, the intersection of finitely many events with sufficiently high probability still has sufficiently high probability.
Furthermore, we let $T\geq 1$ be arbitrary but fixed and choose all of our events as not depending on $T$.  We also implicitly restrict all space-time norms to $[0,T]\times \T^2$.   \\ 

We now split the proof into three substeps. In the first step, we prove a general frequency-localized estimate. In the second and third step, we then use the general frequency-localized estimate to show the cubic and high$\times$high-estimate, respectively. \\ 

\emph{Step 1: A general frequency-localized estimate.} 
Let $q\geq 2$ and let $f\colon \R \times \T^2 \rightarrow \C$ be measurable w.r.t. $\sigma_A$. Then, there exists an event $E_{\lambda,f}$ with sufficiently high probability such that, on this event, the estimate  
\begin{equation}\label{product:eq-general-localized}
 \Big\|  P_K f P_L \chi_{\leq N}  \Big\|_{L_t^q H_x^{-\delta_2}([0,T]\times \T_x^2)} 
 \lesssim \max(K,L)^{-\delta_2/10} \big\|  P_K f \big\|_{L_t^\infty \Cs_x^{-\delta_1}} \exp\Big( C (T\lambda)^C \Big) 
\end{equation}
is satisfied for all $K,L\in \dyadic$. In order to prove \eqref{product:eq-general-localized}, we first estimate
\begin{equation*}
\Big\|  P_K f P_L \chi_{\leq N}  \Big\|_{L_t^q H_x^{-\delta_2}([0,T]\times \T_x^2)} 
 \leq \sum_{M\in \dyadic} M^{-\delta_2}  \Big\| P_M \Big( P_K f P_L \chi_{\leq N}  \Big) \Big\|_{L_t^q L_x^2([0,T]\times \T_x^2)} 
\end{equation*}
We now use different estimates depending on the relative sizes of $K$, $L$, and $M$.  \\ 

\emph{Step 1.(a): Case $M\geq \max(K,L)^{1/5}$.} 
Using Lemma \ref{prelim:lem-Xnub} and Lemma \ref{operator:lem-resolvent}, it holds on an event with sufficiently high probability that 
\begin{align*}
 M^{-\delta_2}  \Big\| P_M \Big( P_K f P_L \chi_{\leq N} \Big) \Big\|_{L_t^q L_x^2([0,T]\times \T_x^2)}  
&\lesssim T  M^{-\delta_2}  \Big\| P_K f P_L \chi_{\leq N} \Big\|_{L_t^\infty L_x^2([0,T]\times \T_x^2)}      \\
&\lesssim T  M^{-\delta_2}  \big\| P_K f \big\|_{L_t^\infty L_x^\infty} \big\| P_L \chi_{\leq N} \big\|_{L_t^\infty L_x^2} \\
&\lesssim T  M^{-\delta_2}  \big\| P_K f \big\|_{L_t^\infty L_x^\infty} 
\big\| P_L \chi_{\leq N} \big\|_{X^{0,b_0}} \\
&\lesssim  T  M^{-\delta_2}  K^{\delta_1} L^{\delta_1} \big\| P_K f \big\|_{L_t^\infty \Cs_x^{\delta_1}} \exp\Big( C (T\lambda)^C \Big). 
\end{align*}
Since $M\geq \max(K,L)$ and $\delta_1\ll \delta_2$, this yields an acceptable contribution.\\ 

\emph{Step 1.(b): Case $M<\max(K,L)^{1/5}$.} In this case, it automatically holds that $K\sim L$.  
 We recall that $\chi_{\leq N}=(1+\Lin[ll][\leq N])^{-1} z$ and decompose
\begin{align}
P_M \Big( P_K f P_L \chi_{\leq N} \Big) 
&= P_M \Big( P_K f \, (1+\Lin[ll][\leq N])^{-1} P_L z \Big) \label{product:eq-general-p1} \\
&+P_M \Big( P_K f \, \Big[ P_L, \big( 1+ \Lin[ll][\leq N]\big)^{-1} \Big] \, z \Big) \label{product:eq-general-p2} 
\end{align}
We now estimate \eqref{product:eq-general-p1} and \eqref{product:eq-general-p2} separately. \\

\emph{Estimate of \eqref{product:eq-general-p1}:}
Using Lemma \ref{product:lem-product}, we obtain on an event with sufficiently high probability that
\begin{align*}
\Big\| P_M \Big( P_K f \, (1+\Lin[ll][\leq N])^{-1} P_L z \Big) \Big\|_{L_t^q H_x^{-\delta_2}}
& \lesssim M^{1-\delta_2} \, \Big\|  P_K f  \,  (1+\Lin[ll][\leq N])^{-1} P_L z  \Big\|_{L_t^q H_x^{-1}} \\
& \lesssim M^{1-\delta_2} L^{-1+\delta_1} \Big\|  P_K f \Big\|_{L_t^q L_x^2} \Big\| (1+\Lin[ll][\leq N])^{-1} \Big\|_{X^{0,b_0}\rightarrow X^{0,b_0}}. 
\end{align*}
Using Lemma \ref{operator:lem-resolvent}, it follows that 
\begin{align*}
&M^{1-\delta_2} L^{-1+\delta_1} \Big\|  P_K f \Big\|_{L_t^q L_x^2} \Big\| (1+\Lin[ll][\leq N])^{-1} \Big\|_{X^{0,b_0}\rightarrow X^{0,b_0}}\\  
\lesssim& \,  M^{1-\delta_2} K^{\delta_1} L^{-1+\delta_1} \Big\|  P_K f \Big\|_{L_t^\infty \Cs_x^{-\delta_1}} \exp\big( C (T\lambda)^C \big). 
\end{align*}
Since $K\sim L$ and $M < \max(K,L)^{1/5}$, this clearly yields an acceptable contribution. \\

\emph{Estimate of \eqref{product:eq-general-p2}:} 
Using  Lemma \ref{prelim:lem-Xnub}, Corollary \ref{vector:cor-regularity-z}, and Lemma \ref{operator:lem-commutator}, we obtain on an event with sufficiently high probability that
\begin{align*}
\Big\| P_M \Big( f \, \Big[ P_L, \big( 1+ \Lin[ll][\leq N]\big)^{-1} \Big] \, z \Big) \Big\|_{L_t^q H_x^{-\delta_2}} 
\lesssim&\,  
\Big\| P_K f \Big\|_{L_t^\infty \Cs_x^{\delta_2}} \Big\| \Big[ P_L, \big( 1+ \Lin[ll][\leq N]\big)^{-1} \Big] \, z \Big\|_{L_t^\infty H_x^{-\delta_1}} \\
\lesssim&\,  K^{2\delta_2} L^{-1/4+\delta_1} \exp\big( C (T\lambda)^C\big) \Big\| P_K f \Big\|_{L_t^\infty \Cs_x^{-\delta_2}} 
\big\| z \big\|_{X^{-\delta_1,b_0}} \\
\lesssim&\, K^{2\delta_2} L^{-1/4+\delta_1} \exp\big( C (T\lambda)^C\big). 
\end{align*}
Since $K\sim L \gtrsim M$ and $\delta_2 \ll 1$, this clearly yields an acceptable contribution. 
This completes the proof of the general frequency-localized estimate \eqref{product:eq-general-localized}.\\

\emph{Step 2: Proof of the cubic estimate.} We only prove \eqref{product:eq-cubic-1}, since \eqref{product:eq-cubic-2} follows from minor modifications. Using \eqref{product:eq-general-localized}, we obtain on an event with sufficiently high probability that 
\begin{equation*}
\begin{aligned}
&\Big\|  \Big( A_{\leq N,\alpha} A^\alpha_{\leq N} - \scrm^{\hspace{-0.4ex}2}_{\leq N} \Big) \,  \chi_{\leq N} \Big\|_{L_t^2 H_x^{-\delta_2}} \\
\lesssim& \,  \Big\|  A_{\leq N,\alpha} A^\alpha_{\leq N} - \scrm^{\hspace{-0.4ex}2}_{\leq N} \Big\|_{L_t^\infty \Cs_x^{-\delta_1}} \exp\Big( C (T\lambda)^C \Big).  
\end{aligned}
\end{equation*}
Together with Lemma \ref{vector:lem-quadratic}, this yields the desired estimate. \\ 

\emph{Step 3: Proof of the high$\times$high-estimate.} 
We only prove \eqref{product:eq-hh-1}, since \eqref{product:eq-hh-2} follows from minor modifications. Using \eqref{product:eq-general-localized} and Corollary \ref{vector:cor-regularity}, we obtain on an event with sufficiently high probability that, for all $K,M,N\in \dyadic$, 
\begin{equation*}
\big\| P_M \Lin[sim][K] \chi_{\leq N} \big\|_{X^{1/4-\delta_2,b_+}}
\lesssim T \max_{\alpha=0,1,2} \big\| P_M \big( A_K^\alpha \parasim \chi_{\leq N} \big) \big\|_{L_t^2 H_x^{1/4-\delta_2}} \lesssim M^{1/4}K^{-\delta_2/10} \exp\big( C (T\lambda)^C \big). 
\end{equation*}
Using Proposition \ref{operator:prop-main}.\ref{operator:item-high-high} and Lemma \ref{operator:lem-resolvent}, we also obtain on an event with sufficiently high probability that, for all $K,M,N\in \dyadic$, 
\begin{align*}
\big\| P_M \Lin[sim][K] \chi_{\leq N} \big\|_{X^{1/4-\delta_2,b_+}} 
&\lesssim M^{-\delta_2} \big\| P_M \Lin[sim][K] \chi_{\leq N} \big\|_{X^{1/4,b_+}} \\ 
&\lesssim T^\theta M^{-\delta_2} 
\big\|\widetilde{P}_K \chi_{\leq N} \big\|_{X^{\delta_1,b_0}}  \\
&\lesssim M^{-\delta_2} K^{2\delta_1} \exp\big( C (T\lambda)^C\big). 
\end{align*}
Since 
\begin{equation*}
\min\Big( M^{1/4}K^{-\delta_2/10} , M^{-\delta_2} K^{2\delta_1} \Big) 
\lesssim \Big( M^{1/4}K^{-\delta_2/10}  \Big)^{\delta_2} \Big( M^{-\delta_2} K^{2\delta_1} \Big)^{1-\delta_2}  \lesssim K^{-2\kappa}, 
\end{equation*}
the combined estimate yields \eqref{product:eq-hh-1}. 
\end{proof} 

 We also obtain the following $\Cs_t^0$-estimate, which essentially corresponds to the endpoint $q=\infty$ in \eqref{product:eq-general-localized} from the proof of Proposition \ref{product:prop-main}. The $\Cs_t^0$-estimate will be needed in the proof of Proposition \ref{proof:prop-phi} below. 

\begin{corollary}[The product of $A^0$ and $\chi$]\label{product:cor-Ct}
 Let $C=C(b_-,b_0,b_+,\delta_0,\delta_1,\delta_2)$ be sufficiently large and $0<c\ll 1$ be sufficiently small. Then, for all $\lambda \geq 1$, there exists an event $E_{\lambda}$ which satisfies 
\begin{equation}\label{product:eq-Ct-probability}
\mathbb{P}\big( E_{\lambda} \big) \geq 1 - c \exp(-\lambda)
\end{equation}
and such that, on this event, the following estimates are satisfied: For all $M,N \in \dyadic$ and $T\geq 1$, it holds that 
\begin{align}
\big\| A^0_{\leq N}(t) \chi_{\leq N}(t) \big\|_{\Cs_t^0 H_x^{-\delta_2}([0,T]\times \T^2)} &\leq \exp\big( C(T\lambda)^C \big) \label{product:eq-Ct-1}, \\ 
\big\| A^0_{\leq M}(t) \chi_{\leq M}(t) - A^0_{\leq N}(t) \chi_{\leq N}(t) \big\|_{\Cs_t^0 H_x^{-\delta_2}([0,T]\times \T^2)} &\leq \exp\big( C(T\lambda)^C \big) \min(M,N)^{-\kappa}. \label{product:eq-Ct-2} 
\end{align}
\end{corollary}

\begin{proof}
We only prove the first estimate \eqref{product:eq-Ct-1}, since \eqref{product:eq-Ct-2} follows from minor modifications. 
Let $1\leq q=q(\delta_2,b_0)<\infty$ remain to be chosen. Using \eqref{product:eq-general-localized} from the proof of Proposition \ref{product:prop-main} and Corollary \ref{vector:cor-regularity}, we obtain on an event with sufficiently high probability that, for all $K,L,N\in \dyadic$, 
\begin{equation}\label{product:eq-Ct-p1}
\big\| A^0_K(t) P_L \chi_{\leq N}(t) \big\|_{L_t^q H_x^{-\delta_2}} 
\leq \max(K,L)^{-\delta_2/10} \exp\big( C(T\lambda)^C \big). 
\end{equation}

We now set $\gamma:= \frac{1}{2} (b_0-1/2)$. Using the Hölder estimate from Lemma \ref{prelim:lem-Xnub} and Corollary \ref{vector:cor-regularity}, we obtain on an event with sufficiently high probability that, for all $K,L,N\in \dyadic$, 
\begin{equation}\label{product:eq-Ct-p2}
\begin{aligned}
\big\| A^0_K(t) P_L \chi_{\leq N}(t) \big\|_{\Cs_t^\gamma H_x^{-\delta_2}} &\lesssim \big\| A^0_K(t) \big\|_{\Cs_t^\gamma \Cs_x^{\delta_2}}
\big\| \chi_{\leq N} \big\|_{\Cs_t^\gamma H_x^{-\delta_1}} 
\lesssim \big\| A^0_K(t) \big\|_{\Cs_t^\gamma \Cs_x^{\delta_2}}
\big\| \chi_{\leq N} \big\|_{X^{-\delta_1,b_0}} \\
&\leq K^{2\delta_2} \exp\big( C (T\lambda)^C \big). 
\end{aligned}
\end{equation}
As long as $q=q(\delta_2,b_0)$ is sufficiently large, \eqref{product:eq-Ct-p1} and \eqref{product:eq-Ct-p2} imply the desired estimate \eqref{product:eq-Ct-1}. 
\end{proof}

\section{The smooth remainder and proof of the main theorem} \label{section:proof-thm}

The goal of this section is to prove our main theorem (Theorem \ref{intro:thm-main}). To this end, we first prove two different estimates. In Proposition \ref{proof:prop-psi}, we control the smooth remainder $\psi_{\leq N}$ (from Section \ref{section:ansatz}). The corresponding estimate essentially follows from our estimates in Section \ref{section:vector}, Section \ref{section:operator}, and Section \ref{section:product}, and the argument therefore primarily consists of references to earlier estimates. Together with Lemma \ref{operator:lem-resolvent}, Proposition \ref{proof:prop-psi} yields convergence of $\phi_{\leq N}$ in $X^{-\delta,b_0}$. \\

In Proposition \ref{proof:prop-phi}, we obtain the convergence of $\phi_{\leq N}[t]=(\phi_{\leq N}(t),\partial_t \phi_{\leq N}(t))$ in $\Cs_t^0 \mathscr{H}_x^{-\delta}$. Due to the embedding $X^{-\delta,b_0}\hookrightarrow \Cs_t^0 H_x^{-\delta}$, the convergence of $\phi_{\leq N}(t)$ directly follows from our previous results. However,  the convergence of the time-derivatives $\partial_t \phi_{\leq N}(t)$ has to be shown separately. This is primarily because of the contribution from $\partial_t (A_{\leq N}^0 \phi_{\leq N})$, since $\partial_t (A_{\leq N}^0 \phi_{\leq N})$ cannot be placed in $L_t^2 H_x^{-1-\delta}$. 
At the end of this section, we then prove Theorem \ref{intro:thm-main}, which follows almost directly from Corollary \ref{vector:cor-regularity} and Proposition \ref{proof:prop-phi}. 

\begin{proposition}[\protect{The smooth remainder $\psi_{\leq N}$}]\label{proof:prop-psi}
 Let $R\geq 1$ and let $(\phi_0,\phi_1) \in \mathscr{H}^{1/4}_x(\T^2)$ satisfy $\| (\phi_0,\phi_1)\|_{\mathscr{H}_x^{1/4}}\leq R$.
Furthermore, let $C=C(b_-,b_0,b_+,\delta_0,\delta_1,\delta_2)$ be sufficiently large and $0<c\ll 1$ be sufficiently small. Then, for all $\lambda \geq 1$, there exists an event $E_{\lambda}$ which satisfies 
\begin{equation}\label{proof:eq-psi-probability}
\mathbb{P}\big( E_{\lambda} \big) \geq 1 - c \exp(-\lambda)
\end{equation}
and such that, on this event, the following estimates are satisfied: For all $M,N \in \dyadic$ and $T\geq 1$, it holds that 
\begin{align}
\big\| \psi_{\leq N} \big\|_{X^{1/4-\delta_2,b_0}([0,T])} &\leq R \exp\big( C (T\lambda)^C \big), \label{proof:eq-psi-1} \\
\big\| \psi_{\leq M} - \psi_{\leq N} \big\|_{X^{1/4-\delta_2,b_0}([0,T])} &\leq R \exp\big( C (T\lambda)^C \big) \min(M,N)^{-\kappa}. \label{proof:eq-psi-2} 
\end{align}
\end{proposition}

\begin{proof}
We only prove \eqref{proof:eq-psi-1}, since \eqref{proof:eq-psi-2} follows from minor modifications. As in the proof of Proposition \ref{product:prop-main}, we let $\lambda\geq 1$ be arbitrary but fixed and refer to events satisfying \eqref{proof:eq-psi-probability} as having sufficiently high probability. We also let $T\geq 1$ be arbitrary but fixed and implicitly restrict all space-time norms to $[0,T] \times \T^2$. \\ 

We let $\tau = \tau(T,\lambda)\in (0,1)$ remain to be chosen and define
\begin{equation*}
D_j := \big\| \psi_{\leq N} \big\|_{X^{1/4-\delta_2,b_0}([0,j\tau])}
\end{equation*}
for all $j\in \Nzero$. If $j\in \Nzero$ satisfies $(j+1)\tau\leq T$, it follows from Lemma \ref{prelim:lem-Xnub} that 
\begin{equation}\label{proof:eq-psi-p0}
\begin{aligned}
D_{j+1} &=  \big\| \psi_{\leq N} \big\|_{X^{1/4-\delta_2,b_0}([0,(j+1)\tau])} \\
&\lesssim D_j + \tau^{b_+-b_0} \big\| \psi_{\leq N} \big\|_{X^{1/4-\delta_2,b_+}([0,(j+1)\tau])}. 
\end{aligned}
\end{equation}
Using the evolution equation for $\psi_{\leq N}$, i.e., \eqref{ansatz:eq-psi-1}-\eqref{ansatz:eq-psi-2}, it follows that \begin{align}
&\hspace{3ex}\big\| \psi_{\leq N}  \big\|_{X^{1/4-\delta_2,b_+}([0,(j+1)\tau])} \notag  \\ 
&\lesssim  
\big\|  \Lin[ll][\leq N] \psi_{\leq N}   \big\|_{X^{1/4-\delta_2,b_+}([0,(j+1)\tau])} \label{proof:eq-psi-p1} \\ 
&+
\big\|    \Lin[sim][\leq N] \big( \chi_{\leq N} + \psi_{\leq N} \big) \big\|_{X^{1/4-\delta_2,b_+}([0,(j+1)\tau])} \label{proof:eq-psi-p2} \\ 
&+
\big\|      \Lin[gg][\leq N] \big( \chi_{\leq N} + \psi_{\leq N} \big) \big\|_{X^{1/4-\delta_2,b_+}([0,(j+1)\tau])} \label{proof:eq-psi-p3} \\ 
&+
\Big\|    \Duh \Big[ \big( A_{\leq N,\alpha} A^\alpha_{\leq N} - \scrm^{\hspace{-0.4ex}2}_{\leq N} \big) \,  \big( \chi_{\leq N} +\psi_{\leq N} \big)  \Big]  \Big\|_{X^{1/4-\delta_2,b_+}([0,(j+1)\tau])} \label{proof:eq-psi-p4} \\ 
&+
\big\| \mathcal{W}(t) (\phi_0,\phi_1)    \big\|_{X^{1/4-\delta_2,b_+}([0,(j+1)\tau])} \label{proof:eq-psi-p5}.
\end{align}
We now let $C_0=C_0(b_0,b_+,\delta_1,\delta_2)$ be sufficiently large. Then, we claim that, on an event with sufficiently high probability, it holds that 
\begin{equation}\label{proof:eq-psi-p6}
\begin{aligned}
&\eqref{proof:eq-psi-p1} + \eqref{proof:eq-psi-p2} + \eqref{proof:eq-psi-p3} + \eqref{proof:eq-psi-p4} + \eqref{proof:eq-psi-p5} \\ 
\leq& \, C_0 (T\lambda)^{C_0} \Big( \big\| \psi_{\leq N}  \big\|_{X^{1/4-\delta_2,b_0}([0,(j+1)\tau])} + R \Big) + \exp\big( C_0 (T\lambda)^{C_0} \big). 
\end{aligned}
\end{equation}
The estimates of the five individual terms in \eqref{proof:eq-psi-p6} can now be obtained as follows: 
\begin{itemize}
    \item[$\bullet$] For \eqref{proof:eq-psi-p1}, this estimate follows from Proposition \ref{operator:prop-main}.\ref{operator:item-low-high}. 
    \item[$\bullet$]  For \eqref{proof:eq-psi-p2}, this estimate follows from Proposition \ref{product:prop-main} (for the $\chi_{\leq N}$-term) and Proposition \ref{operator:prop-main}.\ref{operator:item-high-high} (for the $\psi_{\leq N}$-term).
    \item[$\bullet$] For \eqref{proof:eq-psi-p3}, this estimate follows from Proposition \ref{operator:prop-main}.\ref{operator:item-high-low}.
    \item[$\bullet$] For \eqref{proof:eq-psi-p4}, this estimate follows from Proposition \ref{product:prop-main} (for the $\chi_{\leq N}$-term) and just Corollary \ref{vector:cor-regularity} (for the $\psi_{\leq N}$-term). 
    \item[$\bullet$] Finally, for \eqref{proof:eq-psi-p5}, this estimate follows from Lemma \ref{prelim:lem-Xnub}. 
\end{itemize}
By combining \eqref{proof:eq-psi-p0} and \eqref{proof:eq-psi-p6}, it now follows that 
\begin{equation*}
D_{j+1} \leq C_0 \Big( D_j + (T\lambda)^{C_0} R + \exp\big( C_0 (T\lambda)^{C_0}\big) \Big)  + \tau^{b_+-b_0} C_0 (T\lambda)^{C_0} D_{j+1}. 
\end{equation*}
Similar as in the proof of Lemma \ref{operator:lem-resolvent}, the desired estimate now follows by first choosing $\tau=\tau(\lambda,T)$ and then iterating. 
\end{proof}

\begin{proposition}[Control of $\phi_{\leq N}$]\label{proof:prop-phi}
Let $R\geq 1$ and let $(\phi_0,\phi_1) \in \mathscr{H}^{1/4}_x(\T^2)$ satisfy $\| (\phi_0,\phi_1)\|_{\mathscr{H}_x^{1/4}}\leq R$.
Furthermore, let $C=C(b_-,b_0,b_+,\delta_0,\delta_1,\delta_2)$ be sufficiently large and $0<c\ll 1$ be sufficiently small. Then, for all $\lambda \geq 1$, there exists an event $E_{\lambda}$ which satisfies 
\begin{equation}\label{proof:eq-phi-probability}
\mathbb{P}\big( E_{\lambda} \big) \geq 1 - c \exp(-\lambda)
\end{equation}
and such that, on this event, the following estimates are satisfied: For all $M,N \in \dyadic$ and $T\geq 1$, it holds that 
\begin{align}
\big\| \phi_{\leq N}[t] \big\|_{\Cs_t^0 \mathscr{H}_x^{-\delta}([0,T]\times \T^2)} &\leq R \exp\big( C (T\lambda)^C \big), \label{proof:eq-phi-1} \\
\big\| \phi_{\leq M}[t] - \phi_{\leq N}[t] \big\|_{\Cs_t^0 \mathscr{H}_x^{-\delta}([0,T]\times \T^2)} &\leq R \exp\big( C (T\lambda)^C \big) \min(M,N)^{-\kappa}. \label{proof:eq-phi-2} 
\end{align}
\end{proposition}

\begin{proof}[Proof of Proposition \ref{proof:prop-phi}:]
We only prove \eqref{proof:eq-phi-1}, since \eqref{proof:eq-phi-2} follows from minor modifications. 
Using Proposition \ref{ansatz:prop-ansatz}, it follows that $\phi_{\leq N}=\chi_{\leq N}+\psi_{\leq N}$. Then, the bound on $\phi_{\leq N}$ in $\Cs_t^0 H_x^{-\delta}$ follows directly from the embedding in Lemma \ref{prelim:lem-Xnub}, Lemma \ref{operator:lem-resolvent}, and Proposition \ref{proof:prop-psi}. Thus, it remains to control the time-derivative. \\

In the following proof, all space-time norms are restricted to $[0,T]\times \T^2$. From \eqref{ansatz:eq-phiN-a}, it follows that
\begin{align}
\partial_t \phi_{\leq N} 
&= - 2 i \partial_t \Duh \Big[ \partial_0 \big( A_{\leq N}^0 \phi_{\leq N} \big) \Big] \label{proof:eq-time-p1} \\
&- 2 i \partial_t \Duh \Big[ \partial_a \big( A_{\leq N}^a \phi_{\leq N} \big) \Big]\label{proof:eq-time-p2} \\
&+ \partial_t \Duh \Big[ \Big( A_{\leq N,\alpha} A^\alpha_{\leq N} - \scrm^{\hspace{-0.4ex}2}_{\leq N} \Big) \phi_{\leq N} \Big] \label{proof:eq-time-p3} \\
&+ \partial_t z + \partial_t \mathcal{W}(t) (\phi_0,\phi_1). \label{proof:eq-time-p4} 
\end{align}
We first address the contribution of \eqref{proof:eq-time-p1}, which is the main term. Using $A^0_{\leq N}(0)=0$ and integration by parts, it follows that 
\begin{equation}\label{proof:eq-time-p9}
\begin{aligned}
\partial_t \Duh \Big[ \partial_0 \big( A_{\leq N}^0 \phi_{\leq N} \big) \Big] 
=& \partial_t \int_0^t \frac{\sin \big( (t-t^\prime) |\nabla|\big)}{|\nabla|} \partial_{t^\prime}
\big(A^0_{\leq N}(t^\prime) \phi_{\leq N}(t^\prime) \big) \\ 
=& \int_0^t \cos \big( (t-t^\prime) |\nabla| \big) \partial_{t^\prime}
\big(A^0_{\leq N}(t^\prime) \phi_{\leq N}(t^\prime) \big)  \\
=& A^0_{\leq N}(t) \phi_{\leq N}(t) - |\nabla| \int_0^t \dt^\prime \sin \big( (t-t^\prime) |\nabla| \big) \big( A^0_{\leq N}(t^\prime) \phi_{\leq N}(t^\prime) \big). 
\end{aligned}
\end{equation}
We now estimate 
\begin{equation}\label{proof:eq-time-p7} 
\big\| A_{\leq N}^0(t) \phi_{\leq N}(t) \big\|_{\Cs_t^0 H_x^{-\delta}} 
\lesssim \big\| A_{\leq N}^0(t) \chi_{\leq N}(t) \big\|_{\Cs_t^0 H_x^{-\delta}} 
 + \big\| A_{\leq N}^0(t) \psi_{\leq N}(t) \big\|_{\Cs_t^0 H_x^{-\delta}}. 
\end{equation}
The first summand in \eqref{proof:eq-time-p7} is the subject of Corollary \ref{product:cor-Ct} and the second summand in \eqref{proof:eq-time-p7} can be estimated using Corollary \ref{vector:cor-regularity} and Proposition \ref{proof:prop-psi}. We further estimate 
\begin{equation}\label{proof:eq-time-p8}
\begin{aligned}
&\Big\| |\nabla| \int_0^t \dt^\prime \sin \big( (t-t^\prime) |\nabla| \big) \big( A^0_{\leq N}(t^\prime) \phi_{\leq N}(t^\prime) \big) \Big\|_{C_t^0 H_x^{-1-\delta}} \\
\lesssim& \, \big\| A^0_{\leq N}  \phi_{\leq N} \big\|_{L_t^1 H_x^{-\delta}} 
\lesssim \, T  \big\|  A^0_{\leq N}  \phi_{\leq N} \big\|_{L_t^\infty H_x^{-\delta}}. 
\end{aligned}
\end{equation}
Thus, \eqref{proof:eq-time-p8} is controlled by \eqref{proof:eq-time-p7}. This completes the estimate of \eqref{proof:eq-time-p9} and therefore the estimate of the main term \eqref{proof:eq-time-p1}. \\

It remains to control the contributions of \eqref{proof:eq-time-p2}, \eqref{proof:eq-time-p3}, and \eqref{proof:eq-time-p4}.  Using Lemma \ref{prelim:lem-Xnub} and our Ansatz, it follows that 
\begin{align}
&\big\| \eqref{proof:eq-time-p2} \big\|_{\Cs_t^0 H_x^{-1-\delta}} 
+ \big\| \eqref{proof:eq-time-p3} \big\|_{\Cs_t^0 H_x^{-1-\delta}}  \notag \\
\lesssim& \, 
T^2 \Big\|  \partial_a \big( A_{\leq N}^a \phi_{\leq N} \big) \Big\|_{L_t^2 H_x^{-1-\delta}}
+ T^2 \Big\|  \Big( A_{\leq N,\alpha} A^\alpha_{\leq N} - \scrm^{\hspace{-0.4ex}2}_{\leq N} \Big) \phi_{\leq N}  \Big\|_{L_t^2 H_x^{-1-\delta}} \notag \\
\lesssim& \, 
T^2 \max_{a=1,2} \Big\|   A_{\leq N}^a \chi_{\leq N}  \Big\|_{L_t^2 H_x^{-\delta}}
+ T^2 \Big\|  \Big( A_{\leq N,\alpha} A^\alpha_{\leq N} - \scrm^{\hspace{-0.4ex}2}_{\leq N} \Big) \chi_{\leq N}  \Big\|_{L_t^2 H_x^{-1-\delta}} \label{proof:eq-time-p5} \\
+& \, 
T^2 \max_{a=1,2} \Big\|   A_{\leq N}^a \psi_{\leq N}  \Big\|_{L_t^2 H_x^{-\delta}}
+ T^2 \Big\|  \Big( A_{\leq N,\alpha} A^\alpha_{\leq N} - \scrm^{\hspace{-0.4ex}2}_{\leq N} \Big) \psi_{\leq N}  \Big\|_{L_t^2 H_x^{-1-\delta}}. \label{proof:eq-time-p6}
\end{align}
The terms in \eqref{proof:eq-time-p5} can be bounded using Proposition \ref{product:prop-main}. The terms in \eqref{proof:eq-time-p6} can be bounded using Corollary \ref{vector:cor-regularity}, Lemma \ref{vector:lem-quadratic}, and Proposition \ref{proof:prop-psi}. 

Finally, the contribution of \eqref{proof:eq-time-p4} can bounded directly using Corollary \ref{vector:cor-regularity-z} and Lemma \ref{prelim:lem-Xnub}. 
\end{proof}

Equipped with Proposition \ref{proof:prop-phi}, we can  now prove our main theorem. 

\begin{proof}[Proof of Theorem \ref{intro:thm-main}:]
By time-reversal symmetry, it suffices to prove almost-sure convergence on $[0,T]\times \T^2$ for all $T\geq 1$. The almost-sure convergence of $A_{\leq N}[t]$ follows from Corollary \ref{vector:cor-regularity}. The almost-sure convergence of $\phi_{\leq N}[t]$ follows from Proposition \ref{proof:prop-phi}.
\end{proof}

\section{Failure of a probabilistic null-form estimate} \label{section:null-form}

In this section, we prove the failure of a probabilistic null-form estimate, i.e., Theorem \ref{intro:thm-failure}. To this end, we first state the following lemma regarding the space-time covariances of $z$.

\begin{lemma}[Space-time covariances of $z$]\label{failure:lem-covariance}
For all $k, l \in \Z^2 \backslash \{0\}$ and $t,s\geq 0$, it holds that 
\begin{align*}
&\E \Big[ \widehat{z}(t,k) \overline{\widehat{z}(s,l)}\Big]  \\
=& \,  \mathbf{1} \big\{ k = l \big\} \bigg( \frac{1}{2|k|^2} (t\wedge s) \cos \big( (t-s) |k| \big) + \frac{1}{4|k|^3} \Big( \sin\big( |t-s| |k| \big) - \sin \big( (t+s) |k| \big) \Big) \bigg). 
\end{align*}
\end{lemma}
\begin{proof}
The computation is essentially as in Lemma \ref{vector:lem-covariance} and we therefore omit the details. 
\end{proof}

Equipped with Lemma \ref{failure:lem-covariance}, we now prove Theorem \ref{intro:thm-failure}. 

\begin{proof}[Proof of Theorem \ref{intro:thm-failure}:] 
We first reduce the first claim \eqref{intro:eq-failure-1} to the second claim \eqref{intro:eq-failure-2}. To this end, assume that $\psi\in \Cs^\infty_c((0,\infty) \times \T_x^2 \rightarrow [-1,1])$ satisfies \eqref{intro:eq-failure-2}. Since the vector-field $(-\partial_2 \psi, \, \partial_1 \psi)$ is divergence-free, it holds that 
\begin{equation*}
\int_{\R} \dt \int_{\T^2} \dx \, 
\begin{pmatrix} - \partial_2 \psi \\[-1ex] \partial_1 \psi \end{pmatrix}
\cdot 
\Im \big( P_{\leq N} z \, \overline{\nabla P_{\leq N} z} \big) 
= \int_{\R} \dt \int_{\T^2} \dx \, 
\begin{pmatrix} - \partial_2 \psi \\[-1ex] \partial_1 \psi \end{pmatrix}
\cdot 
\Leray \Im \big( P_{\leq N} z \,  \overline{\nabla P_{\leq N} z} \big).
\end{equation*}
By inserting the identity \eqref{intro:eq-null-identity}, then using that $\partial_1 \psi$ and $\partial_2 \psi$ are mean-zero to eliminate the contribution of the constant term in \eqref{intro:eq-null-identity}, and finally using $Q_{11}=Q_{22}=0$ and $Q_{12}=-Q_{21}$, it follows that 
\begin{align}
&\int_{\R} \dt \int_{\T^2} \dx \, 
\begin{pmatrix} - \partial_2 \psi \\[-1ex] \partial_1 \psi \end{pmatrix}
\cdot 
\Leray \Im \big( P_{\leq N} z \,  \overline{\nabla P_{\leq N} z} \big) \notag \\
=&\, -\int_{\R} \dt \int_{\T^2} \dx \, 
\begin{pmatrix} - \partial_2 \psi \\[-1ex] \partial_1 \psi \end{pmatrix}
\cdot 
\begin{pmatrix} 
\Delta^{-1} \partial_2 \Im Q_{12} \big( P_{\leq N} z, \overline{P_{\leq N} z} \big)  \\[-1ex]
\Delta^{-1} \partial_1 \Im Q_{21} \big( P_{\leq N} z, \overline{P_{\leq N} z} \big) 
\end{pmatrix}  \notag \\
=& \,  -\int_{\R} \dt \int_{\T^2} \dx \, 
\Big( - \Delta^{-1}\partial_2^2 \psi \cdot  \Im Q_{12}\big( P_{\leq N} z, \overline{P_{\leq N} z} \big)    
+ \Delta^{-1} \partial_1^2 \psi \cdot \Im Q_{21} \big( P_{\leq N} z, \overline{P_{\leq N} z} \big)  \Big) \notag \\ 
=& \, \int_{\R} \dt \int_{\T^2} \dx \,  \psi  \cdot \Im Q_{12}  \big( P_{\leq N} z \,  \overline{P_{\leq N} z} \big) . \label{failure:eq-p1}
\end{align}
Since we are currently assuming \eqref{intro:eq-failure-2}, it therefore follows that
\begin{equation}\label{failure:eq-p2}
 \operatorname{Var} 
\left( \int_{\R} \dt \int_{\T^2} \dx \, 
\begin{pmatrix} - \partial_2 \psi \\[-1ex] \partial_1 \psi \end{pmatrix}
\cdot 
\Im \big( P_{\leq N} z \overline{\nabla P_{\leq N} z} \big) \right) \geq c \log(N) - c^{-1}.  
\end{equation}
As a result, it follows that 
\begin{align*}
&\max \left( 
 \operatorname{Var}\left( \int_{\R} \dt \int_{\T^2} \dx \, 
 \partial_1 \psi \,
\Im \big( P_{\leq N} z \overline{\partial_2 P_{\leq N} z} \big) \right),
\operatorname{Var}\left( \int_{\R} \dt \int_{\T^2} \dx \, 
 \partial_2 \psi \,
\Im \big( P_{\leq N} z \overline{\partial_1 P_{\leq N} z} \big) \right)
\right) 
\\ \geq& \,  c \log(N) - c^{-1}. 
\end{align*}
Thus, by choosing $\varphi \in C^\infty_c((0,\infty) \times \T_x^2 \rightarrow [-1,1])$ as a scalar-multiple of $\partial_2 \psi$ or $\partial_1 \psi$, we see that the desired claim \eqref{intro:eq-failure-1} is satisfied for $j=1$ or $j=2$ (or both). Using the symmetry in the first and second coordinate, however, \eqref{intro:eq-failure-1} then has to be satisfied for both $j=1$ and $j=2$ (with possibly different choices of $\varphi$).\\

It now remains to prove the second claim \eqref{intro:eq-failure-2} and we start with elementary simplifications.  We first note that
\begin{equation*}
\overline{Q_{12}\big( P_{\leq N} z, \overline{ P_{\leq N} z} \big)} 
= Q_{12}\big( \overline{ P_{\leq N} z}, P_{\leq N} z \big)
= - Q_{12}\big( P_{\leq N} z, \overline{ P_{\leq N} z} \big). 
\end{equation*}
Thus, $Q_{12}( P_{\leq N} z, \overline{ P_{\leq N} z} )$ is purely imaginary and we can therefore consider \eqref{intro:eq-failure-2} with the imaginary part $\Im Q_{12}( P_{\leq N} z, \overline{ P_{\leq N} z} )$ replaced by $Q_{12}( P_{\leq N} z, \overline{ P_{\leq N} z})$. Furthermore, we can consider $\psi \in C^\infty_c((0,\infty)\times \T_x^2 \rightarrow \C)$ instead of  $\psi \in C^\infty_c((0,\infty)\times \T_x^2 \rightarrow [-1,1])$, since we can eventually replace $\psi$ by a scalar multiple of its real or imaginary part. We now choose
\begin{equation*}
\psi(t,x) = \chi(t) e^{i \langle m,x\rangle},
\end{equation*}
where $m\in \Z^2\backslash \{0\}$ is arbitrary and $\chi=\chi_m\in \Cs^\infty_c((0,\infty)\rightarrow \R)$ remains to be chosen. In order to compute the space-time integral in \eqref{intro:eq-failure-2}, we first compute
\begin{align*}
&\int_{\T_x^2} \dx \, e^{i \langle m ,x \rangle} Q_{12} \big( P_{\leq N} z , \overline{P_{\leq N} z} \big) \\
=& \sum_{k,l\in \Z^2 \backslash \{ 0 \} } \rho_{\leq N}(k) \rho_{\leq N}(l) \big( k_1 l_2 - k_2 l_1\big) \widehat{z}(t,k) \overline{\widehat{z}(t,l)} \int_{\T_x^2} \dx \, e^{i\langle m ,x \rangle} e^{i \langle k-l,x \rangle} \allowdisplaybreaks[3] \\
=& (2\pi)^2 \sum_{\substack{k,l \in \Z^2 \backslash \{ 0 \}  \colon \\ l = k+m}}
 \rho_{\leq N}(k) \rho_{\leq N}(l) \big( k_1 l_2 - k_2 l_1\big) \widehat{z}(t,k) \overline{\widehat{z}(t,l)} \allowdisplaybreaks[3] \\ 
 =& (2\pi)^2 \sum_{ \substack{k\in \Z^2 \colon \\ k \neq 0, \pm m}}
 \rho_{\leq N}(k) \rho_{\leq N}(k+m) \big( k_1 m_2 - k_2 m_1\big) \widehat{z}(t,k) \overline{\widehat{z}(t,k+m)}. 
\end{align*}
In the last line, we used that $k_1 k_2 - k_2 k_1 =0$, i.e., the null structure. Using the   factor \mbox{$(k_1 m_2 - k_2 m_1)$}, we also restricted to $k\neq 0,\pm m$, which will be convenient (for notational reasons) below. It then follows that
\begin{equation}\label{failure:eq-p3}
\begin{aligned}
&\int_{\R} \dt \int_{\T^2} \dx \, \psi \, Q_{12}\big( P_{\leq N} z, \overline{P_{\leq N} z} \big) \\
=& (2\pi)^2 \int_\R \dt \, \chi(t) \sum_{ \substack{k\in \Z^2 \colon \\ k \neq 0, \pm m}} \rho_{\leq N}(k) \rho_{\leq N}(k+m) \big( k_1 m_2 - k_2 m_1\big) \widehat{z}(t,k) \overline{\widehat{z}(t,k+m)}. 
\end{aligned}
\end{equation}
Due to Lemma \ref{failure:lem-covariance}, \eqref{failure:eq-p3} has zero expectation. It therefore follows that 
\begin{align}
&\frac{1}{(2\pi)^4} \operatorname{Var} \left( 
\int_{\R} \dt \int_{\T^2} \dx \, \psi \, Q_{12}\big( P_{\leq N} z, \overline{P_{\leq N} z} \big) 
\right)  \notag \\
=&\, \int_{\R} \dt \int_{\R} \dt^\prime \, \chi(t) \chi(t^\prime) \sum_{ \substack{k,k^\prime\in \Z^2 \colon \\ k,k^\prime \neq 0, \pm m}}  \bigg( 
 \rho_{\leq N}(k) \rho_{\leq N}(k+m)
 \rho_{\leq N}(k^\prime) \rho_{\leq N}(k^\prime+m) \label{failure:eq-p4} \\
 &\times\,\big( k_1 m_2 - k_2 m_1\big) 
  \big( k_1^\prime m_2 - k_2^\prime m_1\big) 
 \E \Big[ \widehat{z}(t,k) \overline{\widehat{z}(t,k+m)} \, 
  \overline{\widehat{z}(t^\prime,k^\prime)} \widehat{z}(t^\prime,k^\prime+m) \Big] \bigg). \notag 
\end{align}
Using Wick's theorem for products of Gaussian random variables, Lemma \ref{failure:lem-covariance}, and $m \neq 0$, it holds that 
\begin{align}
 &\E \Big[ \widehat{z}(t,k) \overline{\widehat{z}(t,k+m)} \, 
  \overline{\widehat{z}(t^\prime,k^\prime)} \widehat{z}(t^\prime,k^\prime+m) \Big] \notag  \\ 
  =& \,   \E \Big[ \widehat{z}(t,k) 
  \overline{\widehat{z}(t^\prime,k^\prime)} \Big]
  \cdot 
  \E \Big[ \overline{\widehat{z}(t,k+m)} \, 
   \widehat{z}(t^\prime,k^\prime+m) \Big] \notag \\ 
   =& 1 \big\{ k=k^\prime \big\} \, \E \Big[ \widehat{z}(t,k) 
  \overline{\widehat{z}(t^\prime,k)} \Big]
  \cdot 
  \overline{\E \Big[ \widehat{z}(t,k+m) \, 
   \overline{\widehat{z}(t^\prime,k+m)}\Big]}. \label{failure:eq-p5}
\end{align}
After inserting \eqref{failure:eq-p5} and using Lemma \ref{failure:lem-covariance}, it follows that
\begin{align}
&\hspace{1ex}\eqref{failure:eq-p4} \notag \\
\geq& \int_{\R} \dt \int_{\R} \dt^\prime \, \chi(t) \chi(t^\prime)
(t\wedge t^\prime)^2 \sum_{ \substack{k\in \Z^2 \colon \\ k \neq 0, \pm m}}  \bigg( 
 \rho_{\leq N}^2(k) \rho_{\leq N}^2(k+m)
  \big( k_1 m_2 - k_2 m_1\big)^2  \label{failure:eq-p6} \\ 
  &\times \,  \frac{\cos\big( (t-t^\prime) |k| \big) \cos\big( (t-t^\prime) |k+m|\big)}{4|k|^2 |k+m|^2} \bigg) \notag \\
  -&\,  C_{m,\chi} \bigg( \sum_{\substack{k\in \Z^2\colon \\ k \neq 0, \pm m}} \frac{(k_1 m_2 - m_2 k_1)^2}{|k|^5} \bigg), \label{failure:eq-p7}
\end{align}
where $C_{\chi,m}$ is a sufficiently large constant depending on $m$ and $\chi$. Since the sum in \eqref{failure:eq-p7} is absolutely convergent, it remains to prove a lower bound for \eqref{failure:eq-p6}. To this end, we recall the trigonometric identity 
\begin{equation}\label{failure:eq-p8}
\begin{aligned}
&\cos\big( (t-t^\prime) |k| \big) \cos\big( (t-t^\prime) |k+m|\big)  \\
=& \frac{1}{2} \Big( \cos\big( (t-t^\prime) (|k+m|-|k|) \big) + \cos\big( (t-t^\prime) (|k+m|+|k|) \big) \Big). 
\end{aligned}
\end{equation}
Due to the favorable sign in the time-frequency, the contribution of
$\cos\big( (t-t^\prime) (|k+m|+|k|) \big)$ to \eqref{failure:eq-p6} can easily be controlled via integration by parts. Thus, it remains to obtain a lower bound on 
\begin{equation}\label{failure:eq-p9} 
\begin{aligned}
&\int_{\R} \dt \int_{\R} \dt^\prime \, \chi(t) \chi(t^\prime) (t\wedge t^\prime)^2 
\sum_{ \substack{k\in \Z^2 \colon \\ k \neq 0, \pm m}}  
\bigg( 
 \rho_{\leq N}^2(k) \rho_{\leq N}^2(k+m)
  \big( k_1 m_2 - k_2 m_1\big)^2  \\ 
  &\times \, \frac{\cos\big( (t-t^\prime) (|k+m|-|k|) \big)}{8|k|^2 |k+m|^2} \bigg). 
\end{aligned}
\end{equation}
We now choose a deterministic,  nonnegative, cut-off function $\chi\in C^\infty_c((0,\infty) \rightarrow [0,1])$, which depends only on $m\in \Z^2\backslash \{0\}$ and satisfies 
\begin{equation*}
\chi(t) = 
\begin{cases} 
\begin{tabular}{ll}
$0$ \hspace{7ex}  & if $t\leq 2^{-8} |m|^{-1}$, \\ 
$1$  & if $2^{-7} |m|^{-1} \leq t\leq 2^{-6} |m|^{-1}$, \\ 
$0$  & if $t\geq 2^{-5} |m|^{-1}$. 
\end{tabular}
\end{cases}
\end{equation*}
Since $\big| |k+m| - |k| \big| \leq |m|$ for all $k\in \Z^2$, it follows that
\begin{equation}\label{failure:eq-p10}
\chi(t) \chi(t^\prime) \cos\big( (t-t^\prime) (|k+m|-|k|) \big) 
\geq \frac{1}{2} \chi(t) \chi(t^\prime). 
\end{equation}
After inserting this into \eqref{failure:eq-p9} and integrating in $t$ and $t^\prime$, it follows that
\begin{equation*}
\eqref{failure:eq-p9} \geq c_{\chi,m} 
\sum_{ \substack{k\in \Z^2 \colon \\ k \neq 0, \pm m}}
 \bigg( 
 \rho_{\leq N}^2(k) \rho_{\leq N}^2(k+m)
  \frac{\big( k_1 m_2 - k_2 m_1\big)^2}{|k|^2 |k+m|^2} \bigg). 
\end{equation*}
This sum diverges logarithmically in $N$, which finally yields the desired claim. 
\end{proof}

\bibliography{Covariant_Library}
\bibliographystyle{myalpha}

\end{document}